\documentclass[11pt,reqno]{amsart}
\usepackage{a4wide}
\numberwithin{equation}{section}
\usepackage{mathrsfs}
\usepackage{amsfonts}
\usepackage{amsmath}
\usepackage{stmaryrd}
\usepackage{amssymb}
\usepackage{amsthm}
\usepackage{mathrsfs}
\usepackage{url}
\usepackage{amsfonts}
\usepackage{amscd}
\usepackage{indentfirst}
\usepackage{enumerate}
\usepackage{amsmath,amsfonts,amssymb,amsthm}
\usepackage{amsmath,amssymb,amsthm,amscd}
\usepackage{graphicx,mathrsfs}
\usepackage{appendix}
\newcommand{\R}{\mathbb{R}}

\renewcommand{\theequation}{\arabic{section}.\arabic{equation}}

\newtheorem{Thm}{Theorem}[section]
\newtheorem{Lem}[Thm]{Lemma}
\newtheorem{Cor}[Thm]{Corollary}
\newtheorem{Prop}[Thm]{Proposition}

\newtheorem{Rem}[Thm]{Remark}

\begin{document}
\title[Schr\"odinger-Newton problem]
{\small{Uniqueness of positive solutions with Concentration for the Schr\"odinger-Newton problem}}

\author[Peng Luo]{Peng Luo}

\address[Peng Luo]{School of Mathematics and Statistics and Hubei Key Laboratory of Mathematical Sciences, Central China Normal University,
              Wuhan 430079, China}
\email{pluo@mail.ccnu.edu.cn}

\author[Shuangjie Peng]{Shuangjie Peng}

\address[Shuangjie Peng]{School of Mathematics and Statistics and Hubei Key Laboratory of Mathematical Sciences, Central China Normal University, Wuhan 430079, China}
\email{sjpeng@mail.ccnu.edu.cn}

\author[Chunhua Wang]{Chunhua Wang}

\address[Chunhua Wang]{School of Mathematics and Statistics and Hubei Key Laboratory of Mathematical Sciences, Central China Normal University, Wuhan 430079, China}
\email{chunhuawang@mail.ccnu.edu.cn}


\date{\today}


\begin{abstract}
We are concerned with the following  Schr\"odinger-Newton problem
\begin{flalign*}
 -\varepsilon^2\Delta u+V(x)u=\frac{1}{8\pi \varepsilon^2}
 \big(\int_{\R^3}\frac{u^2(\xi)}{|x-\xi|}d\xi\big)u,~x\in \R^3.
 \end{flalign*}
For $\varepsilon$ small enough,  we show the uniqueness of positive solutions concentrating  at the nondegenerate critical points of $V(x)$. The main tools are a local Pohozaev type of  identity, blow-up analysis and the maximum principle.
Our results also show that the asymptotic behavior of  concentrated points to Schr\"odinger-Newton problem is quite different from those of Schr\"odinger equations.
\end{abstract}

\maketitle

\section{Introduction and main results}\label{s1}
\setcounter{equation}{0}
 The Schr\"odinger-Newton problem  appeared in \cite{Penrose} and can be used to describe the quantum mechanics of a polaron at rest.
 It was also used by Choquard to describe an electron trapped in its own hole in a certain approximating to
 Hartree-Fock theory of one component plasma in \cite{Lieb}. Penrose in \cite{Penrose} also derived it as a model of self-gravitating matter, in which quantum state reduction is understood as a gravitational phenomenon. Specifically, if $m$ is the mass of the point, the interaction leads to the system in $\R^3$
 \begin{equation}\label{ll}
 \begin{cases}
 \frac{\varepsilon^2}{2m}\Delta u-V(x)u+\psi u=0,\,&x\in \R^{3},\\
 \Delta \psi +4\pi \tau |u|^2=0,\,&x\in \R^{3},
 \end{cases}
 \end{equation}
 where $u$ is the wave function, $\psi$ is the gravitational potential energy, $V(x)$ is a given Schr\"odinger potential, $\varepsilon$ is the Planck constant, $\tau=Gm^2$ and $G$ is  the Newton's constant of gravitation.

 Let
 \begin{equation*}
 u(x)=\frac{\hat u}{4\varepsilon \sqrt{\pi \tau m}}, ~~V(x)=\frac{1}{2m}\hat V(x),~~\psi(x)=\frac{1}{2m}\hat \psi(x).
 \end{equation*}
 Then system \eqref{ll} can be written, maintaining the original notations, as
  \begin{equation}\label{ll1}
 \begin{cases}
 {\varepsilon^2}\Delta u-V(x)u+\psi u=0,\,&x\in \R^{3},\\
 \varepsilon^2\Delta \psi +\frac{|u|^2}{2}=0,\,& x\in \R^{3}.
 \end{cases}
 \end{equation}
 The second equation in \eqref{ll1} can be explicitly solved with respect to $\psi$, so that
 the system turns into the following  single nonlocal equation
 \begin{flalign}\label{1.2}
 -\varepsilon^2\Delta u+V(x)u=\frac{1}{8\pi \varepsilon^2}
 \big(\int_{\R^3}\frac{u^2(\xi)}{|x-\xi|}d\xi\big)u,~x\in \R^3.
 \end{flalign}
 Also, \eqref{1.2} appears in the study of standing waves for the following nonlinear Hartree equations
\begin{flalign*}
i\varepsilon\frac{\partial \varphi}{\partial t}=-\varepsilon^2\Delta_x\varphi +(V(x)+E)\varphi-\frac{1}{8\pi \varepsilon^2}
 \big(\int_{\R^3}\frac{\varphi^2(\xi)}{|x-\xi|}d\xi\big)\varphi,~(x,t)\in \R^3\times\R^+,
\end{flalign*}
with the form $\varphi(x,t)=e^{-iEt/\varepsilon}u(x)$, where $i$ is the imaginary unit and $\varepsilon$ is the Planck constant.

In recent decades,  problem \eqref{1.2} has been extensively investigated.
When $\varepsilon=1$ and $V(x)=1$, \eqref{1.2} changes into
 \begin{flalign}\label{l1.2}
 - \Delta u+ u=\frac{1}{8\pi \varepsilon^2}
 \big(\int_{\R^3}\frac{u^2(\xi)}{|x-\xi|}d\xi\big)u,~x\in \R^3.
 \end{flalign}
The existence and uniqueness of ground states for \eqref{l1.2} was obtained with variational methods by Lieb \cite{Lieb}, Lions \cite{Lions} and Menzala \cite{Menzala}.
Later, the nondegeneracy of the ground states for \eqref{l1.2} was proved by  Tod-Moroz \cite{Tod} and Wei-Winter \cite{Wei}, which can be stated as follows:

\smallskip

\noindent\textbf{Theorem A.} (c.f \cite{Lieb,Wei}) \emph{
There exists a unique radial solution $U_a$ of the problem
\begin{flalign*}
\begin{cases}
-\Delta u+V(a)u=\displaystyle\frac{1}{8\pi}
 \big(\displaystyle\int_{\R^3}\frac{u^2(\xi)}{|x-\xi|}d\xi\big)u,~ \mbox{in} ~\R^3,\\
u(x)>0,~~ \mbox{in} ~\R^3, ~~~~
u(0)=\displaystyle\max_{x\in \R^3}u(x).
\end{cases}
\end{flalign*}
The solution $U_a$ is strictly decreasing and
\begin{flalign*}
\lim_{|x|\rightarrow \infty}U_a(x)e^{|x|}|x|=\lambda_0>0,~
\lim_{|x|\rightarrow \infty}\frac{U'_a(x)}{U_a(x)}=-1,
\end{flalign*}
for some constant $\lambda_0>0$. Moreover, if $\phi(x)\in H^1(\R^3)$ solves the linearized equation
\begin{flalign*}
-\Delta \phi(x) +V(a)\phi(x)=\displaystyle\frac{1}{8\pi}
 \big(\int_{\R^3}\frac{U_a^2(\xi)}{|x-\xi|}d\xi\big)\phi(x)+
 \frac{1}{4\pi}
 \big(\int_{\R^3}\frac{U_a(\xi)\phi(\xi)}{|x-\xi|}d\xi\big)U_a(x),
\end{flalign*}
then $\phi(x)$ is a linear combination of $\partial U_a/\partial x_j$, $j=1,2,3$.
}

If $\varepsilon$ is small and  $V(x)$ is not a constant, the existence of solutions with ground states for \eqref{1.2}  under some conditions on $V(x)$  was proved  by \cite{Lions1}
 since problem \eqref{1.2} has a variational structure.
 Moreover,  the solution with  ground states concentrates at certain  point.
Later, Wei-Winter \cite{Wei} proved that \eqref{1.2} has a solution concentrating at $k$ points which are the local minimum points of $V(x)$. This also means the existence of multiple solutions.
Concerning  the existence of solutions with  concentration  in other cases, we can refer to \cite{Cingolani, Secchi}  and the references therein.

On the other hand, the Schr\"odinger-Newton problem \eqref{1.2} is a special type of following Choquard equation:
 \begin{flalign}\label{ll2}
 -\varepsilon^2\Delta u+V(x)u=\frac{1}{8\pi \varepsilon^2}
 \big(\int_{\R^3}\frac{u^p(\xi)}{|x-\xi|^{N-\alpha}}d\xi\big)|u|^{p-2}u,~x\in \R^N,
 \end{flalign}
 where $\alpha\in (0,N)$ and $p>1$.
For the existence and concentration of positive solutions to the Choquard equation \eqref{ll2}, one can refer to \cite{Moroz,Moroz1} and the references therein.

As far as we know, for nonlinear Schr\"{o}dinger equations, the results on the uniqueness of solutions concentrating at some points are few.
To obtain uniqueness of concentrating solutions, the classical moving plane method does not work.
The main tools are the topological degree and local Pohozaev identity which can be found in \cite{Cao,Cao1,Deng,Glangetas}. However, for problem~\eqref{1.2},  whether the solution with concentration  is unique is still open.
In this paper, we intend to solve this type of  problems partially  by using  local Pohozaev type of identity and  blow-up analysis which was recently developed in \cite{Cao1,Deng,GPY}. However, we should point out that, compared with \cite{Cao1,Deng,GPY},  to handle the nonlocal term in \eqref{1.2}, there are many new difficulties, which will  be discussed  in more details later.

\smallskip

We assume  that  $V(x)$ is a bounded $C^1$ function satisfying $\displaystyle\inf_{x\in \R^3}V(x)>0$. Define the following Sobolev space $H_{\varepsilon}$
$$H_{\varepsilon}:=\left\{u(x) \in H^1(\R^3), \int_{\R^3}\big(\varepsilon^2|\nabla u(x) |^2+V(x)u^2(x)\big)dx<\infty\right\},$$ and the corresponding  norm
\begin{flalign*}
\|u\|_{\varepsilon}=\big(u(x),u(x)\big)^{1/2}_{\varepsilon}=\big(\int_{\R^3}(\varepsilon^2|\nabla u(x)|^2+V(x)u^2(x))dx\big)^{1/2}.
\end{flalign*}
\noindent\textbf{Definition A.} (c.f. \cite{Cao})\emph{
We call a family of nonnegative functions $\{u_\varepsilon\}_{\varepsilon>0}$ concentrate at a set of different points $\{a_1,\cdots,a_k\}\subset \R^3$ if there exist $\{x_{i,\varepsilon}\}_{\varepsilon>0}\subset \R^3$, $|x_{i,\varepsilon}-a_i|=o(1)$ for $i=1,\cdots,k$ and $k$ nonnegative functions $U_i\in H^1(\R^3)$ $(1\leq i\leq k)$ satisfying $U_i(x)\not\equiv 0$ and $U_i(0)=\displaystyle\max_{x\in\R^3}U_i(x)$ such that
\begin{flalign}\label{1-6}
\|u_\varepsilon-\sum^k_{i=1}
U_i(\frac{x-x_{i,\varepsilon}}{\varepsilon})\|_{\varepsilon}=o(\varepsilon^{3/2}).
\end{flalign}
}

\noindent\textbf{Remark A.}\emph{
Here the solutions in \textbf{Definition A} are consistent with those obtained by  Secchi \cite{Secchi}
and Wei-Winter \cite{Wei}.}

\smallskip

Our main results are as follows.

\begin{Thm}\label{th1--1}
Let $\{u_\varepsilon^{(1)}(x)\}_{\varepsilon>0}, \{u_\varepsilon^{(2)}(x)\}_{\varepsilon>0}$ be two families of positive solutions of \eqref{1.2} concentrating at a nondegenerate critical point  $a_1 \in \R^3$ of $V(x)$.  Then for $\varepsilon$ small enough, $u_\varepsilon^{(1)}(x)\equiv u_\varepsilon^{(2)}(x)$
 must be of the form
\begin{flalign}\label{aaa5}
U_{a_1}(\frac{x-x_{1,\varepsilon}}{\varepsilon})+w_\varepsilon(x),
\end{flalign}
with $x_{1,\varepsilon},w_{\varepsilon}(x)$ satisfying,
 as $\varepsilon\rightarrow 0$,
\begin{flalign}\label{aaa6}
|x_{1,\varepsilon}-a_1|=o(\varepsilon)
,~\mbox{and}~\|w_{\varepsilon}\|_{\varepsilon}=O(\varepsilon^{7/2}).
\end{flalign}
\end{Thm}

\begin{Thm}\label{th1.1}
Let $\{u_\varepsilon^{(1)}(x)\}_{\varepsilon>0}, \{u_\varepsilon^{(2)}(x)\}_{\varepsilon>0}$ be two families of positive solutions of \eqref{1.2} concentrating at  $k~(k\geq 2)$ different nondegenerate critical  points $\{a_1,\cdots,a_k\}\subset \R^3$ of $V(x)$.  Then for $\varepsilon$ small enough, $u_\varepsilon^{(1)}(x)\equiv u_\varepsilon^{(2)}(x)$
 must be of the form
\begin{flalign}
\sum_{j=1}^{k}U_{a_j}(\frac{x-x_{j,\varepsilon}}{\varepsilon})+w_\varepsilon(x),
\end{flalign}
with $x_{j,\varepsilon},w_{\varepsilon}(x)$ satisfying,
for $j=1,\cdots,k$, as $\varepsilon\rightarrow 0$,
\begin{flalign}\label{lll}
|x_{j,\varepsilon}-a_j|=O(\varepsilon) ~\mbox{and}~
\|w_{\varepsilon}\|_\varepsilon=O(\varepsilon^{7/2}).
\end{flalign}
Furthermore, there exist  $j_0\in \{1,\cdots,k\}$, $C_1>0$ and $C_2>0$  such that
\begin{flalign}\label{2a6}
C_1\varepsilon \leq |x_{j_0,\varepsilon}-a_{j_0}|\leq C_2 \varepsilon.
\end{flalign}
\end{Thm}

\begin{Rem}
 For the existence of positive solutions to \eqref{1.2} concentrating at $k$ different points, one can refer to Wei-Winter's paper \cite{Wei}. Also, we can prove that if the positive solutions to \eqref{1.2} concentrating at $k$ different points, then these points must be the critical points of $V(x)$ by  Pohozaev identity.
 Theorem \ref{th1.1} shows the uniqueness of the solutions obtained by Wei-Winter in \cite{Wei}.

For Schr\"odinger equations, it is proved  in \cite{Cao1} that  $|x_{j,\varepsilon}-a_{j}|=o(\varepsilon)$ for $j=1,\cdots,k$.
For the Schr\"odinger-Newton problem, we  can also prove that $|x_{1,\varepsilon}-a_1|=o(\varepsilon)$  in the single peak case. However, in the multi-bump case,
we can only deduce  that the order of  $|x_{j_0,\varepsilon}-a_{j_0}|$ is the same as $\varepsilon$ for some $j_0\in \{1,\cdots,k\}$. This means  that the asymptotic behavior of concentrated points to Schr\"odinger-Newton problem is quite different from those of Schr\"odinger equations.
\end{Rem}

\begin{Rem}
Recently, Xiang \cite{Xiang} proved the uniqueness and nondegeneracy of ground states to the above Choquard equation \eqref{ll2} with $V(x)=a>0$ when $p\rightarrow 2$.
However for general $p$, the uniqueness and nondegeneracy of ground states in Wei-Winter \cite{Wei}  is still an open problem. Correspondingly, our results can not generalize to the Choquard equation for general $p$. However, our methods to handle the nonlocal term is useful to study the Schr\"odinger-Possion problem \cite{Li}.
\end{Rem}

Our main idea is inspired by Cao-Li-Luo \cite{Cao1}, Deng-Lin-Yan \cite{Deng} and Guo-Peng-Yan \cite{GPY}.
Let $u^{(1)}_{\varepsilon}(x)$, $u^{(2)}_{\varepsilon}(x)$ be two different positive solutions concentrating at $\{a_1,\cdots, a_k\}$ for $k\geq 1$.
Set
\begin{flalign}\label{3.1}
\eta_{\varepsilon}(x)=\frac{u_{\varepsilon}^{(1)}(x)-u_{\varepsilon}^{(2)}(x)}
{\|u_{\varepsilon}^{(1)}-u_{\varepsilon}^{(2)}\|_{L^{\infty}(\R^3)}}.
\end{flalign}
Then we prove $
\eta_{\varepsilon}(x)=o(1)$ for $x\in \R^3$,
which is incompatible with the fact $\|\eta_{\varepsilon}\|_{L^{\infty}(\R^3)}=1$.
For the estimate near the nondegenerate critical  points, we will use the blow-up analysis and local Pohozaev type of  identity. But for the estimate away from the nondegenerate critical  points, we will use the maximum principle.
\begin{Rem}
 Problem  \eqref{1.2} is nonlocal, which will cause  many differences compared with   \cite{Cao1,Deng}.

First, to apply the blow-up analysis, we need to prove that \begin{equation}\label{ln1}
|x^{(1)}_{j,\varepsilon}-x^{(2)}_{j,\varepsilon}|=o(\varepsilon),
\end{equation}
 which cannot be obtained by
$|x^{(1)}_{j,\varepsilon}-a_j|=O(\varepsilon)$ and $|x^{(2)}_{j,\varepsilon}-a_j|=O(\varepsilon)$ for $j=1,2,\cdots,k$ ($k\geq 2$).
To obtain \eqref{ln1}, we will apply local Pohozaev identity carefully in Proposition \ref{prop3--1} below.

Next, after using the blow-up analysis, we will apply local Pohozaev identity again
to obtain $\eta_{\varepsilon}(x)=o(1)$ near the nondegenerate critical points. To this aim,
we need to estimate the error between the two solutions precisely in Proposition \ref{lp1} below, where the classical Nash-Moser iteration will be used.

On the other hand,  we would like to point out that the corresponding local Pohozaev identity will have two terms involving
volume integral. Then to calculate the two integrals precisely, we need to use some symmetries skillfully by some observations. We will also use the maximum principle carefully due to the nonlocal term.
\end{Rem}

\smallskip

This paper will be organized as follows. In Section \ref{s2} and Section 3, we first establish some basic estimates of the solutions with concentration.
Then we give the detailed proofs of  Theorem \ref{th1--1} by using a local Pohozaev identity, blow-up analysis and the maximum principle.
In Section 4, we obtain a precise estimate on the  errors of the two solutions and the concentrated points. Combining these estimates and applying the methods in the proof of Theorem \ref{th1--1}, we prove
Theorem \ref{th1.1} in Section 5. To make the main clue clear, some important but tedious estimates and inequalities will be delayed  to the Appendix. In the sequel, we will use $C$  to denote various generic positive constants.
 $O(t)$, $o(t)$ mean $|O(t)|\leq C|t|$, $o(t)/t\rightarrow 0$ as $t\rightarrow 0$. $o(1)$ denotes quantities that tend to $0$ as $\varepsilon\rightarrow 0$.

\section{The basic estimates}\label{s2}
\setcounter{equation}{0}

\begin{Prop}\label{pr2.1}
Let $\{u_\varepsilon (x)\}_{\varepsilon>0}$ be a family of positive solutions of \eqref{1.2}
 concentrating at  different points $a_1,\cdots,a_k$ with $k\geq 1$.  Then
$u_{\varepsilon}(x)$ is of the form
\begin{flalign}\label{2-1}
u_\varepsilon(x)=\sum_{j=1}^{k}U_{a_j}(\frac{x-x_{j,\varepsilon}}{\varepsilon})+w_\varepsilon(x),
\end{flalign}
with $x_{j,\varepsilon}$ and $w_{\varepsilon}(x)$  satisfying,
for $j=1,\cdots,k$, as $\varepsilon\rightarrow 0$,
\begin{flalign}\label{2--2}
|x_{j,\varepsilon}-a_j|=o(1) ~\mbox{and}~
\|w_{\varepsilon}\|_\varepsilon=o(\varepsilon^{3/2}),
\end{flalign}
and
\begin{flalign}\label{2--3}
\big(w_{\varepsilon}(x),U_{a_j}(\frac{x-x_{j,\varepsilon}}{\varepsilon})\big)_{\varepsilon}=o(\varepsilon^3),
~\big(w_{\varepsilon}(x),\frac{U_{a_j}(\frac{x-x_{j,\varepsilon}}{\varepsilon})}{\partial{x^i}}\big)_{\varepsilon}=0, ~i=1,2,3.
\end{flalign}
\end{Prop}

\begin{proof}
Let
$u_{i,\varepsilon}(x)=u_{\varepsilon}(\varepsilon x+x_{i,\varepsilon})$. Then
\begin{flalign}\label{2-2}
-\Delta u_{i,\varepsilon}(x)+V(\varepsilon x+x_{i,\varepsilon})u_{i,\varepsilon}(x)=
\frac{1}{8\pi }
 \big(\int_{\R^3}\frac{u_{i,\varepsilon}^2(\xi)}{|x-\xi|}d\xi\big)
 u_{i,\varepsilon}(x),~x \in \R^{3}.
\end{flalign}
Suppose that $\varphi(x)$ is an arbitrarily fixed function in $H^1(\R^3)$. By \eqref{2-2} we get
\begin{flalign}\label{2-3}
\int_{\R^3}\nabla u_{i,\varepsilon}(x)\nabla \varphi(x) +V(\varepsilon x+x_{i,\varepsilon})u_{i,\varepsilon}(x)\varphi(x)dx=
\frac{1}{8\pi }
 \int_{\R^3}\int_{\R^3}\frac{u_{i,\varepsilon}^2(\xi)}{|x-\xi|}u_{i,\varepsilon}(x)\varphi(x)d\xi dx.
\end{flalign}
By \eqref{1-6} and passing to a subsequence if necessary $u_{i,\varepsilon}(x)\rightarrow U_i(x)$ weakly in $H^1(\R^3)$ as $\varepsilon\rightarrow 0$. Then
taking $\varepsilon \rightarrow 0$ in \eqref{2-3} we get
\begin{flalign*}
\int_{\R^3}\nabla U_i(x)\nabla \varphi(x) +V(a_i)U_i(x)\varphi(x)dx=
\frac{1}{8\pi}
 \int_{\R^3}\int_{\R^3}\frac{U^2_i(\xi)}{|x-\xi|}U_i(x)\varphi(x)d\xi dx.
\end{flalign*}
Therefore $U_i(x)$ is a nonnegative weak solution of
\begin{flalign*}
\begin{cases}
-\Delta u+V(a_i)u=\displaystyle\frac{1}{8\pi}
 \big(\int_{\R^3}\frac{u^2(\xi)}{|x-\xi|}d\xi\big)u,~ \mbox{in} ~\R^3,\\
u(x)>0,~~ \mbox{in} ~\R^3,~~~~
u(0)=\displaystyle\max_{x\in \R^3}u(x).
\end{cases}
\end{flalign*}
By regularity theory and maximum principle, we know $U_i(x)>0$ in $\R^3$. Also from the uniqueness result in \cite{Wei}, we can show that $U_i(x)\equiv U_{a_i}(x)$ and $U_i(x)$ decays exponentially at infinity.
Then following the decomposition lemma in Lemma \ref{lem2.5}, we can write $u_{\varepsilon}(x)$ uniquely as
\eqref{2-1}
with $x_{j,\varepsilon}$ and $w_{\varepsilon}(x)$  satisfying \eqref{2--2} and \eqref{2--3}
for $j=1,\cdots,k$, as $\varepsilon\rightarrow 0$.
\end{proof}

\begin{Prop}\label{prop1-1}
Suppose that $u_\varepsilon(x)$ is a positive solution of \eqref{1.2}
 concentrating at  different points $a_1,\cdots,a_k$ with $k\geq 1$. Then
for any fixed $R\gg 1$, there exist $\theta>0$ and $C>0$, such that
\begin{flalign}\label{2--5}
u_\varepsilon(x)\leq Ce^{-\theta|x-x_{l,\varepsilon}|/\varepsilon},~\mbox{for}~
l=1,\cdots,k~\mbox{and}~x\in \R^3\backslash \bigcup^k_{j=1} B_{R \varepsilon}(x_{j,\varepsilon}).
\end{flalign}
\end{Prop}

\begin{proof}
If $u_\varepsilon(x)$ is a positive solution of \eqref{1.2}, then we have
\begin{flalign}\label{1-1}
-\varepsilon^2\Delta u_\varepsilon(x)+\big(V(x)-\frac{1}{8\pi \varepsilon^2}
 \big(\int_{\R^3}\frac{u_\varepsilon^2(\xi)}{|x-\xi|}d\xi\big)\big)u_\varepsilon(x)=0,~x\in \R^3.
\end{flalign}
By \eqref{l1} in the Appendix, we know that, for large fixed $R$  and $\varepsilon$ small enough,
\begin{flalign}\label{b1}
V(x)-\frac{1}{8\pi \varepsilon^2}
 \big(\int_{\R^3}\frac{u_\varepsilon^2(\xi)}{|x-\xi|}d\xi\big)\geq m/2, ~\mbox{in}~ \R^3\backslash\bigcup_{j=1}^k B_{R\varepsilon}(x_{j,\varepsilon}),
\end{flalign}
where $m=\displaystyle\inf_{x\in \R^3}V(x)$.
Then \eqref{1-1} and \eqref{b1} imply
\begin{flalign*}
-\varepsilon^2\Delta u_{\varepsilon}+\frac{m}{2}u_{\varepsilon}\leq 0, ~\mbox{in}~ \R^3\backslash\bigcup_{j=1}^k B_{R\varepsilon}(x_{j,\varepsilon}).
\end{flalign*}
Define the operator $L_\varepsilon$ as follows:
$$
L_\varepsilon v:=
-\varepsilon^2\Delta v+\frac{m}{2}v,~\mbox{for all}~v\in H^1(\R^3).
$$
Then for
 $v_l(x)=e^{-\theta |x-x_{l,\varepsilon}|/\varepsilon}$, where $0<\theta <\sqrt{m/2}$ and $l\in \{1,2,\cdots,k\}$, we have
\begin{flalign*}
\begin{split}
L_\varepsilon v_l(x)&=-\varepsilon^2\big[\frac{\theta^2}{\varepsilon^2}
-\frac{2\theta}{|x-x_{l,\varepsilon}|\varepsilon}\big]e^{-\theta |x-x_{l,\varepsilon}|/\varepsilon}+\frac{m}{2}e^{-\theta |x-x_{l,\varepsilon}|/\varepsilon}\\&=
\big[\frac{2\varepsilon\theta}{|x-x_{l,\varepsilon}|}+\frac{m}{2}-\theta^2\big]e^{-\theta |x-x_{l,\varepsilon}|/\varepsilon}\geq 0.
\end{split}
\end{flalign*}

Next, we extend $u_\varepsilon(x)$ to $\R^3$ by $0$ (still denoted as $u_\varepsilon(x)$) and let $\bar v_l(x)=cv_l(x)-u_\varepsilon(x)$, where $c>0$, then
\begin{flalign*}
L_\varepsilon \bar v_l=cL_\varepsilon v_l-
L_\varepsilon u_\varepsilon\geq 0, ~\mbox{in}~\R^3\backslash\bigcup_{j=1}^k B_{R\varepsilon}(x_{j,\varepsilon}).
\end{flalign*}
Also since $u_\varepsilon(x)\in C(\R^3)$, then there exists $C_0>0$ such that
$|u_\varepsilon(x)|\leq C_0, ~\mbox{on}~\partial \big(\bigcup_{j=1}^k B_{R\varepsilon}(x_{j,\varepsilon})\big).$
So taking $c=C_0e^{R\theta}$, we have
\begin{flalign*}
\begin{split}
\bar v_l(x)&=cv_l(x)-u_\varepsilon(x)
\geq ce^{-R\theta}-C_0\geq 0, ~\mbox{on}~\partial \big(\bigcup_{j=1}^k B_{R\varepsilon}(x_{j,\varepsilon})\big).
\end{split}
\end{flalign*}
Thus for the above fixed large $R$, we obtain
\begin{flalign*}
  \begin{cases}
L_\varepsilon \bar v_l(x)\geq 0,
    &~\mbox{in}~\R^3\backslash\bigcup_{j=1}^k B_{R\varepsilon}(x_{j,\varepsilon}),\\
\bar v_l(x)\geq 0, & ~\mbox{on}~\partial \big(\bigcup_{j=1}^k B_{R\varepsilon}(x_{j,\varepsilon})\big),\\
 \bar v_l(x)\rightarrow 0,&~\mbox{as}~|x|\rightarrow +\infty.
  \end{cases}
\end{flalign*}
Then by the maximum principle, we have
$\bar v_l(x)\geq 0$, for $x\in\R^3 \setminus \displaystyle\bigcup_{j=1}^k B_{R\varepsilon}(x_{j,\varepsilon})$.
This means \eqref{2--5}.
\end{proof}

\begin{Cor}
Suppose that $u_\varepsilon(x)$ is a solution of \eqref{1.2} as in Proposition \ref{prop1-1}.

\textup{(1)}. Then
for any fixed $R\gg 1$, there exists $\theta_1>0$  such that
\begin{equation}\label{l3}
u_\varepsilon(x)=O\big(e^{-\theta_1 R}\big),~\mbox{for}~x\in\R^3\backslash \bigcup_{j=1}^k B_{R\varepsilon }(x_{j,\varepsilon}).
\end{equation}

\textup{(2)}. Then
for any fixed $d>0$, there exists $\theta_2>0$  such that
\begin{equation}\label{l4}
u_\varepsilon(x)=O\big(e^{-\theta_2 /\varepsilon}\big),~\mbox{for}~x\in \R^3\backslash  \bigcup_{j=1}^k B_{d }(x_{j,\varepsilon}).
\end{equation}
\end{Cor}
\begin{proof}
These are the direct results by Proposition \ref{prop1-1}.
\end{proof}
By the regularity theory of elliptic equations in \cite{Gilbarg}, $u_{\varepsilon}(x)$ above is in fact  a classical solution.
\begin{Prop}
Let $u(x)$ be the solution of \eqref{1.2}, then  we have following \textbf{local Pohozaev identity}:
\begin{flalign}\label{2.5}
\begin{split}
\int_{\Omega}\frac{\partial V(x)}{\partial x^i}u^2(x)dx=&
\int_{\partial \Omega}\big(-\varepsilon^2|\nabla u(x)|^2 +V(x)u^2(x)-\frac{1}{8\pi \varepsilon^2}\int_{\R^3}\frac{u^2(\xi)u^2(x)}{|x-\xi|}d\xi\big)\nu_i(x)d
\sigma
\\&+
\frac{1}{8\pi \varepsilon^2}\int_{\Omega}\int_{\R^3}u^2(\xi)u^2(x)\frac{x^i-\xi^i}{|x-\xi|^3}d\xi dx,~\mbox{for}~i=1,2,3,
\end{split}
\end{flalign}
where $\Omega$ is a bounded open domain of $\R^3$, $\nu(x)=\big(\nu_{1}(x),\nu_2(x),\nu_3(x)\big)$ is the outward unit normal of $\partial \Omega$ and $x^i, \xi^i$ are the $i$-th components of $x, \xi$.
 \end{Prop}\begin{proof}
\eqref{2.5} is obtained by multiplying $\frac{\partial u(x)}{\partial x^i}$ on both sides of \eqref{1.2} and integrating on $\Omega$. We omit the details.
 \end{proof}

\section{Proof of Theorem \ref{th1--1}}

\begin{Prop}\label{Prop2.7}
Let $u_{\varepsilon}(x)$ be the solution of \eqref{1.2}  concentrating at a nondegenerate critical point  $a_1 \in \R^3$ of $V(x)$.
 Then it holds
\begin{flalign}\label{l2--12}
|x_{1,\varepsilon}-a_{1}|=o(\varepsilon).
\end{flalign}
\end{Prop}

\begin{proof}
First, for the small fixed constant $\bar d>0$, taking $u(x)=u_\varepsilon(x)$ and  $\Omega=B_{d}(x_{1,\varepsilon})$  in the Pohozaev identity \eqref{2.5} with any $d\in (\bar d, 2\bar d)$,  we have, for $i=1,2,3$,
\begin{flalign}\label{ll2.11}
\begin{aligned}
\int_{B_{d}(x_{1,\varepsilon})}\frac{\partial V(x)}{\partial x^i}u_\varepsilon^2(x)dx
=&\int_{\partial B_{d}(x_{1,\varepsilon})}A(x)d\sigma
+\frac{1}{8\pi \varepsilon^2}\int_{ B_{d}(x_{1,\varepsilon})}\int_{\R^3}u^2_{\varepsilon}(x)u^2_{\varepsilon}(\xi)\frac{x^i-\xi^i}{|x-\xi|^3}d\xi dx,
\end{aligned}
\end{flalign}
where
\begin{equation}\label{l3.3}
A(x)=\big(-\varepsilon^2 |\nabla u_{\varepsilon}(x)|^2+ u_\varepsilon^2(x)(V(x)-\frac{1}{8\pi \varepsilon^2}\int_{\R^3}\frac{ u^2_{\varepsilon}(\xi) }{|x-\xi|}d\xi )\big)\nu_i(x).
\end{equation}
Next for any $d\in (\bar d, 2\bar d)$, \eqref{1-6} and \eqref{laa1} imply
\begin{flalign}\label{l2-11}
\begin{aligned}
\mbox{LHS of}~\eqref{ll2.11}&=\int_{B_{d}(x_{1,\varepsilon})}\frac{\partial V(x)}{\partial x^i} U^2_{a_1}(\frac{x-x_{1,\varepsilon}}{\varepsilon})dx+
o\big(\varepsilon^4+\varepsilon^3|x_{1,\varepsilon}-a_1|\big)
\\&= \varepsilon^{3} \big(\int_{\R^3} U_{a_1}^2(x)dx\big)\sum^3_{l=1}\frac{\partial^2 V(a_{1})}{\partial x^i\partial x^l}
(x^l_{1,\varepsilon}-a^l_{1}) +o\big(\varepsilon^4+\varepsilon^3|x_{1,\varepsilon}-a_1|\big),
\end{aligned}
\end{flalign}
where $x^l_{1,\varepsilon}, a^l_{1}$ are the $l$-th components of $x_{1,\varepsilon}, a_{1}$.

On the other hand, using \eqref{l4},  \eqref{A.14} and \eqref{2.45}, there exists $d_\varepsilon\in (\bar d, 2\bar d)$ such that
\begin{flalign}\label{l5}
\int_{\partial B_{d_\varepsilon}(x_{1,\varepsilon})}A(x)d\sigma=O\big(e^{-\eta/\varepsilon}+\|w_\varepsilon\|^2_\varepsilon\big)
=O\big(\varepsilon^{7}\big).
\end{flalign}
Also for any $d\in (\bar d, 2\bar d)$, by symmetry and \eqref{l4}, we deduce
\begin{flalign}\label{l6}
\begin{split}
\frac{1}{8\pi \varepsilon^2}&\int_{ B_{d}(x_{1,\varepsilon})}\int_{\R^3}u^2_{\varepsilon}(x)u^2_{\varepsilon}(\xi)\frac{x^i-\xi^i}{|x-\xi|^3}d\xi dx\\
=&
\frac{1}{8\pi \varepsilon^2}\int_{ \R^3}\int_{\R^3}u^2_{\varepsilon}(x)u^2_{\varepsilon}(\xi)\frac{x^i-\xi^i}{|x-\xi|^3}d\xi dx+
O\big(e^{-\eta/\varepsilon}\big)=O\big(e^{-\eta/\varepsilon}\big).
\end{split}\end{flalign}
Let $d=d_\varepsilon$ in \eqref{ll2.11}, then \eqref{l2-11}, \eqref{l5} and  \eqref{l6} imply
\begin{flalign*}
\begin{aligned}
\sum^3_{l=1}\frac{\partial^2 V(a_{1})}{\partial x^i\partial x^l}
(x^l_{1,\varepsilon}-a^l_{1})=o\big(|x_{1,\varepsilon}-a_1|\big)+o(\varepsilon).
\end{aligned}
\end{flalign*}
This means that \eqref{l2--12} holds.
\end{proof}

\begin{Prop}\label{3-2}
Let $\eta_{1,\varepsilon}(x)=\eta_{\varepsilon}(\varepsilon x+x_{1,\varepsilon}^{(1)})$, then taking
a subsequence necessarily, it holds
\begin{equation}\label{lp11}
\eta_{1,\varepsilon}(x)\rightarrow\sum_{i=1}^3 b_{1,i}\frac{\partial U_{a_1}(x)}{\partial x^i}
\end{equation}
uniformly in $C^1(B_R(0))$ for any $R>0$, where $\eta_{\varepsilon}(x)$ is the function in \eqref{3.1} and $b_{1,i}$, $i=1,2,3$ are some constants.
\end{Prop}

\begin{proof}
Since $\|\eta_{1,\varepsilon}\|_{L^{\infty}(\R^3)}\leq 1$, by the regularity theory, we know
\begin{flalign*}
\eta_{1,\varepsilon}(x)\in C^{1,\alpha}_{loc}(\R^3)~~\mbox{and}~~\|\eta_{1,\varepsilon}\|_{C^{1,\alpha}_{loc}(\R^3)}\leq C,
~\mbox{for some}~\alpha \in (0,1).
\end{flalign*}
So we may assume that $$\eta_{1,\varepsilon}(x)\rightarrow\eta_{1}(x),~\mbox{in}~C_{loc}(\R^3).$$
 By direct calculations, we have

 \begin{flalign*}
-\Delta\eta_{1,\varepsilon}(x)= -V(\varepsilon x+x_{1,\varepsilon}^{(1)})\eta_{1,\varepsilon}(x)+
E_1(\varepsilon x+x_{1,\varepsilon}^{(1)})\eta_{1,\varepsilon}(x)+E_2(\varepsilon x+x_{1,\varepsilon}^{(1)}),
\end{flalign*}
where $E_1(\varepsilon x+x_{1,\varepsilon}^{(1)})$ and $E_2(\varepsilon x+x_{1,\varepsilon}^{(1)})$ are  in \eqref{e}.
Then from \eqref{l9} and \eqref{l10}, we get
\begin{flalign*}
E_1(\varepsilon x+x_{1,\varepsilon}^{(1)})=
\frac{1}{8\pi}
 \big(\int_{\R^3}\frac{U_{a_1}^2(\xi)}{|x-\xi|}d\xi\big)+o(1), ~x\in B_{d/{\varepsilon}}(0),
\end{flalign*}
and
\begin{flalign}\label{la}
E_2(\varepsilon x+x_{1,\varepsilon}^{(1)})=\frac{U_{a_1}(x)}{4\pi}
 \big(\int_{\R^3}\frac{U_{a_1}(\xi)\eta_{1,\varepsilon}(\xi)}{|x-\xi|}d\xi\big)+o(1), ~x\in B_{d/{\varepsilon}}(0).
\end{flalign}
Next, for any given $\Phi(x)\in C_{0}^{\infty}(\R^3)$, we have
\begin{flalign}\label{3.11}
\begin{aligned}
\displaystyle \int_{\R^3}\big(-\Delta\eta_{1,\varepsilon}(x)&
+V(\varepsilon x+x_{1,\varepsilon}^{(1)})\eta_{1,\varepsilon}(x) \big)\Phi(x)dx-\frac{1}{8\pi}
\int_{\R^3}\int_{\R^3}\frac{U_{a_1}^2(\xi)}{|x-\xi|}\eta_{1,\varepsilon}(x) \Phi(x)d\xi dx\\
&
-\frac{1}{4\pi}\displaystyle \int_{\R^3}\int_{\R^3}
\frac{U_{a_1}(\xi)\eta_{1,\varepsilon}(\xi)}{|x-\xi|}
U_{a_1}(x)\Phi(x) d\xi dx
=o(1)\|\Phi\|_{H^1(\R^3)}.
\end{aligned}
\end{flalign}
Letting $\varepsilon\rightarrow 0$ in \eqref{3.11} and using the elliptic regularity theory, we find that $\eta_{1}(x)$ satisfies
\begin{flalign*}
-\Delta\eta_{1}(x)+V(a_1)\eta_{1}(x)=
\frac{1}{8\pi}
 \big(\int_{\R^3}\frac{U_{a_1}^2(\xi)}{|x-\xi|}d\xi\big)\eta_{1}(x)+\frac{1}{4\pi}
 \big(\int_{\R^3}\frac{U_{a_1}(\xi)\eta_{1}(\xi)}{|x-\xi|}d\xi\big)U_{a_1}(x),~~\textrm{in~}\R^3,
\end{flalign*}
which gives $
\eta_{1}(x)=\displaystyle\sum_{i=1}^3b_{1,i}\frac{\partial U_{a_1}(x)}{\partial x^i}.
$ This means \eqref{lp11}.
\end{proof}

\begin{Prop}\label{3-3}
Let $b_{1,i}$ be as in Proposition \ref{3-2}, then
we have
$$b_{1,i}=0,~~\mbox{for all}~i=1,2,3.$$
\end{Prop}

\begin{proof}
Since $u_{\varepsilon}^{(1)}(x)$, $u_{\varepsilon}^{(2)}(x)$ are the positive solutions of \eqref{1.2},
for the small fixed constant $\bar d>0$, using Pohozaev identity \eqref{2.5} with any $\delta\in (\bar d, 2\bar d)$, we deduce
\begin{flalign}\label{3.14}
\begin{aligned}
&\displaystyle \int_{B_{\delta}(x_{1,\varepsilon}^{(1)})}\frac{\partial V(x)}{\partial x^i}\big(u_{\varepsilon}^{(1)}(x)+u_{\varepsilon}^{(2)}(x)\big)\cdot\eta_{\varepsilon}(x)dx \\=&
\int_{\partial B_{\delta}(x_{1,\varepsilon}^{(1)})}B(x)d\sigma+\frac{1}{8\pi \varepsilon^2}
\int_{ B_{\delta}(x_{1,\varepsilon}^{(1)})}\int_{\R^3}
\big(u^{(1)}_{\varepsilon}(x)\big)^2
\big(u^{(1)}_{\varepsilon}(\xi)+u^{(2)}_{\varepsilon}(\xi)\big)\eta_{\varepsilon}(\xi)
\frac{x^i-\xi^i}{|x-\xi|^3}d\xi dx\\&
+\frac{1}{8\pi \varepsilon^2}
\int_{B_{\delta}(x_{1,\varepsilon}^{(1)})}\int_{\R^3}
\big(u^{(2)}_{\varepsilon}(\xi)\big)^2
\big(u^{(1)}_{\varepsilon}(x)+u^{(2)}_{\varepsilon}(x)\big)\eta_{\varepsilon}(x)\frac{x^i-\xi^i}{|x-\xi|^3}d\xi dx,
\end{aligned}
\end{flalign}
where
\begin{flalign}\label{lb}
\begin{aligned}
B(x)=&
-
\big(\varepsilon^2 \nabla \big(u_{\varepsilon}^{(1)}(x)
+u_{\varepsilon}^{(2)}(x)\big)\cdot \nabla \eta_{\varepsilon}(x)+V(x)\big(u_{\varepsilon}^{(1)}(x)
+u_{\varepsilon}^{(2)}(x)\big)\eta_{\varepsilon}(x)\big)\nu_i(x)\\
&-\frac{\big(u^{(1)}_{\varepsilon}(x)\big)^2\nu_i(x)}{8\pi \varepsilon^2}\int_{\R^3}
\big(u^{(1)}_{\varepsilon}(\xi)+u^{(2)}_{\varepsilon}(\xi)\big)\eta_{\varepsilon}(\xi){|x-\xi|^{-1}}d\xi \\&
-\frac{\big(u^{(1)}_{\varepsilon}(x)+u^{(2)}_{\varepsilon}(x)\big)\eta_{\varepsilon}(x)\nu_i(x)}{8\pi \varepsilon^2}\int_{\R^3}
\big(u^{(2)}_{\varepsilon}(\xi)\big)^2
{|x-\xi|^{-1}}d\xi.
\end{aligned}
\end{flalign}
Now from \eqref{2--2},  Proposition \ref{Prop2.7}, Proposition \ref{3-2} and \eqref{laa1}, for any $\delta\in (\bar d, 2\bar d)$, we have
\begin{flalign}\label{4-1}
\begin{split}
\textrm{LHS of} ~\eqref{3.14}=&
\int_{B_{\delta}(x_{1,\varepsilon}^{(1)})}\frac{\partial V(x)}{\partial x^i}\big(U_{a_j}(\frac{x-x^{(1)}_{1,\varepsilon}}{\varepsilon})+
U_{a_j}(\frac{x-x^{(2)}_{1,\varepsilon}}{\varepsilon})\big)\eta_{\varepsilon}(x)dx+O\big(\varepsilon^5\big)\\=&
\int_{B_{\delta}(x_{1,\varepsilon}^{(1)})}\big(\sum^3_{l=1}(x^l-
a^l_{1})\frac{\partial^2 V(a_{1})}{\partial x^i\partial x^l} \big)\big(U_{a_j}(\frac{x-x^{(1)}_{1,\varepsilon}}{\varepsilon})+
U_{a_1}(\frac{x-x^{(2)}_{1,\varepsilon}}{\varepsilon})\big)\eta_{\varepsilon}(x)dx\\&
+\int_{B_{\delta}(x_{1,\varepsilon}^{(1)})}o(|x-
a_{1}|)\big(U_{a_j}(\frac{x-x^{(1)}_{1,\varepsilon}}{\varepsilon})+
U_{a_1}(\frac{x-x^{(2)}_{1,\varepsilon}}{\varepsilon})\big)\eta_{\varepsilon}(x)dx+O\big(\varepsilon^5\big)
\\=&
2\varepsilon^4\big(\sum^3_{l=1}\frac{\partial^2 V(a_{1})}{\partial x^i\partial x^l}\int_{\R^3}x^l
U_{a_1}(x) \big(b_{1,i}\frac{\partial U_{a_1}(x)}{\partial x^i}\big)dx\big)+o(\varepsilon^4)\\=&
\varepsilon^4\int_{\R^3} U^2_{a_1}(x)dx\big(\sum^3_{l=1}\frac{\partial^2 V(a_1)}{\partial x^i\partial x^l}b_{1,l}\big)+o(\varepsilon^4),
\end{split}
\end{flalign}
where $x^l, a^l_{1}$ are the $l$-th components of $x, a_{1}$.

On the other hand, similar to \eqref{l5}, there exists $\delta_\varepsilon\in (\bar d, 2\bar d)$, we obtain
\begin{flalign}\label{e17l}
\int_{\partial B_{\delta_\varepsilon}(x_{1,\varepsilon}^{(1)})}B(x)d\sigma=O\big(e^{-\eta/\varepsilon}
+\|w_\varepsilon\|_\varepsilon \|\eta_\varepsilon\|_\varepsilon\big)
=O\big(\varepsilon^{5}\big).
\end{flalign}
Also, by \eqref{2--5},  for any $\delta\in (\bar d, 2\bar d)$, we have
\begin{flalign}\label{e18l}
\begin{split}
\frac{1}{8\pi \varepsilon^2}&
\int_{ B_{\delta}(x_{1,\varepsilon}^{(1)})}\int_{\R^3}
\big(u^{(1)}_{\varepsilon}(x)\big)^2
\big(u^{(1)}_{\varepsilon}(\xi)+u^{(2)}_{\varepsilon}(\xi)\big)\eta_{\varepsilon}(\xi)
\frac{x^i-\xi^i}{|x-\xi|^3}d\xi dx\\=&\frac{1}{8\pi \varepsilon^2}
\int_{\R^3}\int_{\R^3}
\big(u^{(1)}_{\varepsilon}(x)\big)^2
\big(u^{(1)}_{\varepsilon}(\xi)+u^{(2)}_{\varepsilon}(\xi)\big)\eta_{\varepsilon}(\xi)
\frac{x^i-\xi^i}{|x-\xi|^3}d\xi dx+O\big(e^{-\eta/\varepsilon}
\big),
\end{split}\end{flalign}
and
\begin{flalign}\label{e19l}
\begin{split}
\frac{1}{8\pi \varepsilon^2}&
\int_{B_{\delta}(x_{1,\varepsilon}^{(1)})}\int_{\R^3}
\big(u^{(2)}_{\varepsilon}(\xi)\big)^2
\big(u^{(1)}_{\varepsilon}(x)+u^{(2)}_{\varepsilon}(x)\big)\eta_{\varepsilon}(x)\frac{x^i-\xi^i}{|x-\xi|^3}d\xi dx\\=&\frac{1}{8\pi \varepsilon^2}
\int_{\R^3}\int_{\R^3}\big(u^{(2)}_{\varepsilon}(\xi)\big)^2
\big(u^{(1)}_{\varepsilon}(x)+u^{(2)}_{\varepsilon}(x)\big)\eta_{\varepsilon}(x)\frac{x^i-\xi^i}{|x-\xi|^3}d\xi dx+O\big(e^{-\eta/\varepsilon}
\big).
\end{split}\end{flalign}
Let $\delta=\delta_\varepsilon$ in \eqref{3.14}, by symmetry, \eqref{e17l}, \eqref{e18l}  and \eqref{e19l}, we get
\begin{flalign}\label{4-2}
\begin{split}
&\textrm{RHS  of}~ \eqref{3.14}\\=&\frac{1}{8\pi \varepsilon^2}
\int_{\R^3}\int_{\R^3}\big(\big(u^{(1)}_{\varepsilon}(x)\big)^2
-\big(u^{(2)}_{\varepsilon}(x)\big)^2\big)
\big(u^{(1)}_{\varepsilon}(\xi)+u^{(2)}_{\varepsilon}(\xi)\big)\eta_{\varepsilon}(\xi)
\frac{x^i-\xi^i}{|x-\xi|^3}d\xi dx
+O\big(e^{-\eta/\varepsilon}\big)\\=&
\frac{J_{\varepsilon}}{8\pi \varepsilon^2}
\int_{\R^3}\int_{\R^3}\big(u^{(1)}_{\varepsilon}(x)+ u^{(2)}_{\varepsilon}(x)\big) \eta_{\varepsilon}(x)
\big(u^{(1)}_{\varepsilon}(\xi)+u^{(2)}_{\varepsilon}(\xi)\big)
 \eta_{\varepsilon}(\xi)
\frac{x^i-\xi^i}{|x-\xi|^3}d\xi dx
+O\big(e^{-\eta/\varepsilon}\big)\\=&
O\big(e^{-\eta/\varepsilon}
\big),
\end{split}\end{flalign}
where $J_{\varepsilon}:=\|u^{(1)}_{\varepsilon}(\cdot)-u^{(2)}_{\varepsilon}(\cdot)\|_{L^{\infty}(\R^3)}$.
Then \eqref{4-1} and \eqref{4-2} imply
\begin{flalign*}
\sum^3_{l=1}\frac{\partial^2 V(a_1)}{\partial x^i\partial x^l}b_{1,l}=o(1).
\end{flalign*}
This means $b_{1,i}=0$.
Similarly, we can obtain
$b_{1,i}=0$, for all $i=1,2,3$.
\end{proof}

\begin{Prop}\label{prop3.4}
For any fixed $R>0$, it holds
$$\eta_{\varepsilon}(x)=o(1),~ x\in B_{R\varepsilon}(x_{1,\varepsilon}^{(1)}).$$
\end{Prop}
\begin{proof}
Proposition \ref{3-2} and Proposition \ref{3-3} show that for any fixed $R>0$,
$\eta_{1,\varepsilon}(x)=o(1)$  in $B_{R}(0)$.
Also, we know $\eta_{1,\varepsilon}(x)=\eta_{\varepsilon}(\varepsilon x+x_{1,\varepsilon}^{(1)})$. Then $\eta_{\varepsilon}(x)=o(1),x\in  B_{R\varepsilon}(x_{1,\varepsilon}^{(1)})$.
\end{proof}

\begin{Prop}\label{prop3.5}
For large $R>0$ and fixed $\gamma\in (0,1)$, there exists $\varepsilon_0$ such that
\begin{flalign*}
|\eta_{\varepsilon}(x)|\leq \gamma,~\mbox{for}~x\in \R^3\backslash B_{R\varepsilon}(x_{1,\varepsilon}^{(1)})~\mbox{and}~\varepsilon\in (0,\varepsilon_0).
\end{flalign*}
\end{Prop}

\begin{proof}
First, $\eta_{\varepsilon}(x)$ satisfies the following equation:
\begin{flalign*}
\begin{split}
-\varepsilon^2\Delta \eta_{\varepsilon}(x)+V(x)\eta_{\varepsilon}(x)=&
E_1(x)\eta_{\varepsilon}(x)+E_2(x),
\end{split}
\end{flalign*}
 where $E_1(x)$ and $E_2(x)$ are the functions in \eqref{e}.
Let $\hat\eta_\varepsilon(x)=\eta_\varepsilon(x)+\gamma$, then
\begin{flalign*}
\begin{split}
-\varepsilon^2\Delta \hat\eta_{\varepsilon}(x)+\big(V(x)-E_1(x)\big)\hat\eta_{\varepsilon}(x)=&
 E_2(x)+\gamma\big(V(x)-E_1(x)\big),
\end{split}
\end{flalign*}
and
$$\hat\eta_{\varepsilon}(x)=\gamma+o_\varepsilon(1), ~\mbox{for}~x\in \partial  B_{R\varepsilon}(x_{1,\varepsilon}^{(1)}) ~\mbox{or}~|x|\rightarrow \infty.$$
Also for large $R>0$, \eqref{l2} and \eqref{e01}  imply
$$V(x)-E_1(x)\geq 0,~\mbox{and}~E_2(x)+\gamma\big(V(x)-E_1(x)\big)\geq 0, ~\mbox{for}~x\in \R^3\backslash B_{R\varepsilon}(x_{1,\varepsilon}^{(1)}).$$
Then by the maximum principle, we have
\begin{flalign*}
\min_{\R^3\backslash B_{R\varepsilon}(x_{1,\varepsilon}^{(1)})}\hat\eta_{\varepsilon}(x)
\geq -\max_{\partial \big(\R^3\backslash B_{R\varepsilon}(x_{1,\varepsilon}^{(1)})\big)}\big(\hat \eta_{\varepsilon}(x)\big)^-=0.
\end{flalign*}
This means
\begin{flalign*}
 \eta_{\varepsilon}(x)\geq -\gamma,~\mbox{for}~x\in \R^3\backslash B_{R\varepsilon}(x_{1,\varepsilon}^{(1)}).
\end{flalign*}
Let $\tilde{\eta}_\varepsilon(x)=\eta_\varepsilon(x)-\gamma$, then
\begin{flalign*}
\begin{split}
-\varepsilon^2\Delta \tilde\eta_{\varepsilon}(x)+\big(V(x)-E_1(x)\big)\tilde\eta_{\varepsilon}(x)=&
 E_2(x)-\gamma\big(V(x)-E_1(x)\big).
\end{split}
\end{flalign*}
Also
$$\tilde\eta_{\varepsilon}(x)=-\gamma+o_\varepsilon(1), ~\mbox{for}~x\in \partial  B_{R\varepsilon}(x_{1,\varepsilon}^{(1)}) ~\mbox{or}~|x|\rightarrow \infty,$$
and for large $R>0$,  \eqref{l2} and \eqref{e01} imply
$$V(x)-E_1(x)\geq 0,~\mbox{and}~E_2(x)-\gamma\big(V(x)-E_1(x)\big)\leq 0, ~\mbox{for}~x\in \R^3\backslash  B_{R\varepsilon}(x_{1,\varepsilon}^{(1)}).$$
Then by the maximum principle, we have
\begin{flalign*}
\max_{\R^3\backslash  B_{R\varepsilon}(x_{1,\varepsilon}^{(1)})}\tilde\eta_{\varepsilon}(x)
\leq \max_{\partial \big(\R^3\backslash   B_{R\varepsilon}(x_{1,\varepsilon}^{(1)})\big)}\big(\tilde \eta_{\varepsilon}(x)\big)^{+}=0.
\end{flalign*}
This means
\begin{flalign*}
 \eta_{\varepsilon}(x)\leq \gamma,~\mbox{for}~x\in \R^3\backslash B_{R\varepsilon}(x_{j,\varepsilon}^{(1)}).
 \end{flalign*}
\end{proof}
\begin{proof}[\textbf{Proof of Theorem \ref{th1--1}:}]
Suppose that $u^{(1)}_{\varepsilon}(x)$, $u^{(2)}_{\varepsilon}(x)$ are two different positive solutions concentrating at  the nondegenerate critical point $a_1$ of $V(x)$. From Proposition \ref{prop3.4} and Proposition \ref{prop3.5}, for small $\varepsilon$ and fixed $\gamma\in (0,1)$, we have
$|\eta_{\varepsilon}(x)|\leq\gamma$,   $x\in\R^3$,
which contradicts with $\|\eta_{\varepsilon}\|_{L^{\infty}(\R^3)}=1$. So, $u^{(1)}_{\varepsilon}(x)\equiv u^{(2)}_{\varepsilon}(x)$ for small $\varepsilon$. Also \eqref{l2--12} and \eqref{2.45} imply \eqref{aaa6}.
\end{proof}

\section{More precise estimates}

\begin{Prop}\label{Prop2.7l}
Let $u_{\varepsilon}(x)$ be the solution of \eqref{1.2} concentrating at $k~(k\geq 2)$ different nondegenerate critical  points $\{a_1,\cdots,a_k\}\subset \R^3$ of $V(x)$.
 Then it holds
\begin{flalign}\label{2--12}
|x_{j,\varepsilon}-a_{j}|=O(\varepsilon),~\mbox{for}~j=1,2,\cdots,k.
\end{flalign}
Furthermore, there exist  $j_0\in \{1,\cdots,k\}$, $C_1>0$ and $C_2>0$  such that
\begin{flalign}\label{2--6}
C_1\varepsilon \leq |x_{j_0,\varepsilon}-a_{j_0}|\leq C_2 \varepsilon.
\end{flalign}
\end{Prop}

\begin{proof}
First, for the small fixed constant $\bar d>0$, taking $u(x)=u_\varepsilon(x)$ and  $\Omega=B_{d}(x_{j,\varepsilon})$  in the Pohozaev identity \eqref{2.5} with any $d\in (\bar d, 2\bar d)$,  we have
\begin{flalign}\label{2.11}
\begin{aligned}
\int_{B_{d}(x_{j,\varepsilon})}\frac{\partial V(x)}{\partial x^i}u^2_\varepsilon(x)dx
=\int_{\partial B_{d}(x_{j,\varepsilon})}A(x)d\sigma
+A_1, ~\mbox{for}~i=1,2,3,
\end{aligned}
\end{flalign}
where $A(x)$ is the function in \eqref{l3.3}
and
\begin{flalign}\label{c-1}
A_1=\frac{1}{8\pi \varepsilon^2}\int_{ B_{d}(x_{j,\varepsilon})}\int_{\R^3}u^2_{\varepsilon}(x)u^2_{\varepsilon}(\xi)\frac{x^i-\xi^i}{|x-\xi|^3}d\xi dx.
\end{flalign}
By \eqref{A.1}, \eqref{B.3} and Proposition \ref{pr2.1}, for any $d\in (\bar d, 2\bar d)$, we obtain
\begin{flalign}\label{2-11}
\begin{aligned}
\mbox{LHS of}~\eqref{2.11}&=\int_{B_{d}(x_{j,\varepsilon})}\frac{\partial V(x)}{\partial x^i} U^2_{a_j}(\frac{x-x_{j,\varepsilon}}{\varepsilon})dx+
O\big(e^{-\eta/\varepsilon}+\|U_{a_j}(\frac{\cdot-x_{j,\varepsilon}}{\varepsilon})
\|_{\varepsilon}\|w_\varepsilon\|_{\varepsilon}
+\|w_\varepsilon\|^2_{\varepsilon}\big)\\&=\int_{B_{d}(x_{j,\varepsilon})}\frac{\partial V(x)}{\partial x^i} U^2_{a_j}(\frac{x-x_{j,\varepsilon}}{\varepsilon})dx+
O\big(\varepsilon^{5}+\varepsilon^
    {3}\max_{j=1,\cdots,k}|x_{j,\varepsilon}-a_j|^{2}\big).
\end{aligned}
\end{flalign}
Next similar to \eqref{4-1}, for any $d\in (\bar d, 2\bar d)$,  from \eqref{laa1}, we have
\begin{flalign}\label{b5}
\begin{aligned}
 &\int_{B_{d}(x_{j,\varepsilon})}\frac{\partial V(x)}{\partial x^i}U_{a_j}^2\big(\frac{x-x_{j,\varepsilon}}{\varepsilon}\big)dx\\&
=  \sum^3_{l=1}\frac{\partial^2 V(a_{j})}{\partial x^i\partial x^l}
\int_{B_{d}(x_{j,\varepsilon})} (x^l-
a^l_{j})U_{a_j}^2\big(\frac{x-x_{j,\varepsilon}}{\varepsilon}\big)dx+o\big(\int_{B_{d}(x_{j,\varepsilon})} |x-a_{j}|U_{a_j}^2\big(\frac{x-x_{j,\varepsilon}}{\varepsilon}\big)dx\big)
 \\&
= \varepsilon^{3} \big(\int_{\R^3} U_{a_j}^2(x)dx\big)\sum^3_{l=1}\frac{\partial^2 V(a_{j})}{\partial x^i\partial x^l}
(x^l_{j,\varepsilon}-a^l_{j}) +o\big(\varepsilon^4+\varepsilon^3\max_{j=1,\cdots,k}|x_{j,\varepsilon}-a_j|\big),
\end{aligned}
\end{flalign}
where $x^l_{j,\varepsilon}, a^l_{j}$ are the $l$-th components of $x_{j,\varepsilon}, a_{j}$.
So \eqref{2-11} and \eqref{b5} imply
\begin{flalign}\label{2-13}
\begin{aligned}
\mbox{LHS of}~\eqref{2.11}&=\varepsilon^{3} \big(\int_{\R^3} U_{a_j}^2(x)dx\big)\sum^3_{l=1}\frac{\partial^2 V(a_{j})}{\partial x^i\partial x^l}
(x^l_{j,\varepsilon}-a^l_{j}) +o\big(\varepsilon^4+\varepsilon^3\max_{j=1,\cdots,k}|x_{j,\varepsilon}-a_j|\big).
\end{aligned}
\end{flalign}
On the other hand, similar to \eqref{l5}, there exists $\delta_\varepsilon\in (\bar d, 2\bar d)$, we get
\begin{flalign}\label{2-19}
\int_{\partial B_{d_\varepsilon}(x_{j,\varepsilon})}A(x)d\sigma=O\big(\varepsilon^{7}\big).
\end{flalign}
Let $d=d_\varepsilon$ in \eqref{2.11}, then combining \eqref{2-19} and \eqref{claim} below in \textbf{Proposition \ref{prop2.3}}, we deduce that
\begin{flalign}\label{2-12}
\begin{split}
\mbox{RHS of}~\eqref{2.11}&=O(\varepsilon^{4}+\varepsilon^
    {3}\max_{j=1,\cdots,k}|x_{j,\varepsilon}-a_j|^{4}\big).
\end{split}
\end{flalign}
So \eqref{2-13} and \eqref{2-12} imply
\begin{flalign*}
\begin{aligned}
\sum^3_{l=1}\frac{\partial^2 V(a_{j})}{\partial x^i\partial x^l}
(x^l_{j,\varepsilon}-a^l_{j})=o\big(\max_{j=1,\cdots,k}|x_{j,\varepsilon}-a_j|\big)+O(\varepsilon).
\end{aligned}
\end{flalign*}
This means that \eqref{2--12} holds.

Moreover, \eqref{2--12}, \eqref{2.11}, \eqref{2-13}, \eqref{2-19} and \eqref{claim1} below in \textbf{Proposition \ref{prop2.3}} imply \eqref{2--6}.
\end{proof}

\begin{Prop}\label{prop2.3}
For the small fixed constant $\bar d>0$ and any $d\in (\bar d, 2\bar d)$, it holds
\begin{flalign}\label{claim}
A_1=O\big(\varepsilon^{4}\big).
\end{flalign}
Furthermore, there exist $i_0\in \{1,2,3\}$, $j_0\in \{1,\cdots,k\}$, $C_3>0$ and $C_4>0$ such that
\begin{flalign}\label{claim1}
C_3\varepsilon^4\leq |A_{1}|\leq C_4\varepsilon^4.
\end{flalign}
\end{Prop}
\begin{proof}
First, $A_1$  can be written as follows:
\begin{flalign}\label{c1}
A_1=A_{1,1}+A_{1,2}+A_{1,3}+A_{1,4},
\end{flalign}
where
\begin{flalign*}
&A_{1,1}=\frac{1}{8\pi \varepsilon^2}\int_{ B_{d}(x_{j,\varepsilon})}\int_{\R^3}U^2_{a_j}(\frac{x-x_{j,\varepsilon}}{\varepsilon})
u^2_{\varepsilon}(\xi)\frac{x^i-\xi^i}{|x-\xi|^3}d\xi dx,&
\end{flalign*}
\begin{flalign*}
&A_{1,2}=\frac{1}{4\pi \varepsilon^2}\int_{ B_{d}(x_{j,\varepsilon})}\int_{\R^3}U_{a_j}(\frac{x-x_{j,\varepsilon}}{\varepsilon})w_\varepsilon(x)
u^2_{\varepsilon}(\xi)\frac{x^i-\xi^i}{|x-\xi|^3}d\xi dx,&
\end{flalign*}
\begin{flalign*}
&A_{1,3}=\frac{1}{8\pi \varepsilon^2}\int_{ B_{d}(x_{j,\varepsilon})}\int_{\R^3} w^2_{\varepsilon}(x)
u^2_{\varepsilon}(\xi)\frac{x^i-\xi^i}{|x-\xi|^3}d\xi dx,&
\end{flalign*}

\begin{flalign*}
&A_{1,4}=\frac{1}{8\pi \varepsilon^2}\int_{ B_{d}(x_{j,\varepsilon})}\int_{\R^3}W_{j,\varepsilon}(x)
 \big(2u_{\varepsilon}(x)-W_{j,\varepsilon}(x)\big)
u^2_{\varepsilon}(\xi)\frac{x^i-\xi^i}{|x-\xi|^3}d\xi dx,&
\end{flalign*}
and $W_{j,\varepsilon}(x)$ is the function in \eqref{lm3}.

Then by \eqref{aa5} and \eqref{2.45}, we get
\begin{flalign}\label{c22}
A_{1,3}=O\big(\varepsilon^{-4}\| u_\varepsilon
\|^2_{\varepsilon}\|w_\varepsilon\|^2_{\varepsilon}\big)=O\big(\varepsilon^{6} \big).
\end{flalign}
Also, from \eqref{l4} and \eqref{A.1}, we obtain
\begin{flalign}\label{c--2}
A_{1,4}=O(e^{-\eta/\varepsilon}).
\end{flalign}
Then \eqref{c1}, \eqref{c22}, \eqref{c--2}, \eqref{c-2}, \eqref{c-3} and \eqref{c40} imply
\begin{flalign}\label{c37}
\begin{split}
A_{1}=&\sum^k_{l=1,l\neq j} \frac{(a^i_{j}-a^i_{l})\varepsilon}{8\pi}\int_{\R^3}\int_{\R^3}
U^2_{a_j}(x)
U^2_{a_l}(\xi+\frac{x_{j,\varepsilon}-x_{l,\varepsilon}}{\varepsilon})|x-\xi|^{-3}d\xi dx+o(\varepsilon^4).
\end{split}
\end{flalign}
Now by the exponential decay of $U_{a_j}(x)$, we have
\begin{flalign}\label{c46}
\begin{split}
\int_{\R^3}\int_{\R^3}&
U^2_{a_j}(x)
U^2_{a_l}(\xi+\frac{x_{j,\varepsilon}-x_{l,\varepsilon}}{\varepsilon})|x-\xi|^{-3}d\xi dx\\
=&\int_{ B_{d/\varepsilon}(0)}\int_{\R^3 \backslash B_{2d/\varepsilon}(0)}
U^2_{a_j}(x)
U^2_{a_l}(\xi+\frac{x_{j,\varepsilon}-x_{l,\varepsilon}}{\varepsilon})|x-\xi|^{-3}d\xi dx+O\big(e^{-\eta/\varepsilon}\big)
\\
=
&O\big(\varepsilon^3
\int_{\R^3}
U^2_{a_j}(x)dx  \cdot \int_{\R^3}
U^2_{a_l}(\xi+\frac{x_{j,\varepsilon}-x_{l,\varepsilon}}{\varepsilon})d\xi\big)+O\big(e^{-\eta/\varepsilon}\big)=O\big(\varepsilon^3\big).
\end{split}
\end{flalign}
Then \eqref{c37} and \eqref{c46} imply \eqref{claim}.

On the other hand, since $a_j$ are different points, we can take $i=i_0$ and $j=j_0$ such that
\begin{flalign}\label{c35}
\sum^k_{l=1,l\neq j_0} (a^{i_0}_{j_0}-a^{i_0}_{l})>0.
\end{flalign}
Also, since
$B_{\tilde{d}/\varepsilon}(-\frac{x_{j,\varepsilon}-x_{l,\varepsilon}}{\varepsilon})\subset \R^3 \backslash B_{2d/\varepsilon}(0)$ for small fixed $\tilde{d}>0$, we deduce
\begin{flalign}\label{c36}
\begin{split}
\int_{\R^3}\int_{\R^3}&
U^2_{a_j}(x)
U^2_{a_l}(\xi+\frac{x_{j,\varepsilon}-x_{l,\varepsilon}}{\varepsilon})|x-\xi|^{-3}d\xi dx\\
=&\int_{ B_{d/\varepsilon}(0)}\int_{\R^3 \backslash B_{2d/\varepsilon}(0)}
U^2_{a_j}(x)
U^2_{a_l}(\xi+\frac{x_{j,\varepsilon}-x_{l,\varepsilon}}{\varepsilon})|x-\xi|^{-3}d\xi dx+O\big(e^{-\eta/\varepsilon}\big)
\\
\geq
&\int_{ B_{d/\varepsilon}(0)}\int_{B_{\tilde{d}/\varepsilon}(-\frac{x_{j,\varepsilon}-x_{l,\varepsilon}}{\varepsilon})}
U^2_{a_j}(x)
U^2_{a_l}(\xi+\frac{x_{j,\varepsilon}-x_{l,\varepsilon}}{\varepsilon})|x-\xi|^{-3}d\xi dx+O\big(e^{-\eta/\varepsilon}\big)
\\
\geq &
\int_{ B_{d/\varepsilon}(0)}\int_{B_{\tilde{d}/\varepsilon}(0)}
U^2_{a_j}(x)
U^2_{a_l}(\xi)|x-(\xi-\frac{x_{j,\varepsilon}-x_{l,\varepsilon}}{\varepsilon})|^{-3}d\xi dx+O\big(e^{-\eta/\varepsilon}\big)
\\
\geq &{C}{\varepsilon^3}\big(\int_{ B_{d/\varepsilon}(0)}
U^2_{a_j}(x)dx \big)\big(\int_{ B_{\tilde{d}/\varepsilon}(0)}
U^2_{a_l}(\xi)d\xi \big)+Ce^{-\eta/\varepsilon}.
\end{split}
\end{flalign}
Then \eqref{claim}, \eqref{c37}, \eqref{c35}, \eqref{c36} imply \eqref{claim1} for some  $i_0$, $j_0$, $C_3>0$ and $C_4>0$.
\end{proof}

\begin{Prop}\label{prop3--1}
Let $u^{(1)}_{\varepsilon}(x)$ and $u^{(2)}_{\varepsilon}(x)$ be the solutions of \eqref{1.2} concentrating at $k~(k\geq 2)$ different nondegenerate critical  points $\{a_1,\cdots,a_k\}\subset \R^3$ of $V(x)$, then
\begin{flalign}\label{aaaa}
|x^{(1)}_{j,\varepsilon}-x^{(2)}_{j,\varepsilon}|=o(\varepsilon),~\mbox{for}~j=1,2,\cdots,k.
\end{flalign}
\end{Prop}
\begin{proof}
First, for the small fixed constant $\bar d>0$, taking $u(x)=u^{(m)}_\varepsilon(x)=\displaystyle\sum_{l=1}^{k}
U_{a_l}(\frac{x-x^{(m)}_{l,\varepsilon}}{\varepsilon})
+w^{(m)}_\varepsilon(x)$ and  $\Omega=B_{d}(x^{(m)}_{j,\varepsilon})$ in the Pohozaev identity \eqref{2.5} with any $d\in (\bar d, 2\bar d)$ and $m=1,2$, then
\begin{flalign}\label{2--11}
\begin{aligned}
\int_{B_{d}(x^{(m)}_{j,\varepsilon})}\frac{\partial V(x)}{\partial x^i}\big(u^{(m)}_\varepsilon(x)\big)^2dx
=\int_{\partial B_{d}(x^{(m)}_{j,\varepsilon})}A^{(m)}(x)d\sigma+A^{(m)}_1,~\mbox{for}~i=1,2,3,
\end{aligned}
\end{flalign}
where
\begin{flalign*}
A^{(m)}(x)=\big(-\varepsilon^2|\nabla u^{(m)}_{\varepsilon}(x)|^2+ \big(u^{(m)}_{\varepsilon}(x)\big)^2(V(x)-\frac{1}{8\pi \varepsilon^2}\int_{\R^3}\frac{ \big(u^{(m)}_{\varepsilon}(\xi)\big)^2}{|x-\xi|}d\xi)\big)
\nu_i(x),
\end{flalign*}
and
\begin{flalign*}
A^{(m)}_1=\frac{1}{8\pi \varepsilon^2}\int_{ B_{d}(x^{(m)}_{j,\varepsilon})}\int_{\R^3}\big(u^{(m)}_{\varepsilon}(x)\big)^2
\big(u^{(m)}_{\varepsilon}(\xi)\big)^2\frac{x^i-\xi^i}{|x-\xi|^3}d\xi dx.
\end{flalign*}
From \eqref{2--11}, for any $d\in (\bar d, 2\bar d)$, we have
\begin{flalign}\label{2---11}
\begin{split}
\int_{B_{d}(x^{(1)}_{j,\varepsilon})}&\frac{\partial V(x)}{\partial x^i}\big(u^{(1)}_\varepsilon(x)\big)^2dx-
\int_{B_{d}(x^{(2)}_{j,\varepsilon})}\frac{\partial V(x)}{\partial x^i}\big(u^{(2)}_\varepsilon(x)\big)^2dx\\&=\int_{\partial B_{d}(x^{(1)}_{j,\varepsilon})}A^{(1)}(x)d\sigma-\int_{\partial B_{d}(x^{(2)}_{j,\varepsilon})}A^{(2)}(x)d\sigma+\big(A^{(1)}_1-A^{(2)}_1\big).
\end{split}
\end{flalign}
Then \eqref{2--12}, \eqref{2.11} and \eqref{2-13} imply
\begin{flalign}\label{2---13}
\begin{aligned}
\mbox{LHS of}~\eqref{2---11}&=\varepsilon^{3} \big(\int_{\R^3} U_{a_j}^2(x)dx\big)\sum^3_{l=1}\frac{\partial^2 V(a_{j})}{\partial x^i\partial x^l}
(x^{(1),l}_{j,\varepsilon}-x^{(2),l}_{j,\varepsilon}) +o\big(\varepsilon^4\big),
\end{aligned}
\end{flalign}
where $x^{(1),l}_{j,\varepsilon}, x^{(2),l}_{j,\varepsilon}$ are the $l$-th components of $x^{(1)}_{j,\varepsilon}, x^{(2)}_{j,\varepsilon}$.

On the other hand, similar to \eqref{2-19}, there exists $d_\varepsilon\in (\bar d, 2\bar d)$ such that
\begin{flalign}\label{2----2}
\int_{\partial B_{d_\varepsilon}(x^{(m)}_{j,\varepsilon})}A^{(m)}(x)d\sigma=O\big(e^{-\eta/\varepsilon}\big),~\mbox{for}~m=1,2.
\end{flalign}
Also, from \eqref{c40}, we know
\begin{flalign}\label{2----1}
A^{(m)}_1=\sum^k_{l=1,l\neq j} \frac{(a^{i}_{j}-a^{i}_{l})\varepsilon}{8\pi}\int_{\R^3}\int_{\R^3}
U^2_{a_j}(x)
U^2_{a_l}(\xi+\frac{x^{(m)}_{j,\varepsilon}-x^{(m)}_{l,\varepsilon}}{\varepsilon})|x-\xi|^{-3}d\xi dx+o(\varepsilon^4),
\end{flalign}
where $a^{i}_{j}, a^{i}_{l}$ are the $i$-th components of $a_{j},a_{l}$.

Then  \eqref{2----2}, \eqref{2----1} and  the exponential decay of $U_{a_j}(x)$ imply
\begin{flalign}\label{2-14}
\begin{aligned}
\mbox{RHS of}~\eqref{2---11}&=\sum^k_{l=1,l\neq j} \frac{(a^{i}_{j}-a^{i}_{l})\varepsilon}{8\pi}\int_{B_{d/\varepsilon}(0)}\int_{B_{d/\varepsilon}(0)}
U^2_{a_j}(x)
U^2_{a_l}(\xi)D(x,\xi)d\xi dx+o(\varepsilon^4),
\end{aligned}
\end{flalign}
where
\begin{flalign*}
D(x,\xi)=|x-\xi+\frac{x^{(1)}_{j,\varepsilon}-x^{(1)}_{l,\varepsilon}}{\varepsilon}|^{-3}
-|x-\xi+\frac{x^{(2)}_{j,\varepsilon}-x^{(2)}_{l,\varepsilon}}{\varepsilon}|^{-3}.
\end{flalign*}
Then using \eqref{2--2}, we get
\begin{flalign}\label{2--23}
D(x,\xi)=O\big(|x^{(1)}_{j,\varepsilon}
-x^{(2)}_{j,\varepsilon}+x^{(1)}_{l,\varepsilon}
-x^{(2)}_{l,\varepsilon}|\varepsilon^3\big)=o(\varepsilon^3),~\mbox{for}~ x\in B_{d/\varepsilon}(0),~\xi\in B_{d/\varepsilon}(0).
\end{flalign}
Then \eqref{2-14} and \eqref{2--23} imply
\begin{flalign}\label{2-15}
\begin{aligned}
\mbox{RHS of}~\eqref{2---11}&= o(\varepsilon^4).
\end{aligned}
\end{flalign}
So from \eqref{2---13} and \eqref{2-15}, we obtain \eqref{aaaa}.
\end{proof}
\begin{Prop}\label{prop3--3}
Let $u^{(1)}_{\varepsilon}(x)$ and $u^{(2)}_{\varepsilon}(x)$ be the solutions of \eqref{1.2} concentrating at $k~(k\geq 2)$ different nondegenerate critical  points $\{a_1,\cdots,a_k\}\subset \R^3$ of $V(x)$, then
\begin{flalign}\label{5-8}
\|w^{(1)}_{\varepsilon}-w^{(2)}_{\varepsilon}\|_{\varepsilon}=o(\varepsilon^{7/2}).
\end{flalign}
\end{Prop}
\begin{proof}
First, from \eqref{B--1}, we obtain
 \begin{flalign}\label{d11}
 \begin{split}
M_{\varepsilon}\big(x,w^{(1)}_{\varepsilon}(x)-w^{(2)}_{\varepsilon}(x)\big)
    =M_{\varepsilon}\big(x,w^{(1)}_{\varepsilon}(x)\big)
   -M_{\varepsilon}\big(x,w^{(2)}_{\varepsilon}(x)\big)=\bar N_{\varepsilon}(x)+\bar l_{\varepsilon}(x),
 \end{split}
 \end{flalign}
where
\begin{flalign} \label{B--3l}
\begin{split}
\bar N_{\varepsilon}(x)=&\frac{1}{8\pi \varepsilon^2}
 \big(\int_{\R^3}\frac{\big(w_{\varepsilon}^{(1)}(\xi)\big)^2}{|x-\xi|}d\xi\big)
\big(R^{(1)}_\varepsilon(x)+ w^{(1)}_{\varepsilon}(x)\big)
+\frac{w^{(1)}_{\varepsilon}(x)}{4\pi \varepsilon^{2}}\int_{\R^{3}}\frac{R^{(1)}_\varepsilon(\xi)w^{(1)}_{\xi}(\xi)}{|x-\xi|}d\xi\\&
-\frac{1}{8\pi \varepsilon^2}
 \big(\int_{\R^3}\frac{\big(w_{\varepsilon}^{(2)}(\xi)\big)^2}{|x-\xi|}d\xi\big)
\big(R^{(2)}_\varepsilon(x)+ w^{(2)}_{\varepsilon}(x)\big)
-\frac{w^{(2)}_{\varepsilon}(x)}{4\pi \varepsilon^{2}}\int_{\R^{3}}\frac{R^{(2)}_\varepsilon(\xi)w^{(2)}_{\xi}(\xi)}{|x-\xi|}d\xi ,
 \end{split}
\end{flalign}
and
 \begin{flalign} \label{B--4l}
 \begin{split}
 \bar l_{\varepsilon}(x)=&
      \frac{W^{(1)}_{j,\varepsilon}(x)}{8\pi \varepsilon^{2}}\int_{\R^{3}}
      \frac{W^{(1)}_{j,\varepsilon}(\xi)U_{a_{j}}(\frac{\xi-x^{(1)}_{j,\varepsilon}}{\varepsilon}) }{|x-\xi|}d\xi
      +\sum_{j=1}^k\big(V(a_j) - V(x)\big)U_{a_j}
      (\frac{x-x^{(1)}_{j,\varepsilon}}{\varepsilon})\\&
      -\frac{W^{(2)}_{j,\varepsilon}(x)}{8\pi \varepsilon^{2}}\int_{\R^{3}}
      \frac{W^{(2)}_{j,\varepsilon}(\xi)U_{a_{j}}(\frac{\xi-x^{(2)}_{j,\varepsilon}}{\varepsilon}) }{|x-\xi|}d\xi
      -\sum_{j=1}^k\big(V(a_j) - V(x)\big)U_{a_j}
      (\frac{x-x^{(2)}_{j,\varepsilon}}{\varepsilon}),
 \end{split}
 \end{flalign}
 where
$R^{(m)}_\varepsilon(x)$ and ${W}^{(m)}_{j,\varepsilon}(x)$ are the functions in \eqref{lm2} and \eqref{lm3} for $m=1,2$.

Next, using \eqref{aa5} and \eqref{2.45}, we obtain
\begin{flalign}\label{d12}
\begin{split}
&\int_{\R^3}\bar N_{\varepsilon}(x)\big(w^{(1)}_{\varepsilon}(x)-w^{(2)}_{\varepsilon}(x)\big)dx
\\=&O\big(\varepsilon^{-4}
|x^{(1)}_{j,\varepsilon}-x^{(2)}_{j,\varepsilon}|\cdot
\|R^{(1)}_\varepsilon\|_{\varepsilon}\big(\sum^2_{m=1}\|w^{(m)}_\varepsilon\|_\varepsilon\big)^3\big)+
O\big(\varepsilon^{-3}
\big(\sum^2_{m=1}\|w^{(m)}_\varepsilon\|_\varepsilon\big)^3
\|w^{(1)}_\varepsilon-w^{(2)}_\varepsilon\|_\varepsilon)\\&
+
O\big(\varepsilon^{-3}\|R^{(1)}_\varepsilon\|_{\varepsilon}\big(\sum^2_{m=1}\|w^{(m)}_\varepsilon\|_\varepsilon\big)^2
\|w^{(1)}_\varepsilon-w^{(2)}_\varepsilon\|_\varepsilon\big)\\=&
o(\varepsilon^{9})+
O(\varepsilon^{11/2})\|w^{(1)}_\varepsilon-w^{(2)}_\varepsilon\|_\varepsilon.
\end{split}
\end{flalign}
Also, similar to \eqref{B.14l}, we have
\begin{flalign}\label{d13}
\begin{split}
\int_{\R^3}\bar l_{\varepsilon}(x)\big(w^{(1)}_{\varepsilon}(x)-w^{(2)}_{\varepsilon}(x)\big)dx
=&o(\varepsilon^{7/2})\|w^{(1)}_\varepsilon-w^{(2)}_\varepsilon\|_\varepsilon.
\end{split}
\end{flalign}
On the other hand, by Proposition \ref{lem-A.1},  there exists $\rho'$ such that
 \begin{flalign}\label{d14}
 \begin{split}
\int_{\R^3} M_{\varepsilon}\big(x,w^{(1)}_{\varepsilon}(x)-w^{(2)}_{\varepsilon}(x)\big)
\big(w^{(1)}_{\varepsilon}(x)-w^{(2)}_{\varepsilon}(x)\big)dx
   \geq \rho'\|w^{(1)}_\varepsilon-w^{(2)}_\varepsilon\|^2_\varepsilon.
 \end{split}
 \end{flalign}
Then \eqref{d11}, \eqref{d12}, \eqref{d13} and \eqref{d14} imply \eqref{5-8}.
\end{proof}

\begin{Prop}\label{lp1}
It holds
\begin{flalign}\label{l2--20}
w^{(1)}_{\varepsilon}(x)-w^{(2)}_{\varepsilon}(x)=o\big(\varepsilon^2\big),~\mbox{in}~ \bigcup_{j=1}^k B_{R\varepsilon }(x^{(1)}_{j,\varepsilon}).
\end{flalign}
\end{Prop}
\begin{proof}
Let $\bar{w}_{j,\varepsilon}(x):=w^{(1)}_{j,\varepsilon}(x)- w^{(2)}_{j,\varepsilon}(x)=w^{(1)}_\varepsilon(\varepsilon x+x^{(1)}_{j,\varepsilon})-w^{(2)}_\varepsilon(\varepsilon x+x^{(1)}_{j,\varepsilon})$, then we have
\begin{flalign*}
\begin{split}
-\Delta \bar{w}_{j,\varepsilon}(x)=& \bar G_\varepsilon(\varepsilon x+x^{(1)}_{j,\varepsilon})
+\bar N_{\varepsilon}(\varepsilon x+x^{(1)}_{j,\varepsilon})+\bar l_{\varepsilon}(\varepsilon x+x^{(1)}_{j,\varepsilon}),
\end{split}
\end{flalign*}
$$ \bar G_\varepsilon(\varepsilon x+x^{(1)}_{j,\varepsilon})=G(\varepsilon x+x^{(1)}_{j,\varepsilon},w^{(1)}_{j,\varepsilon}(\varepsilon x+x^{(1)}_{j,\varepsilon}))- G(\varepsilon x+x^{(1)}_{j,\varepsilon},w^{(2)}_{j,\varepsilon}(\varepsilon x+x^{(1)}_{j,\varepsilon})),$$
where $G(\varepsilon x+x^{(1)}_{j,\varepsilon},w^{(m)}_{j,\varepsilon}(\varepsilon x+x^{(1)}_{j,\varepsilon}))$,
$\bar N_{\varepsilon}(\varepsilon x+x^{(1)}_{j,\varepsilon})$ and $\bar l_{\varepsilon}(\varepsilon x+x^{(1)}_{j,\varepsilon})$ are the functions in \eqref{B--2},
\eqref{B--3l} and \eqref{B--4l} for $m=1,2$.

So by Nash-Moser iteration (Lemma \ref{lem-A-6}), we have, for any fixed $R>0$,
\begin{flalign}\label{d16}
\begin{split}
\sup_{B_R(0)}\bar w_{j,\varepsilon}(x)\leq &
C \| \bar w_{j,\varepsilon}\|_{L^2(B_{2R}(0))}+
C\|\bar G_\varepsilon(\varepsilon \cdot +x^{(1)}_{j,\varepsilon})\|_{L^{2}(\R^3)}\\&
+C\|\bar N_\varepsilon(\varepsilon x+x^{(1)}_{j,\varepsilon})\|_{L^{2}(\R^3)}
   +C\|\bar l_{\varepsilon}(\varepsilon \cdot+x^{(1)}_{j,\varepsilon})\|_{L^{2}(\R^3)}.
\end{split}
\end{flalign}
Then using \eqref{aaaa} and \eqref{5-8}, we obtain
\begin{flalign}\label{d17}
\|\bar w_{j,\varepsilon}\|_{L^2(B_{2R}(0))}=O\big(\|\bar w_{j,\varepsilon}\|_{L^2(\R^3)}\big)=O\big(\varepsilon^{-3/2}\|w^{(1)}_\varepsilon
-w^{(2)}_\varepsilon\|_\varepsilon\big)=o(\varepsilon^2).
\end{flalign}
Also, from \eqref{5-8}, \eqref{aa5}, \eqref{A.1}, \eqref{B--3} and \eqref{2.45}, we deduce
\begin{flalign}\label{d18}
\begin{split}
&\|\bar G_\varepsilon(\varepsilon \cdot +x^{(1)}_{j,\varepsilon})\|_{L^{2}(\R^3)}=\varepsilon^{-\frac{3}{2}}\|\bar G_\varepsilon (\cdot)\|_{L^{2}(\R^3)}
\\=&
O\big(\varepsilon^{-\frac{3}{2}}\|w^{(1)}_\varepsilon
-w^{(2)}_\varepsilon\|_\varepsilon\big)+
O\big(\varepsilon^{-\frac{7}{2}}\|R^{(1)}_\varepsilon\|_\varepsilon
\big(\int_{\R^3}\int_{\R^{3}} \frac{\big(R^{(1)}_\varepsilon(\xi)\big)^2} {|x-\xi|^{2} }
\big(w^{(1)}_{\varepsilon}(x)-w^{(2)}_{\varepsilon}(x)\big)^2dxd\xi\big)^{\frac{1}{2}}\big)\\&
+O\big(\varepsilon^{-\frac{9}{2}}
|x^{(1)}_{j,\varepsilon}-x^{(2)}_{j,\varepsilon}|\cdot\|R^{(1)}_\varepsilon\|_\varepsilon
\big(\int_{\R^3}\int_{\R^{3}} \frac{\big(R^{(1)}_\varepsilon(\xi)\big)^2} {|x-\xi|^{2} }
\big(w^{(1)}_{\varepsilon}(x)\big)^2dxd\xi\big)^{\frac{1}{2}}\big)\\=&
O\big(\varepsilon^{-\frac{3}{2}}\|w^{(1)}_\varepsilon
-w^{(2)}_\varepsilon\|_\varepsilon\big)+O\big( \varepsilon^{-\frac{5}{2}}
|x^{(1)}_{j,\varepsilon}-x^{(2)}_{j,\varepsilon}|\cdot\|w^{(1)}_\varepsilon
\|_\varepsilon\big)=o(\varepsilon^2).
\end{split}
\end{flalign}
Next, similar to \eqref{d18}, \eqref{B.14}  and \eqref{B.14l}, we can obtain
\begin{flalign}\label{d19}
\begin{split}
\|\bar N_{\varepsilon}(\varepsilon x+x^{(1)}_{j,\varepsilon})\|_{L^{2}(\R^3)} =
\varepsilon^{-\frac{3}{2}}\|\bar N_\varepsilon (\cdot)\|_{L^{2}(\R^3)}=
O\big(\varepsilon^{-\frac{3}{2}}\cdot \big(\sum^2_{m=1}\|N\big(x,w^{(m)}_{\varepsilon}(x)\|_{L^{2}(\R^3)}\big)
\big)=
O(\varepsilon^4),
\end{split}
\end{flalign}
and
\begin{flalign}\label{d20}
\begin{split}
\|\bar l_{\varepsilon}(\varepsilon x+x^{(1)}_{j,\varepsilon})\|_{L^{2}(\R^3)}=
O(\varepsilon ^{-\frac{1}{2}}\cdot|x^{(1)}_{j,\varepsilon}-x^{(2)}_{j,\varepsilon}|
\cdot\|U_{a_j}(\frac{\cdot-x_{j,\varepsilon}}{\varepsilon})
\|_{\varepsilon})+O\big(e^{-\eta/\varepsilon}\big)=o(\varepsilon^2).
\end{split}\end{flalign}
Then \eqref{d16}, \eqref{d17}, \eqref{d18}, \eqref{d19} and \eqref{d20} imply
$$\displaystyle\sup_{B_R(0)}\bar w_{j,\varepsilon}(x)=o(\varepsilon^2), ~\mbox{for}~j=1,2,\cdots,k.$$
This means \eqref{l2--20}.
\end{proof}

\section{Proof of Theorem \ref{th1.1}}\label{s4}
\setcounter{equation}{0}
\begin{Prop}\label{l3-2}
Let $\eta_{j,\varepsilon}(x)=\eta_{\varepsilon}(\varepsilon x+x_{j,\varepsilon}^{(1)})$ for $j=1,2,\cdots, k$
and $k\geq 2$, then taking
a subsequence necessarily, it holds
$$\eta_{j,\varepsilon}(x)\rightarrow\sum_{i=1}^3 d_{j,i}\frac{\partial U_{a_j}(x)}{\partial x^i}$$
uniformly in $C^1(B_R(0))$ for any $R>0$, where  $\eta_{\varepsilon}(x)$ is the function in \eqref{3.1} and $d_{j,i}$, $i=1,2,3$ are some constants.
\end{Prop}
\begin{proof}
The proof is just as similar as that of  Proposition \ref{3-2}.
Here we want to point out that
to obtain \eqref{la} for the case $k\geq 2$, the estimate \eqref{aaaa} is crucial.
\end{proof}
\begin{Prop}\label{l3-3}
Let $d_{j,i}$ be as in Proposition \ref{l3-2}, then
we have
\begin{equation}\label{lpp}
d_{j,i}=0,~~\mbox{for all}~j=1,\cdots,k,~i=1,2,3.
\end{equation}
\end{Prop}
\begin{proof}
Since $u_{\varepsilon}^{(1)}(x)$, $u_{\varepsilon}^{(2)}(x)$ are the positive solutions of \eqref{1.2},
 for the small fixed constant $\bar d>0$ and any $\delta\in (\bar d, 2\bar d)$, similar to \eqref{3.14},  we have
\begin{flalign}\label{l3.14}
\begin{aligned}
\displaystyle \int_{B_{\delta}(x_{1,\varepsilon}^{(1)})}\frac{\partial V(x)}{\partial x^i}\big(u_{\varepsilon}^{(1)}(x)+u_{\varepsilon}^{(2)}(x)\big)\cdot\eta_{\varepsilon}(x)dx =&
\int_{\partial B_{\delta}(x_{j,\varepsilon}^{(1)})}B(x)d\sigma+F_1+F_2,
\end{aligned}
\end{flalign}
where $B(x)$ is the function in \eqref{lb},
 $$F_1=\frac{1}{8\pi \varepsilon^2}
\int_{ B_{\delta}(x_{j,\varepsilon}^{(1)})}\int_{\R^3}
\big(u^{(1)}_{\varepsilon}(x)\big)^2
\big(u^{(1)}_{\varepsilon}(\xi)+u^{(2)}_{\varepsilon}(\xi)\big)\eta_{\varepsilon}(\xi)
\frac{x^i-\xi^i}{|x-\xi|^3}d\xi dx,$$
and
 $$F_2=\frac{1}{8\pi \varepsilon^2}
\int_{B_{\delta}(x_{j,\varepsilon}^{(1)})}\int_{\R^3}
\big(u^{(2)}_{\varepsilon}(\xi)\big)^2
\big(u^{(1)}_{\varepsilon}(x)+u^{(2)}_{\varepsilon}(x)\big)\eta_{\varepsilon}(x)\frac{x^i-\xi^i}{|x-\xi|^3}d\xi dx.$$
Then similar to \eqref{4-1} and \eqref{e17l}, we know
\begin{flalign}\label{l4-1}
\textrm{LHS of \eqref{l3.14}}= \varepsilon^4 \int_{\R^3} U^2_{a_j}(x)dx\big(\sum^3_{l=1}\frac{\partial^2 V(a_j)}{\partial x^i\partial x^l}d_{j,l}\big)+o(\varepsilon^4),
\end{flalign}
and
\begin{flalign}\label{e17la}
\int_{\partial B_{\delta_\varepsilon}(x_{j,\varepsilon}^{(1)})}B(x)d\sigma=O\big(\varepsilon^{5}
\big),
~\mbox{for some}~\delta_\varepsilon\in (\bar d, 2\bar d).
\end{flalign}
Let $\delta=\delta_\varepsilon$ in \eqref{l3.14}, then  \eqref{l4-1}, \eqref{e17la} and \eqref{f-1} below in \textbf{Proposition \ref{prop3.3}}  imply
\begin{equation*}
\sum^3_{l=1}\frac{\partial^2 V(a_j)}{\partial x^i\partial x^l}d_{j,l}=o(1).
\end{equation*}
This means $d_{j,i}=0$, for $i=1,2,3$.
Similarly, we can obtain \eqref{lpp}.
\end{proof}
\begin{Prop}\label{prop3.3}
For the small fixed constant $\bar d>0$ and any $\delta\in (\bar d, 2\bar d)$, it holds
\begin{flalign}\label{f-1}
F_1+F_2=o\big(\varepsilon^{4}\big).
\end{flalign}
\end{Prop}
\begin{proof}
First, $F_{1}$ can be written as
\begin{flalign}\label{lll1}
F_1=F_{1,1}+F_{1,2}+F_{1,3}+F_{1,4},
\end{flalign}
where
\begin{flalign*}
&F_{1,1}=\frac{1}{8\pi \varepsilon^2}
\int_{ B_{\delta}(x_{j,\varepsilon}^{(1)})}\int_{\R^3}
U_{a_j}^2(\frac{x-x^{(1)}_{j,\varepsilon}}{\varepsilon})
\big(u^{(1)}_{\varepsilon}(\xi)+u^{(2)}_{\varepsilon}(\xi)\big)\eta_{\varepsilon}(\xi)
\frac{x^i-\xi^i}{|x-\xi|^3}d\xi dx,&
\end{flalign*}
\begin{flalign*}
&F_{1,2}=\frac{1}{4\pi \varepsilon^2}
\int_{ B_{\delta}(x_{j,\varepsilon}^{(1)})}\int_{\R^3}
U_{a_j} (\frac{x-x^{(1)}_{j,\varepsilon}}{\varepsilon})w_\varepsilon^{(1)}(x)
\big(u^{(1)}_{\varepsilon}(\xi)+u^{(2)}_{\varepsilon}(\xi)\big)\eta_{\varepsilon}(\xi)
\frac{x^i-\xi^i}{|x-\xi|^3}d\xi dx,&
\end{flalign*}
\begin{flalign*}
&F_{1,3}=\frac{1}{8\pi \varepsilon^2}
\int_{ B_{\delta}(x_{j,\varepsilon}^{(1)})}\int_{\R^3}\big(w_\varepsilon^{(1)}(x)\big)^2
\big(u^{(1)}_{\varepsilon}(\xi)+u^{(2)}_{\varepsilon}(\xi)\big)\eta_{\varepsilon}(\xi)
\frac{x^i-\xi^i}{|x-\xi|^3}d\xi dx,&
\end{flalign*}
\begin{flalign*}
&F_{1,4}=\frac{1}{8\pi \varepsilon^2}
\int_{ B_{\delta}(x_{j,\varepsilon}^{(1)})}\int_{\R^3}W^{(1)}_{j,\varepsilon}(x)
\big(2u_\varepsilon^{(1)}(x)-W^{(1)}_{j,\varepsilon}(x)\big)
\big(u^{(1)}_{\varepsilon}(\xi)+u^{(2)}_{\varepsilon}(\xi)\big)\eta_{\varepsilon}(\xi)
\frac{x^i-\xi^i}{|x-\xi|^3}d\xi dx,&
\end{flalign*}
and $W^{(1)}_{j,\varepsilon}(x)$ is the function in \eqref{lm3}.

Next, $F_2$ can be written as follows:
\begin{flalign}\label{lll2}
F_2=F_{2,1}+F_{2,2}+F_{2,3},
\end{flalign}
where
\begin{flalign*}
&F_{2,1}=\frac{1}{8\pi \varepsilon^2}
\int_{B_{\delta}(x_{j,\varepsilon}^{(1)})}\int_{\R^3}
\big(U_{a_j}(\frac{x-x^{(1)}_{j,\varepsilon}}{\varepsilon})+U_{a_j}(\frac{x-x^{(2)}_{j,\varepsilon}}{\varepsilon}) \big)
\eta_{\varepsilon}(x)
\big(u^{(2)}_{\varepsilon}(\xi)\big)^2\frac{x^i-\xi^i}{|x-\xi|^3}d\xi dx,&
\end{flalign*}
\begin{flalign*}
&F_{2,2}=\frac{1}{8\pi \varepsilon^2}
\int_{B_{\delta}(x_{j,\varepsilon}^{(1)})}\int_{\R^3}\big(W^{(1)}_{j,\varepsilon}(x)+W^{(2)}_{j,\varepsilon}(x)\big)
\eta_{\varepsilon}(x)
\big(u^{(2)}_{\varepsilon}(\xi)\big)^2\frac{x^i-\xi^i}{|x-\xi|^3}d\xi dx,&
\end{flalign*}
\begin{flalign*}
&
F_{2,3}=\frac{1}{8\pi \varepsilon^2}
\int_{B_{\delta}(x_{j,\varepsilon}^{(1)})}\int_{\R^3}
\big(w^{(1)}_{\varepsilon}(x)+w^{(2)}_{\varepsilon}(x)\big)
\eta_{\varepsilon}(x)\big(u^{(2)}_{\varepsilon}(\xi)\big)^2\frac{x^i-\xi^i}{|x-\xi|^3}d\xi dx.
&
\end{flalign*}
Then \eqref{aa5}, \eqref{A.1}, \eqref{3.3}, \eqref{2.45} imply
\begin{flalign}\label{f33}
\begin{split}
F_{1,3}&=O\big(\varepsilon^{-4}\|w^{(1)}_{\varepsilon}\|^2_{\varepsilon}
\|u^{(1)}_{\varepsilon}+u^{(2)}_{\varepsilon}\|_{\varepsilon}\|\eta_{\varepsilon}\|_{\varepsilon}\big)
=O\big(\varepsilon^{6}\big),~
 F_{1,4}=O\big(e^{-\eta/\varepsilon}\big)~\mbox{and}~F_{2,2}=O\big(e^{-\eta/\varepsilon}\big).
\end{split}\end{flalign}
Next from \eqref{f33},  \eqref{f1}, \eqref{f8}, \eqref{f13} and \eqref{f25}, we know
\begin{flalign}\label{f35}
F_1+F_2=G+o\big( \varepsilon^4\big),
\end{flalign}
where
$$G=\frac{1}{8\pi \varepsilon^2}
\int_{\R^3}\int_{\R^3}
\big(U_{a_j}^2(\frac{x-x^{(1)}_{j,\varepsilon}}{\varepsilon})
-U_{a_j}^2(\frac{x-x^{(2)}_{j,\varepsilon}}{\varepsilon})\big)
\big(\sum^{2}_{m=1}U_{a_j}(\frac{\xi-x^{(m)}_{j,\varepsilon}}{\varepsilon})
\big)\eta_{\varepsilon}(\xi)
\frac{x^i-\xi^i}{|x-\xi|^3}d\xi dx.$$

\smallskip

\noindent To estimate the term $G$,  we divide it into two cases.

\smallskip

\noindent{\textbf{Case 1:}}
\begin{equation}\label{le1}
|x^{(1)}_{j,\varepsilon}-x^{(2)}_{j,\varepsilon}|
=o\big(\varepsilon^{3}\big).
\end{equation}
Then from \eqref{le1}, \eqref{aa5} and \eqref{lg1}, we obtain
\begin{flalign}\label{f36}
G=O\big(\varepsilon^{-4}\cdot (\frac{|x^{(1)}_{j,\varepsilon}-x^{(2)}_{j,\varepsilon}|}{\varepsilon})
\|U_{a_j}(\frac{\cdot-x^{(1)}_{j,\varepsilon}}{\varepsilon})\|^3_{\varepsilon}
\|\eta_\varepsilon\|_\varepsilon\big)=
o\big(\varepsilon^{4}\big).
\end{flalign}
So in this case, \eqref{f35} and  \eqref{f36} imply \eqref{f-1}.

\smallskip

\noindent{\textbf{Case 2:}}  For any fixed $C_0>0$, there exists $\{\varepsilon_{i}\}^{\infty}_{i=1}$  such that
$$\displaystyle\lim_{i\rightarrow +\infty}\varepsilon_{i}=0~\mbox{and}~|x^{(1)}_{j,\varepsilon_i}-x^{(2)}_{j,\varepsilon_i}|\geq C_0\varepsilon_i^{3}, ~\mbox{for all}~i=1,\cdots,+\infty.$$
Then from \eqref{l2--20} and \eqref{A.1}, there exists some $C>0$ such that
\begin{flalign}\label{phi}
\begin{split}
J_{\varepsilon}:=&\|u^{(1)}_{\varepsilon}(\cdot)-u^{(2)}_{\varepsilon}(\cdot)\|_{L^{\infty}(\R^3)}
\geq \|u^{(1)}_{\varepsilon}(\cdot)-u^{(2)}_{\varepsilon}(\cdot)\|_{L^{\infty}
(B_{\delta}(x^{(1)}_{j,\varepsilon}))}\\ \geq &
C\big(\varepsilon^{-1}|x^{(1)}_{j,\varepsilon}-x^{(2)}_{j,\varepsilon}
|-\frac{C_0}{2}\varepsilon^{2}-e^{-\eta/\varepsilon}\big)\geq \frac{CC_0}{4\varepsilon}|x^{(1)}_{j,\varepsilon}-x^{(2)}_{j,\varepsilon}
|.
\end{split}
\end{flalign}
On the other hand, from \eqref{aaaa} and \eqref{A.1}, for small fixed $d>0$, we obtain
\begin{flalign}\label{l21}
\begin{split}
\eta_{\varepsilon}(x)=
J^{-1}_{\varepsilon}\big(U_{a_j}(\frac{x-x^{(1)}_{j,\varepsilon}}{\varepsilon})
-U_{a_j}(\frac{x-x^{(2)}_{j,\varepsilon}}{\varepsilon})\big)
+ J^{-1}_{\varepsilon}  w_{j,\varepsilon}(x)+ O\big(e^{-\eta/\varepsilon}\big),
~\mbox{in}~B_{d}(x^{(1)}_{j,\varepsilon}).
\end{split}\end{flalign}
Then   the exponential decay of $U_{a_j}(x)$ and \eqref{l21} imply
\begin{flalign}\label{l22}
\begin{split}
U_{a_j}(\frac{x-x^{(1)}_{j,\varepsilon}}{\varepsilon})\eta_{\varepsilon}(x)=&
J^{-1}_{\varepsilon}U_{a_j}(\frac{x-x^{(1)}_{j,\varepsilon}}{\varepsilon})\big(U_{a_j}(\frac{x-x^{(1)}_{j,\varepsilon}}{\varepsilon})
-U_{a_j}(\frac{x-x^{(2)}_{j,\varepsilon}}{\varepsilon})\big)\\&
+ J^{-1}_{\varepsilon} U_{a_j}(\frac{x-x^{(1)}_{j,\varepsilon}}{\varepsilon}) w_{j,\varepsilon}(x)+ O\big(e^{-\eta/\varepsilon}\big),
~\mbox{in}~\R^3.
\end{split}\end{flalign}
Then from the symmetry and \eqref{l22}, we have
\begin{flalign}\label{f40}
\begin{split}
G=&\frac{1}{8\pi \varepsilon^2J_{\varepsilon}}
\int_{\R^3}\int_{\R^3}
\big(U_{a_j}^2(\frac{x-x^{(1)}_{j,\varepsilon}}{\varepsilon})
-U_{a_j}^2(\frac{x-x^{(2)}_{j,\varepsilon}}{\varepsilon})\big)
\big(U_{a_j}(\frac{\xi-x^{(1)}_{j,\varepsilon}}{\varepsilon})+
U_{a_j}(\frac{\xi-x^{(2)}_{j,\varepsilon}}{\varepsilon})
\big)\\&
\cdot \big(w^{(1)}_{\varepsilon}(\xi)-w^{(2)}_{\varepsilon}(\xi)\big)
\frac{x^i-\xi^i}{|x-\xi|^3}d\xi dx+O\big(e^{-\eta/\varepsilon}\big)\\
=&O\big({\varepsilon}^{-2}
\|U_{a_j}(\frac{\cdot-x^{(1)}_{j,\varepsilon}}{\varepsilon})\|^3_{\varepsilon}
\|w^{(1)}_{\varepsilon}-w^{(2)}_{\varepsilon}\|_\varepsilon\big)+O\big(e^{-\eta/\varepsilon}\big)
=o\big({\varepsilon}^{6}\big).
\end{split}
\end{flalign}
So in this case, \eqref{f35} and \eqref{f40} imply \eqref{f-1}.
\end{proof}
\begin{Prop}\label{prop3.4l}
For any fixed $R>0$, it holds
$$\eta_{\varepsilon}(x)=o(1),~ x\in \bigcup_{j=1}^kB_{R\varepsilon}(x_{j,\varepsilon}^{(1)}).$$
\end{Prop}
\begin{proof}
Similar to Proposition \ref{prop3.4}, this is the result of
Proposition \ref{l3-2} and Proposition \ref{l3-3}.
\end{proof}

\begin{Prop}\label{prop3.5l}
For large $R>0$ and fixed $\gamma_1\in (0,1)$, there exists $\varepsilon_0$ such that
\begin{equation*}
|\eta_{\varepsilon}(x)|\leq \gamma_1,~\mbox{for}~x\in \R^3\backslash\bigcup_{j=1}^k B_{R\varepsilon}(x_{j,\varepsilon}^{(1)})~\mbox{and}~\varepsilon\in (0,\varepsilon_0).
\end{equation*}
\end{Prop}

\begin{proof}
Similar to the proof of Proposition \ref{prop3.5}, we replace
$B_{R\varepsilon}(x_{1,\varepsilon}^{(1)})$ by $\bigcup_{j=1}^k B_{R\varepsilon}(x_{j,\varepsilon}^{(1)})$.
\end{proof}

\begin{proof}[\textbf{Proof of Theorem \ref{th1.1}:}]
Let $u^{(1)}_{\varepsilon}(x)$, $u^{(2)}_{\varepsilon}(x)$ be two different positive solutions concentrating at  the nondegenerate critical points $\{a_1,\cdots,a_k\}$ of $V(x)$ for $k\geq 2$. Then Proposition \ref{prop3.4l} and Proposition \ref{prop3.5l} imply $|\eta_{\varepsilon}(x)|\leq\gamma_1$, for  $x\in\R^3$, small $\varepsilon$ and fixed $\gamma_1\in (0,1)$,
which contradicts with $\|\eta_{\varepsilon}\|_{L^{\infty}(\R^3)}=1$. So  $u^{(1)}_{\varepsilon}(x)\equiv u^{(2)}_{\varepsilon}(x)$ for small $\varepsilon$. Also \eqref{2--12}, \eqref{2--6} and \eqref{2.45} imply \eqref{lll} and \eqref{2a6}.
\end{proof}

\section*{Appendix}

\appendix
\renewcommand{\theequation}{A.\arabic{equation}}

\setcounter{equation}{0}

\section{Some basic estimates}
\setcounter{equation}{0}

\begin{Lem}[\textbf{Hardy-Littlewood-Sobolev inequality}, c.f. \cite{Lieb}] \label{lem-A-5}
Let $p,r>1, 0<\lambda<3$, $\frac{1}{p}+\frac{1}{r}+\frac{\lambda}{3}=2$, $f\in L^p(\R^3)$, $h\in L^r(\R^3)$, then there exists $C(\lambda,p)>0$ such that
\begin{flalign}
\int_{\R^3}\int_{\R^3}f(x)|x-y|^{-\lambda}h(y)dxdy \leq C(\lambda,p)\|f\|_{L^p(\R^3)}\|g\|_{L^r(\R^3)}.
\end{flalign}
\end{Lem}

\begin{Lem}[\textbf{Nash-Moser iteration}, c.f. Theorem 8.17 in \cite{Gilbarg}]\label{lem-A-6}
If $u\in H^1(\R^3)$ is the solution of $-\Delta u=f(x)$ in $\R^3$ and $f\in L^{q/2}(\R^3)$ for some $q>3$, then for any ball $B_{2R}(y)\subset \R^3$ and $p>1$, there exists $C=C(p,q)$ such that
\begin{flalign*}
\sup_{B_R(y)}u(x)\leq C\big(R^{-3/p}\|u\|_{L^2(B_{2R}(y))}+R^{2(1-3/q)}\|f\|_{L^{q/2}(\R^3)}\big).
\end{flalign*}
\end{Lem}

\begin{Lem}[\textbf{Decomposition lemma}, c.f. \cite{Bahri} or Lemma A.1 in
 \cite{Cao}]\label{lem2.5}
For $u(x)\in H_\varepsilon$, if there exist $\delta_0>0$, $\varepsilon_0>0$ such that
$$\|u-\sum^k_{j=1}U_{a_j} (\frac{x-x_{j,\varepsilon}}{\varepsilon})\|_\varepsilon\leq \delta_0 \varepsilon^3,~\mbox{and}~|x_{j,\varepsilon}-a_j|\leq \delta,~\mbox{for all}~\delta\in (0,\delta_0]~\mbox{and}~\varepsilon\in (0,\varepsilon_0],$$
then for all $ \delta\in (0,\delta_0]$ and $\varepsilon\in (0,\varepsilon_0]$, the following minimization problem
$$\inf\{\varepsilon^{-3}\|u-\sum^k_{j=1}U_{a_j} (\frac{x-x_{j,\varepsilon}}{\varepsilon})\|_\varepsilon; ~x_{j,\varepsilon}\in B_{4\delta}(a_j)\}$$
has a unique solution which can be written as
\begin{flalign*}
u=\sum^k_{j=1}\alpha_{j,\varepsilon} U_{a_j}(\frac{x-x_{j,\varepsilon}}{\varepsilon})+v_\varepsilon(x),
\end{flalign*}
where $|\alpha_{j,\varepsilon}-1|\leq 2\delta$, $v_\varepsilon\in
\bigcap ^k_{j=1}E_{\varepsilon,a_j,x_{j,\varepsilon}}$ and
\begin{flalign*}
\begin{split}
E_{\varepsilon,a_j,x_{j,\varepsilon}}&=\left \{ u(x)\in H_\varepsilon: \big(u(x),U_{a_j}(\frac{x-x_{j,\varepsilon}}{\varepsilon})\big)_{\varepsilon}=0,
~\big(u(x),\frac{\partial{U_{a_j}(\frac{x-x_{j,\varepsilon}}{\varepsilon})}}{\partial{x^i}}\big)_{\varepsilon}=0, ~i=1,2,3
 \right\}.
 \end{split}
\end{flalign*}
\end{Lem}

\begin{Lem}\label{lem-A-4}
Suppose $f_{\varepsilon}\in L^1({\R^3})\cap C(\R^3)$, for any fixed small $\bar d>0$ independent of $\varepsilon$ and  $x_{\varepsilon}$,
there exists a small constant $d_\varepsilon\in (\bar d,2\bar d)$  such that
\begin{flalign}\label{A.14}
\int_{\partial B_{d_\varepsilon}(x_{\varepsilon})}|f_{\varepsilon}(x)|d\sigma \leq \frac{1}{\bar d} \int_{\R^3}|f_{\varepsilon}(x)|dx.
\end{flalign}
\end{Lem}
\begin{proof}
First, for any fixed small $\bar d>0$ and  $x_{\varepsilon}$,
\begin{flalign}\label{aA.14}
\int^{2\bar d}_{\bar d}\int_{\partial B_{r}(x_{\varepsilon})}|f_{\varepsilon}(x)|d\sigma dr= \int_{B_{2\bar d}(x_{\varepsilon})\setminus B_{\bar d}(x_{\varepsilon})}|f_{\varepsilon}(x)|dx\leq  \int_{\R^3}|f_{\varepsilon}(x)|dx.
\end{flalign}
Also $\int_{\partial B_{r}(x_{\varepsilon})}|f_{\varepsilon}(x)|d\sigma$ is continuous with respect to $r$.
By mean value theorem of integrals, there exists $d_\varepsilon\in (\bar d,2\bar d)$ such that
\begin{flalign}\label{bA.14}
\int^{2\bar d}_{\bar d}\int_{\partial B_{r}(x_{\varepsilon})}|f_{\varepsilon}(x)|d\sigma dr=
d_\varepsilon \int_{\partial B_{r}(x_{\varepsilon})}|f_{\varepsilon}(x)|d\sigma.
\end{flalign}
Then \eqref{aA.14} and \eqref{bA.14} imply \eqref{A.14}.
\end{proof}

\begin{Lem}\label{lem2.4}
For any $u_1,u_2,u_3,u_4\in H_{\varepsilon}$ and $0<\lambda \leq 2$,  then
\begin{flalign}\label{aa5}
\int_{\R^3}\int_{\R^3}u_1(\xi)u_2(\xi)u_3(x)u_4(x)\cdot {|x-\xi|^{-\lambda}}d\xi dx\leq C\varepsilon^{-\lambda}\|u_1\|_{\varepsilon}\|u_2\|_{\varepsilon}\|u_3\|_{\varepsilon}\|u_4\|_{\varepsilon}.
\end{flalign}
\end{Lem}
\begin{proof}
First, by Hardy-Littlewood-Sobolev inequality in Lemma \ref{lem-A-5}, we have
\begin{flalign}\label{aa6}
\int_{\R^3}\int_{\R^3}u_1(\xi)u_2(\xi)u_3(x)u_4(x)\cdot{|x-\xi|^{-\lambda}}d\xi dx
\leq C \|u_1\cdot u_2\|_{L^{\frac{6}{6-\lambda}}(\R^3)} \cdot \|u_3\cdot u_4\|_{L^{\frac{6}{6-\lambda}}(\R^3)}.
\end{flalign}
Next, for $0<\lambda \leq 2$, by H\"older's inequality and Sobolev embedding, we get
\begin{flalign}\label{aa7}
\begin{split}
\|u_1\cdot u_2\|_{L^{\frac{6}{6-\lambda}}(\R^3)}&
\leq \|u_1\|_{L^2(\R^3)}\|u_2\|_{L^{\frac{6}{3-\lambda}}(\R^3)}\leq
\|u_1\|_{\varepsilon}\|u_2\|^{{\frac{3(2-\lambda)}{6}}}_{L^2(\R^3)}\|u_2\|^{{\frac{\lambda}{2}}}_{L^6(\R^3)}
\leq \varepsilon^{-\frac{\lambda}{2}}\|u_1\|_{\varepsilon}\|u_2\|_{\varepsilon}.
\end{split}
\end{flalign}
Similarly, for $0<\lambda \leq 2$, we have
\begin{flalign}\label{aa8}
 \|u_3\cdot u_4\|_{L^{\frac{6}{6-\lambda}}(\R^3)}\leq \varepsilon^{-\lambda/2}\|u_3\|_{\varepsilon}\|u_4\|_{\varepsilon}.
\end{flalign}
Then \eqref{aa6}, \eqref{aa7} and  \eqref{aa8} imply \eqref{aa5}.
\end{proof}
\begin{Lem}
For any $u_1,u_2,u_3,u_4\in H^1(\R^3)$, and $0<\lambda \leq 2$,  then
\begin{flalign}\label{aa9}
\int_{\R^3}\int_{\R^3}\frac{u_1(\xi)u_2(\xi)u_3(x)u_4(x)}{|x-\xi|^{\lambda}}d\xi dx\leq C\|u_1\|_{H^1(\R^3)}\|u_2\|_{H^1(\R^3)}\|u_3\|_{H^1(\R^3)}\|u_4\|_{H^1(\R^3)}.
\end{flalign}
\end{Lem}
\begin{proof}
Similar to the proof of Lemma \ref{lem2.4}, we can obtain \eqref{aa9} by Hardy-Littlewood-Sobolev inequality,
H\"older's inequality and Sobolev embedding.
\end{proof}
\begin{Lem}
\textup{(1)}  There exist two positive constants $d_1$ and $\eta$ such that, for $~j=1,2,\cdots,k$,
 \begin{flalign}\label{A.1}
  U_{a_j}(\frac{x-x_{j,\varepsilon}}{\varepsilon})=O(e^{-\eta/\varepsilon}), ~\mbox{for}~ x\in \R^3\backslash B_{d}(x_{j,\varepsilon}),~\mbox{and}~0<d<d_1.
 \end{flalign}
 \textup{(2)}  Let $\{a_1,\cdots,a_k\}\subset \R^3$  be the different nondegenerate critical  points  of $V(x)$ with $k\geq 1$, then  it holds
 \begin{flalign}\label{A.5}
  \int_{\R^3}\big(V(a_j)-V(x)\big)U_{a_j}\big(\frac{x-x_{j,\varepsilon}}{\varepsilon}\big)
u(x)dx=O\big(\varepsilon^{7/2}+\varepsilon^
   {3/2}|x_{j,\varepsilon}-a_j|^2\big)\|u\|_{\varepsilon},
\end{flalign}
and
\begin{flalign}\label{A.6}
\begin{split}
\int_{\R^3}\frac{\partial V(x)}{\partial x^i}U_{a_j}\big(\frac{x-x_{j,\varepsilon}}{\varepsilon}\big)u(x)dx=
O\big(\varepsilon^{5/2}+\varepsilon^
   {3/2}|x_{j,\varepsilon}-a_j|\big)\|u\|_{\varepsilon},
\end{split}
\end{flalign}
where $u(x)\in H_{\varepsilon}$ and $j=1,2,\cdots,k$.
\end{Lem}
\begin{proof}
First,
the exponential decay of $U_{a_j}(x)$ implies \eqref{A.1}.
Next, since $a_j$ is a nondegenerate critical point of $V(x)$, we know
\begin{flalign}\label{A-1}
V(a_j)-V(x)=-\sum^3_{i=1}\sum^3_{l=1}(x^i-a^{i}_{j})(x^l-a^{l}_{j})\frac{\partial^2 V(a_{j})}{\partial x^i\partial x^l}+o(|x-a_j|^2).
\end{flalign}
Then using \eqref{A-1} and H\"older's inequality, for any small constant $d$, we have
 \begin{flalign}\label{A.7}
   \begin{split}
    \big|\int_{B_{d}(x_{j,\varepsilon})}&\big(V(a_j)-V(x)\big)U_{a_j}\big(\frac{x-x_{j,\varepsilon}}
    {\varepsilon}\big)u(x)dx\big|\\&
    \leq C\int_{B_{d}(x_{j,\varepsilon})}|x-a_{j}|
    ^2U_{a_j}\big(\frac{x-x_{j,\varepsilon}}{\varepsilon}\big)|u(x)|dx
    \\&\leq C\big(\int_{B_{d}(x_{j,\varepsilon})}|x-a_j|^{4}U_{a_j}^2\big(\frac{x-x_{j,\varepsilon}}
    {\varepsilon}\big)dx\big)^{\frac{1}{2}}\big(\int_{B_{d}(x_{j,\varepsilon})}u^2(x)
    dx\big)^{\frac{1}{2}}\\
    &\leq C\varepsilon^{\frac{3}{2}}\big(\int_{B_{{d}/{\varepsilon}}(0)}|\varepsilon y+(x_{j,\varepsilon}-a_j)|^{4}U_{a_j}^2(y)\mathrm{d}y\big)^{\frac{1}{2}}\|u\|_{\varepsilon}\\
    &\leq C\varepsilon^{\frac{3}{2}}\big(\varepsilon^2+|x_{j,\varepsilon}-a_j|^2\big)\|u
    \|_{\varepsilon}.
   \end{split}
 \end{flalign}
Also, by \eqref{A.1}, we can deduce that
 \begin{flalign}\label{A.8}
   \big|\int_{\R^3\backslash B_{d}(x_{j,\varepsilon})}\big(V(a_j)-V(x)\big)U_{a_j}
   \big(\frac{x-x_{j,\varepsilon}}{\varepsilon}\big)u(x)
   dx\big|\leq Ce^{-\eta/\varepsilon}\|u\|_{\varepsilon}.
 \end{flalign}
Then from \eqref{A.7} and \eqref{A.8}, we get \eqref{A.5}.

 Similarly, since $a_j$ is the nondegenerate critical point of $V(x)$,  we know
\begin{flalign}\label{laa1}
\frac{\partial V(x)}{\partial x^i}=\frac{\partial V(x)}{\partial x^i}-
\frac{\partial V(a_{j})}{\partial x^i}=\sum^3_{l=1}(x^l-
a^l_{j})\frac{\partial^2 V(a_{j})}{\partial x^i\partial x^l}+o(|x-
a_{j}|),~\mbox{for}~i=1,2,3.
\end{flalign}
So similar to \eqref{A.7}, from \eqref{laa1} and H\"older's inequality, for any small fixed ${d}$, we have
\begin{flalign}\label{aa2}
\begin{split}
\big|\int_{B_{{d}}(x_{j,\varepsilon})}\frac{\partial V(x)}{\partial x^i}U_{a_j}\big(\frac{x-x_{j,\varepsilon}}{\varepsilon}\big)u(x)dx\big|&
\leq C\int_{B_{{d}}(x_{j,\varepsilon})}|x-a_{j}|
    U_{a_j}\big(\frac{x-x_{j,\varepsilon}}{\varepsilon}\big)|u(x)|dx
\\&\leq C\varepsilon^{\frac{3}{2}}\big(\varepsilon+|x_{j,\varepsilon}-a_j|\big)\|u
    \|_{\varepsilon}.
\end{split}
\end{flalign}
Also, by \eqref{A.1}, we  know
 \begin{flalign}\label{aa3}
   \big|\int_{\R^3\backslash B_{{d}}(x_{j,\varepsilon})}\frac{\partial V(x)}{\partial x^i}U_{a_j}\big(\frac{x-x_{j,\varepsilon}}{\varepsilon}\big)u(x)dx\big|\leq Ce^{-\eta/\varepsilon}\|u\|_{\varepsilon}.
 \end{flalign}
Then \eqref{aa2} and \eqref{aa3} imply \eqref{A.6}.

\end{proof}

\renewcommand{\theequation}{B.\arabic{equation}}

\setcounter{equation}{0}

\section{Regularization and some calculations}
Let $u^{(1)}_{\varepsilon}(x)$, $u^{(2)}_{\varepsilon}(x)$ be two different positive solutions concentrating at $\{a_1,\cdots, a_k\}$. Set
\begin{flalign}\label{a3.1}
\eta_{\varepsilon}(x)=\frac{u_{\varepsilon}^{(1)}(x)-u_{\varepsilon}^{(2)}(x)}
{\|u_{\varepsilon}^{(1)}-u_{\varepsilon}^{(2)}\|_{L^{\infty}(\R^3)}}.
\end{flalign}
Then we know $\|\eta_{\varepsilon}\|_{L^{\infty}(\R^3)}=1$ and
\begin{flalign}\label{3.2}
\begin{split}
-\varepsilon^2\Delta \eta_{\varepsilon}(x)+V(x)\eta_{\varepsilon}(x)=&
E_1(x)\eta_{\varepsilon}(x)+E_2(x),~~x\in \R^{3},
\end{split}
\end{flalign}
where
\begin{flalign}\label{e}
E_1(x)=
\frac{1}{8\pi \varepsilon^2}
\int_{\R^3} \frac{\big(u^{(1)}_{\varepsilon}(\xi)\big)^2}{|x-\xi|}d\xi,
~E_2(x)=\frac{u^{(2)}_{\varepsilon}(x)}{8\pi \varepsilon^2}
\int_{\R^3} \frac{u^{(1)}_{\varepsilon}(\xi)+u^{(2)}_{\varepsilon}(\xi)}{|x-\xi|}\eta_\varepsilon(\xi)d\xi.
\end{flalign}

\begin{Prop}\label{prop3.1}
For $\eta_{\varepsilon}(x)$ defined by \eqref{a3.1}, we have
\begin{flalign}\label{3.3}
\|\eta_{\varepsilon}\|_{\varepsilon}=O(\varepsilon^{3/2}).
\end{flalign}
\end{Prop}

\begin{proof}
From \eqref{3.2} we have
\begin{flalign}\label{e5}
\|\eta_{\varepsilon}\|^2_{\varepsilon}=\int_{\R^3}E_1(x)
\eta^2_{\varepsilon}(x)dx+\int_{\R^3}E_2(x)
\eta_{\varepsilon}(x)dx.
\end{flalign}
Next, by Hardy-Littlewood-Sobolev inequality, H\"older's inequality and the fact $|\eta_{\varepsilon}(x)|\leq 1$, we know
\begin{flalign}\label{e9}
\begin{split}
\big|\int_{\R^3}E_1(x)
\eta^2_{\varepsilon}(x)dx\big|
&\leq  C \varepsilon^{-2}\big(
\int_{\R^3} \big|u^{(1)}_{\varepsilon}(\xi)\big|^{12/5}d\xi\big)^{5/6}
\cdot \big(
\int_{\R^3}
\big|\eta_\varepsilon(x)\big|^{12/5}
dx\big)^{5/6}\\&
\leq   C \varepsilon^{-2}\big(
\int_{\R^3} \big|u^{(1)}_{\varepsilon}(\xi)\big|^{12/5}d\xi\big)^{5/6}
\cdot \big(
\int_{\R^3}
\big|\eta_\varepsilon(x)\big|^{2}
dx\big)^{5/6} \\&
\leq
C\varepsilon^{1/2}\|\eta_{\varepsilon}\|^{5/3}_{\varepsilon}\leq C
\varepsilon^{3}+\frac{1}{2}\|\eta_{\varepsilon}\|^{2}_{\varepsilon},
 \end{split}
\end{flalign}
and
\begin{flalign}\label{e6}
\begin{split}
\big|\int_{\R^3}&E_2(x)
\eta_{\varepsilon}(x)dx \big|
\leq
 C \varepsilon^{-2}\big(
\int_{\R^3} \big|u^{(2)}_{\varepsilon}(x)\big|^{\frac{6}{5}}dx\big)^{\frac{5}{6}}
\cdot \big(
\int_{\R^3}
\big|(u^{(1)}_{\varepsilon}(\xi)+u^{(2)}_{\varepsilon}(\xi))\big|^{\frac{6}{5}}
d\xi\big)^{\frac{5}{6}}\leq C\varepsilon^{3}.
\end{split}
\end{flalign}
Then \eqref{e5}, \eqref{e9} and \eqref{e6} imply \eqref{3.3}.
\end{proof}

\begin{Lem}For any fixed $R>0$, it holds
\begin{flalign}\label{l1}
E_1(x)=o(1)\cdot R+O(\frac{1}{R}),~ \mbox{for} ~x\in \R^3\backslash\bigcup_{j=1}^k B_{R\varepsilon }(x_{j,\varepsilon}),
\end{flalign}
and
\begin{flalign}\label{l2}
E_2(x)=O\big(e^{-\theta' R}\big),~ \mbox{for} ~x\in \R^3\backslash\bigcup_{j=1}^k B_{R\varepsilon }(x_{j,\varepsilon})~ \mbox{and some}~\theta'>0.
\end{flalign}
\end{Lem}
\begin{proof}
First, we know
$$
\{\xi,|x-\xi|\leq R \varepsilon/2\}\subset \R^3\backslash\bigcup_{j=1}^k B_{R\varepsilon/2}(x_{j,\varepsilon}), ~\mbox{for}~x\in \R^3\backslash\bigcup_{j=1}^k B_{R\varepsilon}(x_{j,\varepsilon}),
$$
and $\|u_{\varepsilon}\|_{\varepsilon}=O(\varepsilon^{3/2})$. Then by \eqref{l3}, for  $x\in \R^3\backslash\bigcup_{j=1}^k B_{R\varepsilon }(x_{j,\varepsilon})$, it holds
\begin{flalign}\label{e01}
\begin{split}
E_1(x)=&
\frac{1}{8\pi \varepsilon^2}\int_{|x-\xi|\leq R\varepsilon/2} (u^{(1)}_{\varepsilon}(\xi))^2|x-\xi|^{-1}d\xi+
\frac{1}{8\pi \varepsilon^2}\int_{|x-\xi|> R\varepsilon/2} (u^{(1)}_{\varepsilon}(\xi))^2|x-\xi|^{-1}d\xi\\=&O\big( \varepsilon^{-2}\int_{|x-\xi|\leq R\varepsilon/2} \big(w^{(1)}_{\varepsilon}(\xi)\big)^2|x-\xi|^{-1}d\xi\big)
+O\big(e^{-2\theta R}R^{2}\big)+O(\frac{1}{R}\big).
\end{split}
\end{flalign}
Also, by H\"older's inequality, we have
\begin{flalign}\label{ee01}
\begin{split}
\int_{|x-\xi|\leq R\varepsilon/2}  &\big(w^{(1)}_{\varepsilon}(\xi)\big)^2|x-\xi|^{-1}d\xi\\
=&O\big(\big(\int_{\R^3} \big(w^{(1)}_{\varepsilon}(\xi)\big)^6d\xi\big)^{1/3}\cdot \big(\int_{|x-\xi|\leq R\varepsilon/2} {|x-\xi|^{-3/2}}d\xi\big)^{2/3}\big)\\=&
R\cdot O\big(\varepsilon^{-1}\|w^{(1)}_{\varepsilon}(\xi)\|_{\varepsilon}^2\big)=o(\varepsilon^2)\cdot R.
\end{split}\end{flalign}
Then \eqref{e01} and \eqref{ee01} imply \eqref{l1}.

Next for $x\in \R^3\backslash\bigcup_{j=1}^k B_{R\varepsilon }(x_{j,\varepsilon})$, we have
\begin{flalign}\label{e19}
E_2(x)=O\big(e^{-\theta R}\big)\cdot \varepsilon^{-2}
\int_{\R^3} \big(u^{(1)}_{\varepsilon}(\xi)+u^{(2)}_{\varepsilon}(\xi)\big) |x-\xi|^{-1}\cdot
|\eta_\varepsilon(\xi)|d\xi,
\end{flalign}
and
\begin{flalign}\label{e20}
\begin{split}
\int_{\R^3}
&  \big(u^{(1)}_{\varepsilon}(\xi)+u^{(2)}_{\varepsilon}(\xi)\big) |x-\xi|^{-1}\cdot
|\eta_\varepsilon(\xi)|d\xi
\\=&\int_{|x-\xi|\leq R\varepsilon/2} \big(u^{(1)}_{\varepsilon}(\xi)+u^{(2)}_{\varepsilon}(\xi)\big) |x-\xi|^{-1}\cdot
|\eta_\varepsilon(\xi)|d\xi + O\big((R\varepsilon)^{-1}\| u^{(1)}_{\varepsilon}+u^{(2)}_{\varepsilon}\|_{\varepsilon}\|\eta_\varepsilon\|_{\varepsilon}\big)\\
=&O\big(\|(u^{(1)}_{\varepsilon}(\cdot)+u^{(2)}_{\varepsilon}(\cdot)\|_\varepsilon \big) \cdot
\big(\int_{|x-\xi|\leq R\varepsilon/2} |x-\xi|^{-2}d\xi\big)^{\frac{1}{2}}+O\big(R^{-1}\varepsilon^2\big)
=O\big( (R^{\frac{1}{2}}+R^{-1}) \varepsilon^{2}\big).
\end{split}\end{flalign}
Then \eqref{e19} and \eqref{e20} imply \eqref{l2}.
\end{proof}
\begin{Lem}
For any fixed small $d>0$, it holds
\begin{equation}\label{l9}
E_1(x)=\frac{1}{8\pi\varepsilon^2}
 \big(\int_{\R^3}U_{a_j}^2(\frac{\xi-x^{(1)}_{j,\varepsilon}}{\varepsilon}){|x-\xi|}^{-1}d\xi\big)+o(1),~ \mbox{in}~B_{d}(x^{(1)}_{j,\varepsilon}),
 \end{equation}
and

\begin{flalign}\label{l10}
E_2(x)=\frac{1}{4\pi}\cdot U_{a_j}(\frac{x-x^{(1)}_{j,\varepsilon}}{\varepsilon})
 \big(\int_{\R^3} U_{a_j}(\frac{\xi-x^{(1)}_{j,\varepsilon}}{\varepsilon})
 \eta_{\varepsilon}(\xi)|x-\xi|^{-1}d\xi\big)+o(1), ~\mbox{in}~B_{d}(x^{(1)}_{j,\varepsilon}).
\end{flalign}
\end{Lem}

\begin{proof}
For $x\in B_{d}(x^{(1)}_{j,\varepsilon})$, we have
\begin{equation}\label{l7}
\begin{split}
\big|E_1(x)-&\frac{1}{8\pi\varepsilon^2}
 \big(\int_{\R^3}U_{a_j}^2(\frac{\xi-x^{(1)}_{j,\varepsilon}}{\varepsilon}){|x-\xi|^{-1}}d\xi\big)\big|
\\=&O\big(\varepsilon^{-2}
\int_{\R^3} \big|w^{(1)}_{\varepsilon}(\xi)\big|\cdot\big(u^{(1)}_{\varepsilon}(\xi)+
U_{a_j} (\frac{\xi-x^{(1)}_{j,\varepsilon}}{\varepsilon})\big) |x-\xi|^{-1}d\xi\big)+ O(e^{-\eta/\varepsilon})
\\=&
O\big(\varepsilon^{-2}
\int_{|x-\xi|\leq C} \big|w^{(1)}_{\varepsilon}(\xi)\big|\cdot\big(u^{(1)}_{\varepsilon}(\xi)+
U_{a_j} (\frac{\xi-x^{(1)}_{j,\varepsilon}}{\varepsilon})\big) |x-\xi|^{-1}d\xi\big)\\&
+O\big(\varepsilon^{-2}\| w^{(1)}_{\varepsilon}(\cdot)\|_{\varepsilon}\cdot \|u^{(1)}_{\varepsilon}(\cdot)+
U_{a_j} (\frac{\cdot-x^{(1)}_{j,\varepsilon}}{\varepsilon})\|_{\varepsilon}\big)
+ O(e^{-\eta/\varepsilon}),
\end{split}
\end{equation}
where $C$ is a fixed constant.

On the other hand, by H\"older's inequality, we know
\begin{equation}\label{l8}
\begin{split}
\int_{|x-\xi|\leq C} &\big|w^{(1)}_{\varepsilon}(\xi)\big|\cdot\big(u^{(1)}_{\varepsilon}(\xi)+
U_{a_j} (\frac{\xi-x^{(1)}_{j,\varepsilon}}{\varepsilon})\big) |x-\xi|^{-1}d\xi
\\=&O\big(\|w^{(1)}_{\varepsilon}(\cdot)\|_{L^6(\R^3)}\cdot
\|u^{(1)}_{\varepsilon}(\cdot)+
U_{a_j} (\frac{\cdot-x^{(1)}_{j,\varepsilon}}{\varepsilon})\|_{L^2(\R^3)}
\big(\int_{|x-\xi|\leq C}  |x-\xi|^{-3}d\xi\big)^{1/3}\big)\\
=&
O\big(\varepsilon^{-1}\| w^{(1)}_{\varepsilon}(\cdot)\|_{\varepsilon}\cdot \|u^{(1)}_{\varepsilon}(\cdot)+
U_{a_j} (\frac{\cdot-x^{(1)}_{j,\varepsilon}}{\varepsilon})\|_{\varepsilon}\big)=o\big(\varepsilon^2\big).
\end{split}
\end{equation}
Then \eqref{l7} and \eqref{l8} imply \eqref{l9}.
Similar to the estimates of \eqref{l9}, combining  Proposition \ref{Prop2.7}
and Proposition \ref{prop3--1}, we deduce \eqref{l10}.
\end{proof}

\renewcommand{\theequation}{C.\arabic{equation}}

\setcounter{equation}{0}

\section{Estimates of the term $w_\varepsilon$}
\setcounter{equation}{0}

For convenience, we define the following notations:
\begin{flalign}\label{lm2}
R_\varepsilon(x)=\sum^k_{l=1}U_{a_l} (\frac{x-x_{l,\varepsilon}}{\varepsilon}),~R^{(m)}_\varepsilon(x)=\sum^k_{l=1}U_{a_l} (\frac{x-x^{(m)}_{l,\varepsilon}}{\varepsilon}), ~\mbox{for}~m=1,2,
\end{flalign}
and
\begin{flalign}\label{lm3}
W_{j,\varepsilon}(x)=\sum^k_{l=1,l\neq j}U_{a_l} (\frac{x-x_{l,\varepsilon}}{\varepsilon}),~W^{(m)}_{j,\varepsilon}(x)=\sum^k_{l=1,l\neq j}U_{a_l} (\frac{x-x^{(m)}_{l,\varepsilon}}{\varepsilon}), ~\mbox{for}~m=1,2.
\end{flalign}

Let  $M_{\varepsilon}(x,w_{\varepsilon}(x))$ as follows:
\begin{flalign}\label{B--1}
\begin{split}
M_{\varepsilon}(x,w_{\varepsilon}(x)):=&-\varepsilon^2\Delta w_{\varepsilon}(x)+G(x,w_{\varepsilon}(x)),
\end{split}
\end{flalign}
where
\begin{flalign}\label{B--2}
\begin{split}
G(x,w_{\varepsilon}(x))=&V(x)w_{\varepsilon}(x)-\frac{1}{8\pi \varepsilon^2}
\big(\int_{\R^3}\frac{\big(R_\varepsilon(\xi)\big)^2}{|x-\xi|}d\xi\big)w_{\varepsilon}(x)
+
\frac{1}{4\pi \varepsilon^2}
 \big(\int_{\R^3}\frac{R_\varepsilon(\xi) w_{\varepsilon}(\xi)}{|x-\xi|}d\xi\big)R_\varepsilon(x).
 \end{split}
\end{flalign}
Let $u_\varepsilon(x)=R_{\varepsilon}(x)+w_\varepsilon(x)$ be the solution of \eqref{1.2}, then
\begin{flalign}\label{B---1}
\begin{split}
M_{\varepsilon}(x,w_{\varepsilon}(x))=N\big(x,w_{\varepsilon}(x)\big)+ l_{\varepsilon}(x),
\end{split}
\end{flalign}
where
\begin{flalign} \label{B--3}
\begin{split}
N\big(x,w_{\varepsilon}(x)\big)=&\frac{1}{8\pi \varepsilon^2}
 \big(\int_{\R^3}\frac{w_{\varepsilon}^2(\xi)}{|x-\xi|}d\xi\big)
\big(R_\varepsilon(x)+ w_{\varepsilon}(x)\big)
+\frac{w_{\varepsilon}(x)}{4\pi \varepsilon^{2}}\int_{\R^{3}}\frac{R_\varepsilon(\xi)w_{\xi}(\xi)}{|x-\xi|}d\xi ,
 \end{split}
\end{flalign}
and
 \begin{flalign} \label{B--4}
 \begin{split}
 l_{\varepsilon}(x)=&
      \frac{W_{j,\varepsilon}(x)}{8\pi \varepsilon^{2}}\int_{\R^{3}}
      \frac{W_{j,\varepsilon}(\xi)U_{a_{j}}(\frac{\xi-x_{j,\varepsilon}}{\varepsilon}) }{|x-\xi|}d\xi
      +\sum_{j=1}^k\big(V(a_j) - V(x)\big)U_{a_j}
      (\frac{x-x_{j,\varepsilon}}{\varepsilon}).
 \end{split}
 \end{flalign}

\begin{Prop}\label{lem-A.1}
Let $u_\varepsilon(x)=R_{\varepsilon}(x)+w_\varepsilon(x)$ be the solution of \eqref{1.2}, then there exists a constant $\bar{\rho}>0$ independent of $\varepsilon$ such that
 \begin{flalign}\label{B.2}
\int_{\R^3} M_{\varepsilon}\big(x,w_\varepsilon(x)\big)w_\varepsilon(x)dx \geq \bar{\rho} \|w_\varepsilon\|^2_{\varepsilon}.
 \end{flalign}
\end{Prop}
\begin{proof}
Similar to the proof of Proposition 3.1 in \cite{Cao2}, we can prove \eqref{B.2}
by the contradiction argument and blow-up analysis. For the more details, one can refer to \cite{Cao1,Cao2}.
\end{proof}

\begin{Prop}\label{Prop2.6}
Suppose that $u_\varepsilon(x)=R_{\varepsilon}(x)+w_\varepsilon(x)$ is a positive solution of \eqref{1.2}
 and $\{a_1,\cdots,a_k\}\subset \R^3$  are the different nondegenerate critical  points  of $V(x)$ with $k\geq 1$, then it holds
  \begin{flalign}\label{B.3}
    \|w_{\varepsilon}\|_{\varepsilon}=O(\varepsilon^{7/2})+O\big(\varepsilon^
    {3/2}\max_{j=1,\cdots,k}|x_{j,\varepsilon}-a_j|^{2}\big).
  \end{flalign}
\end{Prop}
\begin{proof}
First, from  Proposition \ref{lem-A.1}, we know
 \begin{flalign}\label{B.5}
 \begin{split}
   \|w_{\varepsilon}\|^2_{\varepsilon} & \leq C\int_{\R^3}N\big(x,w_{\varepsilon}(x)\big)w_{\varepsilon}(x)dx+C
   \int_{\R^3}l_{\varepsilon}(x)w_{\varepsilon}(x)dx.
\end{split}
 \end{flalign}
Next, using \eqref{2--2} and \eqref{aa5}, we deduce
 \begin{flalign}\label{B.14}
 \begin{split}
   \int_{\R^3}N\big(x,w_{\varepsilon}(x)\big)w_{\varepsilon}(x)dx&=\frac{1}{8\pi \varepsilon^2}
\int_{\R^3}\int_{\R^3}\frac{w_{\varepsilon}^2(\xi)}{|x-\xi|}
\big(R_\varepsilon(x)+ w_{\varepsilon}(x)\big)w_{\varepsilon}(x)dxd\xi\\&\quad
+\frac{1}{4\pi \varepsilon^{2}}\int_{\R^3}\int_{\R^{3}}\frac{R_\varepsilon(\xi)w_{\xi}(\xi)}{|x-\xi|} w^2_{\varepsilon}(x)dxd\xi\\&
=O\big(\varepsilon^{-3}\|w_{\varepsilon}\|^3_\varepsilon\cdot \|w_{\varepsilon}+R_\varepsilon\|_\varepsilon\big)=o(1)\|w_{\varepsilon}\|^2_\varepsilon.
\end{split}
 \end{flalign}
Also from \eqref{A.1} and \eqref{A-1}, we have
  \begin{flalign}\label{B.14l}
  \begin{split}
  \int_{\R^3}l_{\varepsilon}(x)w_{\varepsilon}(x)dx=&\sum^k_{j=1}\int_{\R^3}(V(a_j)-V(x)) U_{a_j}(\frac{x-x_{j,\varepsilon}}{\varepsilon})w_\varepsilon(x)d\xi dx\\&
  -\frac{1}{8\pi \varepsilon^{2}}\int_{\R^{3}}\int_{\R^{3}}
      W_{j,\varepsilon}(\xi)U_{a_{j}}(\frac{\xi-x_{j,\varepsilon}}{\varepsilon})W_{j,\varepsilon}(x) w_\varepsilon(x){|x-\xi|}^{-1}d\xi\\=&
O\big(\varepsilon^{3/2}\|w_\varepsilon\|_{\varepsilon}(\varepsilon^2
+\max_{j=1,\cdots,k}|x_{j,\varepsilon}-a_j|^{2})+e^{-\eta/\varepsilon}\big).
 \end{split}
 \end{flalign}
Then \eqref{B.5}, \eqref{B.14}  and \eqref{B.14l} imply \eqref{B.3}.
\end{proof}

\begin{Prop}\label{prop2.5}
Let $u_\varepsilon(x)$ be a positive solution of \eqref{1.2} as in Proposition \ref{Prop2.6},
then it holds
\begin{flalign}\label{2.45}
\|w_{\varepsilon}\|_{\varepsilon}=O(\varepsilon^{7/2}).
\end{flalign}
\end{Prop}

\begin{proof}
It follows from the results of Proposition \ref{Prop2.7} and Proposition \ref{Prop2.6} directly.
\end{proof}

\renewcommand{\theequation}{D.\arabic{equation}}

\setcounter{equation}{0}
\section{The Estimates of  $A_{1,1}$ and $A_{1,2}$ in \eqref{c1}}

\begin{Lem}\label{lem3-1}
It holds
\begin{flalign}\label{c-2}
\begin{split}
A_{1,1}=&\frac{1}{4\pi \varepsilon^2}\int_{\R^3}\int_{\R^3}U^2_{a_j}(\frac{x-x_{j,\varepsilon}}{\varepsilon})
U_{a_j}(\frac{\xi-x_{j,\varepsilon}}{\varepsilon})w_\varepsilon(\xi)
\frac{x^i-\xi^i}{|x-\xi|^3}d\xi dx\\
&+\frac{1}{8\pi \varepsilon^2}\sum^k_{l=1,l\neq j}\int_{\R^3}\int_{\R^3}
U^2_{a_j}(\frac{x-x_{j,\varepsilon}}{\varepsilon})U^2_{a_l}(\frac{\xi-x_{l,\varepsilon}}{\varepsilon})
\frac{x^i-\xi^i}{|x-\xi|^3}d\xi dx+O(\varepsilon^6).
\end{split}
\end{flalign}
\end{Lem}
\begin{proof}

First, $A_{1,1}$ can be written as follows:
\begin{flalign}\label{c3}
A_{1,1}=A_{1,1,1}+A_{1,1,2}+A_{1,1,3}+A_{1,1,4}+A_{1,1,5},
\end{flalign}
where
\begin{flalign*}
&A_{1,1,1}=\frac{1}{8\pi \varepsilon^2}\int_{ B_{d}(x_{j,\varepsilon})}\int_{\R^3}U^2_{a_j}(\frac{x-x_{j,\varepsilon}}{\varepsilon})
U^2_{a_j}(\frac{\xi-x_{j,\varepsilon}}{\varepsilon})\frac{x^i-\xi^i}{|x-\xi|^3}d\xi dx,&
\end{flalign*}

\begin{flalign*}
&A_{1,1,2}=\frac{1}{4\pi \varepsilon^2}\int_{ B_{d}(x_{j,\varepsilon})}\int_{\R^3}U^2_{a_j}(\frac{x-x_{j,\varepsilon}}{\varepsilon})
U_{a_j}(\frac{\xi-x_{j,\varepsilon}}{\varepsilon})w_\varepsilon(\xi)
\frac{x^i-\xi^i}{|x-\xi|^3}d\xi dx,&
\end{flalign*}

\begin{flalign*}
&A_{1,1,3}=\frac{1}{8\pi\varepsilon^2}\int_{ B_{d}(x_{j,\varepsilon})}\int_{\R^3}
U^2_{a_j}(\frac{x-x_{j,\varepsilon}}{\varepsilon}) w^2_{\varepsilon}(\xi)
\frac{x^i-\xi^i}{|x-\xi|^3}d\xi dx,&
\end{flalign*}

\begin{flalign*}
&A_{1,1,4}= \frac{1}{4\pi \varepsilon^2}\int_{ B_{d}(x_{j,\varepsilon})}\int_{\R^3}
U^2_{a_j}(\frac{x-x_{j,\varepsilon}}{\varepsilon})
W_{j,\varepsilon}(\xi)\big(U_{a_j}(\frac{\xi-x_{j,\varepsilon}}{\varepsilon})
+w_\varepsilon(\xi)\big)
\frac{x^i-\xi^i}{|x-\xi|^3}d\xi dx,&
\end{flalign*}
\begin{flalign*}
&A_{1,1,5}= \frac{1}{8\pi \varepsilon^2}\int_{ B_{d}(x_{j,\varepsilon})}\int_{\R^3}
U^2_{a_j}(\frac{x-x_{j,\varepsilon}}{\varepsilon})\big(W_{j,\varepsilon}(\xi)\big)^2
\frac{x^i-\xi^i}{|x-\xi|^3}d\xi dx.&
\end{flalign*}
Now by symmetry and \eqref{A.1}, we have
\begin{flalign}\label{c4}
\begin{split}
A_{1,1,1}&=-\frac{1}{8\pi \varepsilon^2}\int_{ \R^3 \backslash B_{d}(x_{j,\varepsilon})}\int_{\R^3}U^2_{a_j}(\frac{x-x_{j,\varepsilon}}{\varepsilon})
U^2_{a_j}(\frac{\xi-x_{j,\varepsilon}}{\varepsilon})\frac{x^i-\xi^i}{|x-\xi|^3}d\xi dx
=O(e^{-\eta/\varepsilon}),
\end{split}
\end{flalign}
and
\begin{flalign}\label{c5}
\begin{split}
A_{1,1,2}&=\frac{1}{4\pi \varepsilon^2}\int_{\R^3}\int_{\R^3}U^2_{a_j}(\frac{x-x_{j,\varepsilon}}{\varepsilon})
U_{a_j}(\frac{\xi-x_{j,\varepsilon}}{\varepsilon})w_\varepsilon(\xi)
\frac{x^i-\xi^i}{|x-\xi|^3}d\xi dx+O(e^{-\eta/\varepsilon}).
\end{split}
\end{flalign}
Next, by \eqref{aa5} and \eqref{2.45}, we get
\begin{flalign}\label{c6}
\begin{split}
A_{1,1,3}&=O\big(\varepsilon^{-4}\|U_{a_j}(\frac{\cdot-x_{j,\varepsilon}}{\varepsilon})
\|^2_{\varepsilon}\cdot\|w_\varepsilon\|^2_{\varepsilon}\big)=O\big(\varepsilon^{6} \big).
\end{split}
\end{flalign}
Also, \eqref{l4} and \eqref{A.1} imply
\begin{equation}\label{ll3}
W_{j,\varepsilon}(x)\big(U_{a_j}(\frac{x-x_{j,\varepsilon}}{\varepsilon})
+w_\varepsilon(x)\big)=O(e^{-\eta/\varepsilon}),~\mbox{for}~x\in \R^3.
\end{equation}
This means
\begin{flalign}\label{c7}
\begin{split}
A_{1,1,4}&=O(e^{-\eta/\varepsilon}).
\end{split}
\end{flalign}
Also, from \eqref{A.1}, we can deduce
\begin{flalign}\label{c9}
\begin{split}
A_{1,1,5}=&\frac{1}{8\pi \varepsilon^2}\sum^k_{l=1,l\neq j}\int_{\R^3}\int_{\R^3}
U^2_{a_j}(\frac{x-x_{j,\varepsilon}}{\varepsilon})U^2_{a_l}(\frac{\xi-x_{l,\varepsilon}}{\varepsilon})
\frac{x^i-\xi^i}{|x-\xi|^3}d\xi dx+O\big(e^{-\eta/\varepsilon}\big).
\end{split}
\end{flalign}
Then \eqref{c3}, \eqref{c4}, \eqref{c5}, \eqref{c6}, \eqref{c7} and \eqref{c9} imply \eqref{c-2}.
\end{proof}

\begin{Lem}
It holds
\begin{flalign}\label{c-3}
\begin{split}
A_{1,2}=&-\frac{1}{4\pi \varepsilon^2}\int_{\R^3}\int_{\R^3}U^2_{a_j}(\frac{x-x_{j,\varepsilon}}{\varepsilon})
U_{a_j}(\frac{\xi-x_{j,\varepsilon}}{\varepsilon})w_\varepsilon(\xi)
\frac{x^i-\xi^i}{|x-\xi|^3}d\xi dx+O\big(\varepsilon^{6}\big).
\end{split}
\end{flalign}
\end{Lem}
\begin{proof}
First, $A_{1,2}$ can be written as follows:
\begin{flalign}\label{c12}
A_{1,2}=A_{1,2,1}+A_{1,2,2}+A_{1,2,3}+A_{1,2,4}+A_{1,2,5},
\end{flalign}
where
\begin{flalign*}
&A_{1,2,1}=\frac{1}{4\pi \varepsilon^2}\int_{ B_{d}(x_{j,\varepsilon})}\int_{\R^3}U_{a_j}(\frac{x-x_{j,\varepsilon}}{\varepsilon})w_\varepsilon(x)
U^2_{a_j}(\frac{\xi-x_{j,\varepsilon}}{\varepsilon})\frac{x^i-\xi^i}{|x-\xi|^3}d\xi dx,&
\end{flalign*}
\begin{flalign*}
&A_{1,2,2}=\frac{1}{2\pi \varepsilon^2}\int_{ B_{d}(x_{j,\varepsilon})}\int_{\R^3}U_{a_j}(\frac{x-x_{j,\varepsilon}}{\varepsilon})w_\varepsilon(x)
U_{a_j}(\frac{\xi-x_{j,\varepsilon}}{\varepsilon})w_\varepsilon(\xi)
\frac{x^i-\xi^i}{|x-\xi|^3}d\xi dx,&
\end{flalign*}

\begin{flalign*}
&A_{1,2,3}=\frac{1}{4\pi \varepsilon^2}\int_{ B_{d}(x_{j,\varepsilon})}\int_{\R^3}
U_{a_j}(\frac{x-x_{j,\varepsilon}}{\varepsilon})w_\varepsilon(x) w^2_{\varepsilon}(\xi)
\frac{x^i-\xi^i}{|x-\xi|^3}d\xi dx,&
\end{flalign*}

\begin{flalign*}
&A_{1,2,4}=\frac{1}{2\pi \varepsilon^2}\int_{ B_{d}(x_{j,\varepsilon})}\int_{\R^3}
U_{a_j}(\frac{x-x_{j,\varepsilon}}{\varepsilon})w_\varepsilon(x)
W_{j,\varepsilon}(\xi)\big(U_{a_j}(\frac{\xi-x_{j,\varepsilon}}{\varepsilon})
+w_\varepsilon(\xi)\big)
\frac{x^i-\xi^i}{|x-\xi|^3}d\xi dx,&
\end{flalign*}
\begin{flalign*}
&
A_{1,2,5}=
\frac{1}{4\pi \varepsilon^2}\int_{ B_{d}(x_{j,\varepsilon})}
\int_{\R^3}
U_{a_j}(\frac{x-x_{j,\varepsilon}}{\varepsilon})w_\varepsilon(x)
\big(W_{j,\varepsilon}(\xi)\big)^2
\frac{x^i-\xi^i}{|x-\xi|^3}d\xi dx.
&
\end{flalign*}
Now similar to the calculations of \eqref{c4} and \eqref{c5},  by symmetry and \eqref{A.1}, we know
\begin{flalign}\label{c13}
\begin{split}
A_{1,2,1}&
=-\frac{1}{4\pi \varepsilon^2}\int_{\R^3}\int_{\R^3}U^2_{a_j}(\frac{x-x_{j,\varepsilon}}{\varepsilon})
U_{a_j}(\frac{\xi-x_{j,\varepsilon}}{\varepsilon})w_\varepsilon(\xi)
\frac{x^i-\xi^i}{|x-\xi|^3}d\xi dx+O(e^{-\eta/\varepsilon}).
\end{split}
\end{flalign}
Next, by \eqref{aa5} and \eqref{2.45}, we get
\begin{flalign}\label{c14}
\begin{split}
A_{1,2,2}=O\big(\varepsilon^{-4}\|U_{a_j}(\frac{\cdot-x_{j,\varepsilon}}{\varepsilon})
\|^2_{\varepsilon}\cdot\|w_\varepsilon\|^2_{\varepsilon}\big)=O\big(\varepsilon^{6} \big),
\end{split}
\end{flalign}
and
\begin{flalign}\label{c15}
A_{1,2,3}=O\big(\varepsilon^{-4}\|U_{a_j}(\frac{\cdot-x_{j,\varepsilon}}{\varepsilon})
\|_{\varepsilon}\|w_\varepsilon\|^3_{\varepsilon}\big)=O\big(\varepsilon^{8} \big).
\end{flalign}
Also, similar to \eqref{c7}, we have
\begin{flalign}\label{c17}
A_{1,2,4}=O(e^{-\eta/\varepsilon}).
\end{flalign}
On the other hand, for $l\neq j$ and fixed small $d$, from \eqref{A.1} and \eqref{2.45}, we have
\begin{equation}\label{cll}
\begin{split}
\int_{ B_{d}(x_{j,\varepsilon})}&
\int_{\R^3}
U_{a_j}(\frac{x-x_{j,\varepsilon}}{\varepsilon})w_\varepsilon(x)U^2_{a_l}(\frac{\xi-x_{l,\varepsilon}}{\varepsilon})
\frac{x^i-\xi^i}{|x-\xi|^3}d\xi dx\\
=&
\int_{ B_{d}(x_{j,\varepsilon})}
\int_{\R^3\backslash B_{2d}(x_{j,\varepsilon})}
U_{a_j}(\frac{x-x_{j,\varepsilon}}{\varepsilon})w_\varepsilon(x)U^2_{a_l}(\frac{\xi-x_{l,\varepsilon}}{\varepsilon})
\frac{x^i-\xi^i}{|x-\xi|^3}d\xi dx+O\big(e^{-\eta/\varepsilon}\big)\\
=&
O\big(\int_{ B_{d}(x_{j,\varepsilon})}
\int_{\R^3\backslash B_{2d}(x_{j,\varepsilon})}
U_{a_j}(\frac{x-x_{j,\varepsilon}}{\varepsilon})w_\varepsilon(x)U^2_{a_l}(\frac{\xi-x_{l,\varepsilon}}{\varepsilon})d\xi dx\big)+O\big(e^{-\eta/\varepsilon}\big)\\=&
O\big(\|
U_{a_j}(\frac{\cdot-x_{j,\varepsilon}}{\varepsilon})\|_{\varepsilon}\cdot
\|w_\varepsilon\|_{\varepsilon}\cdot\|U_{a_l}(\frac{\cdot-x_{l,\varepsilon}}{\varepsilon})\|^2_{\varepsilon}\big)+O\big(e^{-\eta/\varepsilon}\big)
=O\big(\varepsilon^{8}\big).
\end{split}
\end{equation}
Then \eqref{A.1} and \eqref{cll} imply
\begin{flalign}\label{c18}
\begin{split}
A_{1,2,5}=&O\big(\varepsilon^{6}\big)+O\big(e^{-\eta/\varepsilon}\big)=O\big(\varepsilon^{6}\big).
\end{split}
\end{flalign}
Then \eqref{c12}, \eqref{c13}, \eqref{c14}, \eqref{c15}, \eqref{c17}  and \eqref{c18} imply
\eqref{c-3}.
\end{proof}

\begin{Lem}\label{lem3.4}
For $l\neq j$, it holds
\begin{flalign}\label{c40}
\begin{split}
\int_{\R^3}& \int_{\R^3}
U^2_{a_j}(\frac{x-x_{j,\varepsilon}}{\varepsilon})U^2_{a_l}(\frac{\xi-x_{l,\varepsilon}}{\varepsilon})
\frac{x^i-\xi^i}{|x-\xi|^3}d\xi dx\\
& =\varepsilon^3(a^{i}_{j}-a^{i}_{l})\int_{\R^3}\int_{\R^3}
U^2_{a_j}(x)
U^2_{a_l}(\xi+\frac{x_{j,\varepsilon}-x_{l,\varepsilon}}{\varepsilon})|x-\xi|^{-3}d\xi dx+o(\varepsilon^6).
\end{split}\end{flalign}
\end{Lem}
\begin{proof}
First, we have
\begin{flalign}\label{c31}
\int_{\R^3}&\int_{\R^3}
U^2_{a_j}(\frac{x-x_{j,\varepsilon}}{\varepsilon})U^2_{a_l}(\frac{\xi-x_{l,\varepsilon}}{\varepsilon})
\frac{x^i-\xi^i}{|x-\xi|^3}d\xi dx
=B_{1}+B_{2},
\end{flalign}
where
$$
B_{1}={\varepsilon^4}\int_{\R^3}\int_{\R^3}
U^2_{a_j}(x)
U^2_{a_l}(\xi+\frac{x_{j,\varepsilon}-x_{l,\varepsilon}}{\varepsilon})\frac{x^i}{|x-\xi|^3}d\xi dx,
$$
and
$$
B_{2}=-{\varepsilon^4}\int_{\R^3}\int_{\R^3}
U^2_{a_j}(x)
U^2_{a_l}(\xi+\frac{x_{j,\varepsilon}-x_{l,\varepsilon}}{\varepsilon})\frac{\xi^i}{|x-\xi|^3}d\xi dx.
$$
Then for small fixed $d>0$, we know
\begin{flalign}\label{c32}
\begin{split}
B_{1}=&{\varepsilon^4}\int_{B_{d/\varepsilon}(0)}\int_{\R^3\backslash B_{(2d)/\varepsilon}(0)}
U^2_{a_j}(x)
U^2_{a_l}(\xi+\frac{x_{j,\varepsilon}-x_{l,\varepsilon}}{\varepsilon})\frac{x^i}{|x-\xi|^3}d\xi dx+O\big(e^{-\eta/\varepsilon}\big)\\=
&O\big(\varepsilon^7 \big(\int_{\R^3}
U^2_{a_j}(x)|x|dx\big)\cdot\big(\int_{\R^3}
U^2_{a_l}(\xi+\frac{x_{j,\varepsilon}-x_{l,\varepsilon}}{\varepsilon})d\xi\big) \big)+O\big(e^{-\eta/\varepsilon}\big)=O\big(\varepsilon^7\big).
\end{split}
\end{flalign}
Also we have
\begin{flalign}\label{c33}
\begin{split}
B_{2}=&-{\varepsilon^4}\int_{\R^3}\int_{\R^3}
U^2_{a_j}(x)
U^2_{a_l}(\xi+\frac{x_{j,\varepsilon}-x_{l,\varepsilon}}{\varepsilon})
\frac{\xi^i+\frac{x^i_{j,\varepsilon}-x^i_{l,\varepsilon}}{\varepsilon}}{|x-\xi|^3}d\xi dx\\&
+ (x^i_{j,\varepsilon}-x^i_{l,\varepsilon})\varepsilon^3\int_{\R^3}\int_{\R^3}
U^2_{a_j}(x)
U^2_{a_l}(\xi+\frac{x_{j,\varepsilon}-x_{l,\varepsilon}}{\varepsilon})|x-\xi|^{-3}d\xi dx.
\end{split}
\end{flalign}
Next, similar to \eqref{c32}, we deduce
\begin{flalign}\label{c32l}
\begin{split}
\int_{\R^3}&\int_{\R^3}
 U^2_{a_j}(x)
U^2_{a_l}(\xi+\frac{x_{j,\varepsilon}-x_{l,\varepsilon}}{\varepsilon})
|x-\xi|^{-3}d\xi dx=O\big(\varepsilon^3\big),
\end{split}
\end{flalign}
and
\begin{flalign}\label{c33l}
\begin{split}
\int_{\R^3}&\int_{\R^3}
 U^2_{a_j}(x)
U^2_{a_l}(\xi+\frac{x_{j,\varepsilon}-x_{l,\varepsilon}}{\varepsilon})
\frac{\xi^i+\frac{x^i_{j,\varepsilon}-x^i_{l,\varepsilon}}{\varepsilon}}{|x-\xi|^3}d\xi dx=O\big(\varepsilon^3\big).
\end{split}
\end{flalign}
Then using \eqref{2--2}, \eqref{c33}, \eqref{c32l} and \eqref{c33l}, we obtain
\begin{flalign}\label{c33ll}
\begin{split}
B_{2}=& \varepsilon^3(a^i_{j}-a^i_{l}) \int_{\R^3}\int_{\R^3}
U^2_{a_j}(x)
U^2_{a_l}(\xi+\frac{x_{j,\varepsilon}-x_{l,\varepsilon}}{\varepsilon})|x-\xi|^{-3}d\xi dx+o\big(\varepsilon^6\big).
\end{split}
\end{flalign}
Then \eqref{c31}, \eqref{c32} and \eqref{c33ll} imply \eqref{c40}.
\end{proof}

\renewcommand{\theequation}{E.\arabic{equation}}

\setcounter{equation}{0}
\section{The estimates of $F_{1,1}$, $F_{1,2}$, $F_{2,1}$ and $F_{2,3}$ in  \eqref{lll1} and \eqref{lll2}}
\setcounter{equation}{0}

\begin{Lem}\label{ld1}
It holds
\begin{flalign}\label{f1}
\begin{split}
F_{1,1}&
=G_1+\frac{1}{8\pi \varepsilon^2}
\int_{\R^3}\int_{\R^3}
U_{a_j}^2(\frac{x-x^{(1)}_{j,\varepsilon}}{\varepsilon})
\big(w^{(1)}_{\varepsilon}(\xi)+w^{(2)}_{\varepsilon}(\xi)\big)\eta_{\varepsilon}(\xi)
\frac{x^i-\xi^i}{|x-\xi|^3}d\xi dx+o\big(\varepsilon^4\big),
\end{split}
\end{flalign}
where
\begin{flalign*}
G_1=\frac{1}{8\pi \varepsilon^2}
\int_{\R^3}\int_{\R^3}
U_{a_j}^2(\frac{x-x^{(1)}_{j,\varepsilon}}{\varepsilon})
\big(U_{a_j}(\frac{\xi-x^{(1)}_{j,\varepsilon}}{\varepsilon})+
U_{a_j}(\frac{\xi-x^{(2)}_{j,\varepsilon}}{\varepsilon})
\big)\eta_{\varepsilon}(\xi)
\frac{x^i-\xi^i}{|x-\xi|^3}d\xi dx.
\end{flalign*}
\end{Lem}
\begin{proof}
 $F_{1,1}$ can be written as
\begin{flalign}\label{f2}
\begin{split}
F_{1,1}&=F_{1,1,1}+F_{1,1,2}+F_{1,1,3},
\end{split}
\end{flalign}
where
\begin{flalign*}
&F_{1,1,1}=\frac{1}{8\pi \varepsilon^2}
\int_{ B_{\delta}(x_{j,\varepsilon}^{(1)})}\int_{\R^3}
U_{a_j}^2(\frac{x-x^{(1)}_{j,\varepsilon}}{\varepsilon})
\big(U_{a_j}(\frac{\xi-x^{(1)}_{j,\varepsilon}}{\varepsilon})+
U_{a_j}(\frac{\xi-x^{(2)}_{j,\varepsilon}}{\varepsilon})
\big)\eta_{\varepsilon}(\xi)
\frac{x^i-\xi^i}{|x-\xi|^3}d\xi dx,&
\end{flalign*}
\begin{flalign*}
&F_{1,1,2}=\frac{1}{8\pi \varepsilon^2}
\int_{ B_{\delta}(x_{j,\varepsilon}^{(1)})}\int_{\R^3}
U_{a_j}^2(\frac{x-x^{(1)}_{j,\varepsilon}}{\varepsilon})
\big(w^{(1)}_{\varepsilon}(\xi)+w^{(2)}_{\varepsilon}(\xi)\big)\eta_{\varepsilon}(\xi)
\frac{x^i-\xi^i}{|x-\xi|^3}d\xi dx,&
\end{flalign*}
\begin{flalign*}
&F_{1,1,3}=\frac{1}{8\pi \varepsilon^2}
\int_{ B_{\delta}(x_{j,\varepsilon}^{(1)})}\int_{\R^3}
U_{a_j}^2(\frac{x-x^{(1)}_{j,\varepsilon}}{\varepsilon})
\big(
W^{(1)}_{j,\varepsilon}(\xi)+W^{(2)}_{j,\varepsilon}(\xi)\big)\eta_{\varepsilon}(\xi)
\frac{x^i-\xi^i}{|x-\xi|^3}d\xi dx.&
\end{flalign*}
Now, by \eqref{A.1}, we get
\begin{flalign}\label{f3}
\begin{split}
F_{1,1,1}& =G_1+O\big(e^{-\eta/\varepsilon}\big),
\end{split}
\end{flalign}
and
\begin{flalign}\label{f4}
\begin{split}
F_{1,1,2}&=\frac{1}{8\pi \varepsilon^2}
\int_{\R^3}\int_{\R^3}
U_{a_j}^2(\frac{x-x^{(1)}_{j,\varepsilon}}{\varepsilon})
\big(w^{(1)}_{\varepsilon}(\xi)+w^{(2)}_{\varepsilon}(\xi)\big)\eta_{\varepsilon}(\xi)
\frac{x^i-\xi^i}{|x-\xi|^3}d\xi dx+O\big(e^{-\eta/\varepsilon}\big).
\end{split}
\end{flalign}
Next, using Proposition \ref{l3-2}, we can calculate that, for $l\neq j$,
\begin{flalign}\label{f5}
\begin{split}
\frac{1}{8\pi \varepsilon^2}&
\int_{ B_{\delta}(x_{j,\varepsilon}^{(1)})}\int_{\R^3}
U_{a_j}^2(\frac{x-x^{(1)}_{j,\varepsilon}}{\varepsilon})
U_{a_l}(\frac{\xi-x^{(1)}_{l,\varepsilon}}{\varepsilon})\eta_{\varepsilon}(\xi)
\frac{x^i-\xi^i}{|x-\xi|^3}d\xi dx\\=&
\frac{\varepsilon^2}{8\pi}
\int_{\R^3}\int_{\R^3}
U_{a_j}^2(x-\frac{x^{(1)}_{j,\varepsilon}-x^{(1)}_{l,\varepsilon}}{\varepsilon})
U_{a_l}(\xi)\eta_{l,\varepsilon}(\xi)
\frac{x^i-\xi^i}{|x-\xi|^3}d\xi dx+O\big(e^{-\eta/\varepsilon}\big)\\=&
\frac{\varepsilon^2}{8\pi}
\int_{\R^3}\int_{\R^3}
U_{a_j}^2(x-\frac{x^{(1)}_{j,\varepsilon}-x^{(1)}_{l,\varepsilon}}{\varepsilon})
U_{a_l}(\xi)\big(\sum^3_{m=1}d_{m,l} \frac{\partial U_{a_l}(\xi)}{\partial \xi_m}\big)
\frac{x^i-\xi^i}{|x-\xi|^3}d\xi dx\\&
+o\big(\varepsilon^2\int_{\R^3}\int_{\R^3}
U_{a_j}^2(x-\frac{x^{(1)}_{j,\varepsilon}-x^{(1)}_{l,\varepsilon}}{\varepsilon})
U_{a_l}(\xi)
\frac{x^i-\xi^i}{|x-\xi|^3}d\xi dx\big)+O\big(e^{-\eta/\varepsilon}\big)\\=&
\frac{\varepsilon^2}{16\pi}
\int_{\R^3}\int_{\R^3}
U_{a_j}^2(x-\frac{x^{(1)}_{j,\varepsilon}-x^{(1)}_{l,\varepsilon}}{\varepsilon})
U^2_{a_l}(\xi)\big(\sum^3_{m=1}d_{m,l} \frac{\partial \frac{x^i-\xi^i}{|x-\xi|^3}}{\partial \xi_m}\big)
d\xi dx
+o\big(\varepsilon^4\big),
\end{split}
\end{flalign}
and
\begin{flalign}\label{af5}
\begin{split}
\int_{\R^3}\int_{\R^3}&
U_{a_j}^2(x-\frac{x^{(1)}_{j,\varepsilon}-x^{(1)}_{l,\varepsilon}}{\varepsilon})
U^2_{a_l}(\xi)\big(\sum^3_{m=1}d_{m,l} \frac{\partial \frac{x^i-\xi^i}{|x-\xi|^3}}{\partial \xi_m}\big)
d\xi dx \\=&
O\big(
\int_{\R^3}\int_{\R^3}
U_{a_j}^2(x-\frac{x^{(1)}_{j,\varepsilon}-x^{(1)}_{l,\varepsilon}}{\varepsilon})
U^2_{a_l}(\xi)|x-\xi|^{-3}
d\xi dx\big) = O\big(\varepsilon^3\big),
\end{split}
\end{flalign}
here we also use the following estimate, which can be found by \eqref{aa9},
\begin{equation*}
\int_{\R^3}\int_{\R^3}
U_{a_j}^2(x-\frac{x^{(1)}_{j,\varepsilon}-x^{(1)}_{l,\varepsilon}}{\varepsilon})
U^2_{a_l}(\xi)|x-\xi|^{-\alpha}
d\xi dx=O\big(\varepsilon^\alpha\big),~\mbox{for}~\alpha>0,~\mbox{and}~l\neq j.
\end{equation*}
Similar to \eqref{f5} and \eqref{af5}, we have
\begin{flalign}\label{f6}
\begin{split}
\frac{1}{8\pi \varepsilon^2}
\int_{ B_{\delta}(x_{j,\varepsilon}^{(1)})}&\int_{\R^3}
U_{a_j}^2(\frac{x-x^{(1)}_{j,\varepsilon}}{\varepsilon})
U_{a_l}(\frac{\xi-x^{(2)}_{l,\varepsilon}}{\varepsilon})\eta_{\varepsilon}(\xi)
\frac{x^i-\xi^i}{|x-\xi|^3}d\xi dx=o\big(\varepsilon^4\big).
\end{split}
\end{flalign}
Then \eqref{f5}, \eqref{af5} and \eqref{f6} imply
\begin{flalign}\label{f7}
\begin{split}
F_{1,1,3}=o\big(\varepsilon^4\big).
\end{split}
\end{flalign}
Then \eqref{f1} can be deduced by \eqref{f2}, \eqref{f3}, \eqref{f4} and \eqref{f7}.
\end{proof}
\begin{Lem}\label{ld2}
It holds
\begin{flalign}\label{f8}
\begin{split}
F_{1,2}&
=\frac{1}{2\pi \varepsilon^2}
\int_{\R^3}\int_{\R^3}
U_{a_j}(\frac{x-x^{(1)}_{j,\varepsilon}}{\varepsilon})w_\varepsilon^{(1)}(x)
U_{a_j}(\frac{\xi-x^{(1)}_{j,\varepsilon}}{\varepsilon})
\big)\eta_{\varepsilon}(\xi)
\frac{x^i-\xi^i}{|x-\xi|^3}d\xi dx
+o\big(\varepsilon^4\big).
\end{split}
\end{flalign}
\end{Lem}

\begin{proof}
First, we write $F_{1,2}$ as follows:
\begin{flalign}\label{f9}
F_{1,2}=F_{1,2,1}+F_{1,2,2}+F_{1,2,3},
\end{flalign}
where
\begin{flalign*}
&F_{1,2,1}=\frac{1}{4\pi \varepsilon^2}
\int_{ B_{\delta}(x_{j,\varepsilon}^{(1)})}\int_{\R^3}
U_{a_j}(\frac{x-x^{(1)}_{j,\varepsilon}}{\varepsilon})w_\varepsilon^{(1)}(x)
\big(U_{a_j}(\frac{\xi-x^{(1)}_{j,\varepsilon}}{\varepsilon})+
U_{a_j}(\frac{\xi-x^{(2)}_{j,\varepsilon}}{\varepsilon})
\big)\eta_{\varepsilon}(\xi)
\frac{x^i-\xi^i}{|x-\xi|^3}d\xi dx,&
\end{flalign*}
\begin{flalign*}
&F_{1,2,2}=\frac{1}{4\pi \varepsilon^2}
\int_{ B_{\delta}(x_{j,\varepsilon}^{(1)})}\int_{\R^3}
U_{a_j}(\frac{x-x^{(1)}_{j,\varepsilon}}{\varepsilon})w_\varepsilon^{(1)}(x)\big(
W^{(1)}_{j,\varepsilon}(\xi)+W^{(2)}_{j,\varepsilon}(\xi)\big)\eta_{\varepsilon}(\xi)
\frac{x^i-\xi^i}{|x-\xi|^3}d\xi dx,&
\end{flalign*}
\begin{flalign*}
&F_{1,2,3}=\frac{1}{4\pi \varepsilon^2}
\int_{ B_{\delta}(x_{j,\varepsilon}^{(1)})}\int_{\R^3}
U_{a_j}(\frac{x-x^{(1)}_{j,\varepsilon}}{\varepsilon})w_\varepsilon^{(1)}(x)
\big(w^{(1)}_{\varepsilon}(\xi)+w^{(2)}_{\varepsilon}(\xi)\big)\eta_{\varepsilon}(\xi)
\frac{x^i-\xi^i}{|x-\xi|^3}d\xi dx.&
\end{flalign*}
Now by direct calculation, we get
\begin{flalign}\label{lg1}
\begin{split}
U_{a_j}(\frac{x-x^{(1)}_{j,\varepsilon}}{\varepsilon})
-U_{a_j}(\frac{x-x^{(2)}_{j,\varepsilon}}{\varepsilon})
=&O\big(|\frac{x^{(1)}_{j,\varepsilon}-x^{(2)}_{j,\varepsilon}}{\varepsilon}|\big)\cdot|\nabla U_{a_j}(\frac{x-x^{(1)}_{j,\varepsilon}}{\varepsilon}x)|\\
=&O\big(\frac{|x^{(1)}_{j,\varepsilon}-x^{(2)}_{j,\varepsilon}|}{\varepsilon}\big)U_{a_j}
(\frac{x-x^{(1)}_{j,\varepsilon}}{\varepsilon}).
\end{split}
\end{flalign}
Then by \eqref{A.1}, \eqref{lg1}, we have
\begin{flalign}\label{f10}
\begin{split}
F_{1,2,1}=\frac{1}{2\pi \varepsilon^2}
\int_{\R^3}\int_{\R^3}
U_{a_j}(\frac{x-x^{(1)}_{j,\varepsilon}}{\varepsilon})w_\varepsilon^{(1)}(x)
U_{a_j}(\frac{\xi-x^{(1)}_{j,\varepsilon}}{\varepsilon})
\big)\eta_{\varepsilon}(\xi)
\frac{x^i-\xi^i}{|x-\xi|^3}d\xi dx+o\big(\varepsilon^{4}\big).
\end{split}
\end{flalign}
Also, similar to \eqref{cll}, we get
\begin{flalign}\label{f11}
\begin{split}
F_{1,2,2}=&O\big(\varepsilon^{-2} \|
U_{a_j}(\frac{\cdot-x^{(1)}_{j,\varepsilon}}{\varepsilon})\|_{\varepsilon}\cdot\|w_\varepsilon^{(1)}
\|_{\varepsilon}\cdot
\|W^{(1)}_{j,\varepsilon}(\cdot)+W^{(2)}_{j,\varepsilon}(\cdot)\|_{\varepsilon}\cdot
\|\eta_{\varepsilon}\|_{\varepsilon} \big)
+O\big(e^{-\eta/\varepsilon}\big)=
O\big(\varepsilon^6\big).
\end{split}
\end{flalign}
And by \eqref{aa5} and \eqref{2.45}, we obtain
\begin{flalign}\label{f12}
\begin{split}
F_{1,2,3}&=O\big(\varepsilon^{-4}\|U_{a_j}(\frac{\cdot -x^{(1)}_{j,\varepsilon}}{\varepsilon})\|_{\varepsilon}\|w^{(1)}_{\varepsilon}\|_{\varepsilon}
\cdot\|w^{(1)}_{\varepsilon}+w^{(2)}_{\varepsilon}\|_{\varepsilon}\cdot
\|\eta_\varepsilon\|_{\varepsilon}\big)=O\big(\varepsilon^6\big).
\end{split}
\end{flalign}
Then \eqref{f8} can be deduced by \eqref{f9}, \eqref{f10}, \eqref{f11} and \eqref{f12}.
\end{proof}
\begin{Lem}\label{ld3}
It holds
\begin{flalign}\label{f13}
\begin{split}
F_{2,1}&=G_2-\frac{1}{2\pi \varepsilon^2}
\int_{\R^3}\int_{\R^3}
U_{a_j}(\frac{x-x^{(1)}_{j,\varepsilon}}{\varepsilon})w_\varepsilon^{(1)}(x)
U_{a_j}(\frac{\xi-x^{(1)}_{j,\varepsilon}}{\varepsilon})
\big)\eta_{\varepsilon}(\xi)
\frac{x^i-\xi^i}{|x-\xi|^3}d\xi dx
+o\big(\varepsilon^4\big),
\end{split}
\end{flalign}
where
\begin{flalign*}
G_2=-\frac{1}{8\pi \varepsilon^2}
\int_{\R^3}\int_{\R^3}U^2_{a_j}(\frac{x-x^{(2)}_{j,\varepsilon}}{\varepsilon})
\big(
U_{a_j}(\frac{\xi-x^{(1)}_{j,\varepsilon}}{\varepsilon})+
U_{a_j}(\frac{\xi-x^{(2)}_{j,\varepsilon}}{\varepsilon}) \big)
\eta_{\varepsilon}(x)\frac{x^i-\xi^i}{|x-\xi|^3}d\xi dx.
\end{flalign*}
\end{Lem}
\begin{proof}
First, we write $F_{2,1}$ as follows:
\begin{flalign}\label{f14}
F_{2,1}=F_{2,1,1}+F_{2,1,2}+F_{2,1,3}+F_{2,1,4}+F_{2,1,5}+F_{2,1,6},
\end{flalign}
where
\begin{flalign*}
&
F_{2,1,1}=\frac{1}{8\pi \varepsilon^2}
\int_{B_{\delta}(x_{j,\varepsilon}^{(1)})}\int_{\R^3}\big(
U_{a_j}(\frac{x-x^{(1)}_{j,\varepsilon}}{\varepsilon})+
U_{a_j}(\frac{x-x^{(2)}_{j,\varepsilon}}{\varepsilon}) \big)
\eta_{\varepsilon}(x)U^2_{a_j}(\frac{\xi-x^{(2)}_{j,\varepsilon}}{\varepsilon})\frac{x^i-\xi^i}{|x-\xi|^3}d\xi dx,&
\end{flalign*}
\begin{flalign*}
&
F_{2,1,2}=\frac{1}{4\pi \varepsilon^2}
\int_{B_{\delta}(x_{j,\varepsilon}^{(1)})}\int_{\R^3}\big(
U_{a_j}(\frac{x-x^{(1)}_{j,\varepsilon}}{\varepsilon})+
U_{a_j}(\frac{x-x^{(2)}_{j,\varepsilon}}{\varepsilon}) \big)
\eta_{\varepsilon}(x)U_{a_j}(\frac{\xi-x^{(2)}_{j,\varepsilon}}{\varepsilon})
w^{(2)}_{\varepsilon}(\xi)\frac{x^i-\xi^i}{|x-\xi|^3}d\xi dx,
&
\end{flalign*}
\begin{flalign*}
&
F_{2,1,3}=\frac{1}{8\pi \varepsilon^2}
\int_{B_{\delta}(x_{j,\varepsilon}^{(1)})}\int_{\R^3}\big(
U_{a_j}(\frac{x-x^{(1)}_{j,\varepsilon}}{\varepsilon})+
U_{a_j}(\frac{x-x^{(2)}_{j,\varepsilon}}{\varepsilon}) \big)
\eta_{\varepsilon}(x)\big(w^{(2)}_{\varepsilon}(\xi)\big)^2\frac{x^i-\xi^i}{|x-\xi|^3}d\xi dx,
&
\end{flalign*}

\begin{flalign*}
&
F_{2,1,4}=\frac{1}{4\pi \varepsilon^2}
\int_{B_{\delta}(x_{j,\varepsilon}^{(1)})}\int_{\R^3}
\big(
U_{a_j}(\frac{x-x^{(1)}_{j,\varepsilon}}{\varepsilon})+
U_{a_j}(\frac{x-x^{(2)}_{j,\varepsilon}}{\varepsilon}) \big)\eta_{\varepsilon}(x)W^{(2)}_{j,\varepsilon}(\xi)
 U_{a_j}(\frac{\xi-x^{(2)}_{j,\varepsilon}}{\varepsilon})\frac{x^i-\xi^i}{|x-\xi|^3}d\xi dx,
&
\end{flalign*}
\begin{flalign*}
&
F_{2,1,5}=\frac{1}{4\pi \varepsilon^2}
\int_{B_{\delta}(x_{j,\varepsilon}^{(1)})}\int_{\R^3}
\big(
U_{a_j}(\frac{x-x^{(1)}_{j,\varepsilon}}{\varepsilon})+
U_{a_j}(\frac{x-x^{(2)}_{j,\varepsilon}}{\varepsilon}) \big)\eta_{\varepsilon}(x)w^{(2)}_{\varepsilon}(\xi)
 U_{a_j}(\frac{\xi-x^{(2)}_{j,\varepsilon}}{\varepsilon}) \frac{x^i-\xi^i}{|x-\xi|^3}d\xi dx,
&
\end{flalign*}
\begin{flalign*}
&F_{2,1,6}=\frac{1}{8\pi \varepsilon^2}
\int_{B_{\delta}(x_{j,\varepsilon}^{(1)})}\int_{\R^3}\big(
U_{a_j}(\frac{x-x^{(1)}_{j,\varepsilon}}{\varepsilon})+
U_{a_j}(\frac{x-x^{(2)}_{j,\varepsilon}}{\varepsilon}) \big)\eta_{\varepsilon}(x)
\big(W^{(2)}_{j,\varepsilon}(\xi)\big)^2\frac{x^i-\xi^i}{|x-\xi|^3}d\xi dx.&
\end{flalign*}
Now, by \eqref{A.1} and symmetry, we have
\begin{flalign}\label{f15}
\begin{split}
F_{2,1,1}=
G_{2}+O\big(e^{-\eta/\varepsilon}\big).
\end{split}
\end{flalign}
Also similar to \eqref{f10}, we know
\begin{flalign}\label{f16}
\begin{split}
F_{2,1,2}
=&-\frac{1}{2\pi \varepsilon^2}
\int_{\R^3}\int_{\R^3}
U_{a_j}(\frac{x-x^{(1)}_{j,\varepsilon}}{\varepsilon})w_\varepsilon^{(2)}(x)
U_{a_j}(\frac{\xi-x^{(1)}_{j,\varepsilon}}{\varepsilon})
\big)\eta_{\varepsilon}(\xi)
\frac{x^i-\xi^i}{|x-\xi|^3}d\xi dx
+o\big(\varepsilon^4\big)\\=&
-\frac{1}{2\pi \varepsilon^2}
\int_{\R^3}\int_{\R^3}
U_{a_j}(\frac{x-x^{(1)}_{j,\varepsilon}}{\varepsilon})w_\varepsilon^{(1)}(x)
U_{a_j}(\frac{\xi-x^{(1)}_{j,\varepsilon}}{\varepsilon})
\big)\eta_{\varepsilon}(\xi)
\frac{x^i-\xi^i}{|x-\xi|^3}d\xi dx\\
&+O\big(\|U_{a_j}(\frac{\cdot-x^{(1)}_{j,\varepsilon}}{\varepsilon})\|^2_{\varepsilon}
\|w_\varepsilon^{(1)}-w_\varepsilon^{(2)}\|_\varepsilon
\|\eta_{\varepsilon}\|_\varepsilon\big)+o\big(\varepsilon^4\big)\\
=&-\frac{1}{2\pi \varepsilon^2}
\int_{\R^3}\int_{\R^3}
U_{a_j}(\frac{x-x^{(1)}_{j,\varepsilon}}{\varepsilon})w_\varepsilon^{(1)}(x)
U_{a_j}(\frac{\xi-x^{(1)}_{j,\varepsilon}}{\varepsilon})
\big)\eta_{\varepsilon}(\xi)
\frac{x^i-\xi^i}{|x-\xi|^3}d\xi dx
+o\big(\varepsilon^4\big).
\end{split}
\end{flalign}
Next, by \eqref{aa5} and \eqref{2.45}, we get
\begin{flalign}\label{f17}
\begin{split}
F_{2,1,3}&=O\big( \varepsilon^{-4} \|
U_{a_j}(\frac{\cdot-x^{(1)}_{j,\varepsilon}}{\varepsilon})+
U_{a_j}(\frac{\cdot-x^{(2)}_{j,\varepsilon}}{\varepsilon})\|_{\varepsilon}
\|\eta_{\varepsilon}\|_{\varepsilon}\|w^{(2)}_{\varepsilon}\|_{\varepsilon}^2\big)=
O\big(\varepsilon^6\big).
\end{split}
\end{flalign}
And by \eqref{ll3}, we obtain
\begin{flalign}\label{f18}
\begin{split}
F_{2,1,4}&=O\big(e^{-\eta/\varepsilon}\big)~\mbox{and}~
F_{2,1,5}=O\big(\varepsilon^6\big).
\end{split}
\end{flalign}
Also, $l\neq j$,  similar to \eqref{f5} and \eqref{af5}, we have
\begin{flalign}\label{f22}
\begin{split}
\frac{1}{8\pi \varepsilon^2}&
\int_{B_{\delta}(x_{j,\varepsilon}^{(1)})}\int_{\R^3}
U_{a_j}(\frac{x-x^{(1)}_{j,\varepsilon}}{\varepsilon})\eta_{\varepsilon}(x)
U^2_{a_l}(\frac{\xi-x^{(2)}_{l,\varepsilon}}{\varepsilon})\frac{x^i-\xi^i}{|x-\xi|^3}d\xi dx\\=&
\frac{\varepsilon^2}{8\pi}
\int_{\R^3}\int_{\R^3}
U_{a_j}(x)\eta_{j,\varepsilon}(x)
U^2_{a_l}(\xi+\frac{x^{(1)}_{j,\varepsilon}-x^{(2)}_{l,\varepsilon}}{\varepsilon})
\frac{x^i-\xi^i}{|x-\xi|^3}d\xi dx+O\big(e^{-\eta/\varepsilon}\big),
\end{split}
\end{flalign}
and
\begin{equation}\label{f221}
\begin{split}
\int_{\R^3}\int_{\R^3}&
U_{a_j}(x)\eta_{j,\varepsilon}(x)
U^2_{a_l}(\xi+\frac{x^{(1)}_{j,\varepsilon}-x^{(2)}_{l,\varepsilon}}{\varepsilon})
\frac{x^i-\xi^i}{|x-\xi|^3}d\xi dx\\=&
\int_{\R^3}\int_{\R^3}U_{a_j}(x)\big(\sum^3_{m=1}d_{m,j} \frac{\partial U_{a_j}(x)}{\partial x_m}\big)
U^2_{a_l}(\xi+\frac{x^{(1)}_{j,\varepsilon}-x^{(2)}_{l,\varepsilon}}{\varepsilon})
\frac{x^i-\xi^i}{|x-\xi|^3}d\xi dx \\&+
 o\big( \int_{\R^3}\int_{\R^3}
U_{a_j}(x)
U^2_{a_l}(\xi+\frac{x^{(1)}_{j,\varepsilon}-x^{(2)}_{l,\varepsilon}}{\varepsilon})
\frac{1}{|x-\xi|^2}d\xi dx\big)\\
=&
O\big(
\int_{\R^3}\int_{\R^3}
U_{a_j}^2(x)U^2_{a_l}(\xi+\frac{x^{(1)}_{j,\varepsilon}-x^{(2)}_{l,\varepsilon}}{\varepsilon})
|x-\xi|^{-3}d\xi dx\big)+o\big(\varepsilon^2\big)\\=&
O\big(\varepsilon^3\big)+o\big(\varepsilon^2\big)=o\big(\varepsilon^2\big).
\end{split}
\end{equation}
Similar to \eqref{f22} and  \eqref{f221}, we obtain
\begin{flalign}\label{f23}
\begin{split}
\frac{1}{8\pi \varepsilon^2}
\int_{B_{\delta}(x_{j,\varepsilon}^{(1)})}&\int_{\R^3}
U_{a_j}(\frac{x-x^{(2)}_{j,\varepsilon}}{\varepsilon})\eta_{\varepsilon}(x)
U^2_{a_l}(\frac{\xi-x^{(2)}_{l,\varepsilon}}{\varepsilon})\frac{x^i-\xi^i}{|x-\xi|^3}d\xi dx =o\big(\varepsilon^4\big).
\end{split}
\end{flalign}
Then \eqref{f22}, \eqref{f221} and \eqref{f23} imply
\begin{flalign}\label{f24}
\begin{split}
F_{2,1,6}=&o\big(\varepsilon^4\big).
 \end{split}
\end{flalign}
Then \eqref{f14}, \eqref{f15}, \eqref{f16}, \eqref{f17}, \eqref{f18}  and \eqref{f24} imply \eqref{f13}.
\end{proof}
\begin{Lem}\label{ld4}
It holds
\begin{flalign}\label{f25}
\begin{split}
F_{2,3}=-\frac{1}{8\pi \varepsilon^2}
\int_{\R^3}\int_{\R^3}
U^2_{a_j}(\frac{x-x^{(1)}_{j,\varepsilon}}{\varepsilon})
\big(w^{(1)}_{\varepsilon}(\xi)+w^{(2)}_{\varepsilon}(\xi)\big)
\eta_{\varepsilon}(\xi)\frac{x^i-\xi^i}{|x-\xi|^3}d\xi dx+o\big(\varepsilon^{4}\big).
\end{split}
\end{flalign}
\end{Lem}
\begin{proof}
First, we write $F_{2,3}$ as follows:
\begin{flalign}\label{f26}
\begin{split}
F_{2,3}&=F_{2,3,1}+F_{2,3,2}+F_{2,3,3}+F_{2,3,4}+F_{2,3,5}+F_{2,3,6},
\end{split}
\end{flalign}
where
\begin{flalign*}
&
F_{2,3,1}=\frac{1}{8\pi \varepsilon^2}
\int_{B_{\delta}(x_{j,\varepsilon}^{(1)})}\int_{\R^3}
\big(w^{(1)}_{\varepsilon}(x)+w^{(2)}_{\varepsilon}(x)\big)
\eta_{\varepsilon}(x)U^2_{a_j}(\frac{\xi-x^{(2)}_{j,\varepsilon}}{\varepsilon})\frac{x^i-\xi^i}{|x-\xi|^3}d\xi dx,&
\end{flalign*}
\begin{flalign*}
&
F_{2,3,2}=\frac{1}{4\pi \varepsilon^2}
\int_{B_{\delta}(x_{j,\varepsilon}^{(1)})}\int_{\R^3}
\big(w^{(1)}_{\varepsilon}(x)+w^{(2)}_{\varepsilon}(x)\big)
\eta_{\varepsilon}(x)U_{a_j}(\frac{\xi-x^{(2)}_{j,\varepsilon}}{\varepsilon})
w^{(2)}_{\varepsilon}(\xi)\frac{x^i-\xi^i}{|x-\xi|^3}d\xi dx,
&
\end{flalign*}
\begin{flalign*}
&
F_{2,3,3}=\frac{1}{8\pi \varepsilon^2}
\int_{B_{\delta}(x_{j,\varepsilon}^{(1)})}\int_{\R^3}
\big(w^{(1)}_{\varepsilon}(x)+w^{(2)}_{\varepsilon}(x)\big)
\eta_{\varepsilon}(x)\big(w^{(2)}_{\varepsilon}(\xi)\big)^2\frac{x^i-\xi^i}{|x-\xi|^3}d\xi dx,&
\end{flalign*}
\begin{flalign*}
&F_{2,3,4}=\frac{1}{4\pi \varepsilon^2}
\int_{B_{\delta}(x_{j,\varepsilon}^{(1)})}\int_{\R^3}
\big(w^{(1)}_{\varepsilon}(x)+w^{(2)}_{\varepsilon}(x)\big)\eta_{\varepsilon}(x)
W^{(2)}_{j,\varepsilon}(\xi)
 U_{a_j}(\frac{\xi-x^{(2)}_{j,\varepsilon}}{\varepsilon})\frac{x^i-\xi^i}{|x-\xi|^3}d\xi dx,&
\end{flalign*}

\begin{flalign*}
&F_{2,3,5}=\frac{1}{4\pi \varepsilon^2}
\int_{B_{\delta}(x_{j,\varepsilon}^{(1)})}\int_{\R^3}
\big(w^{(1)}_{\varepsilon}(x)+w^{(2)}_{\varepsilon}(x)\big)\eta_{\varepsilon}(x)
w^{(2)}_{\varepsilon}(\xi)
 U_{a_j}(\frac{\xi-x^{(2)}_{j,\varepsilon}}{\varepsilon})\frac{x^i-\xi^i}{|x-\xi|^3}d\xi dx,&
\end{flalign*}
\begin{flalign*}
&
F_{2,3,6}=\frac{1}{8\pi \varepsilon^2}
\int_{B_{\delta}(x_{j,\varepsilon}^{(1)})}\int_{\R^3}
\big(w^{(1)}_{\varepsilon}(x)+w^{(2)}_{\varepsilon}(x)\big)\eta_{\varepsilon}(x)
\big(W^{(2)}_{j,\varepsilon}(\xi)\big)^2\frac{x^i-\xi^i}{|x-\xi|^3}d\xi dx.
&
\end{flalign*}
Now by \eqref{A.1}, \eqref{lg1} and symmetry, we have
\begin{flalign}\label{f27}
\begin{split}
F_{2,3,1} =-\frac{1}{8\pi \varepsilon^2}
\int_{\R^3}\int_{\R^3}
U^2_{a_j}(\frac{x-x^{(1)}_{j,\varepsilon}}{\varepsilon})
\big(w^{(1)}_{\varepsilon}(\xi)+w^{(2)}_{\varepsilon}(\xi)\big)
\eta_{\varepsilon}(\xi)\frac{x^i-\xi^i}{|x-\xi|^3}d\xi dx+o\big(\varepsilon^{4}\big).
\end{split}
\end{flalign}
Next, by \eqref{aa5} and \eqref{2.45}, we get
\begin{flalign}\label{f28}
\begin{split}
F_{2,3,2}&=O\big( \varepsilon^{-4}\|w^{(1)}_{\varepsilon}+w^{(2)}_{\varepsilon}\|_{\varepsilon}  \|
U_{a_j}(\frac{\cdot-x^{(2)}_{j,\varepsilon}}{\varepsilon})\|_{\varepsilon}
\|\eta_{\varepsilon}\|_{\varepsilon}\|w^{(2)}_{\varepsilon}\|_{\varepsilon}\big)=
O\big(\varepsilon^6\big),
\end{split}
\end{flalign}
and
\begin{flalign}\label{f29}
\begin{split}
F_{2,3,3}&=O\big( \varepsilon^{-4} \|w^{(1)}_{\varepsilon}+w^{(2)}_{\varepsilon}\|_{\varepsilon}
\|\eta_{\varepsilon}\|_{\varepsilon}\|w^{(2)}_{\varepsilon}\|^2_{\varepsilon}\big)=
O\big(\varepsilon^8\big).
\end{split}
\end{flalign}
Also by \eqref{ll3}, we know
\begin{flalign}\label{f30}
\begin{split}
F_{2,3,4}&=O\big(e^{-\eta/\varepsilon}\big)~\mbox{and}~
F_{2,3,5}=O\big(\varepsilon^6\big).
\end{split}
\end{flalign}
Next, similar to \eqref{cll},  we obtain
\begin{flalign}\label{f32}
\begin{split}
F_{2,3,6}=O\big(\varepsilon^{-2}\|w^{(1)}_{\varepsilon}+w^{(2)}_{\varepsilon}\|_{\varepsilon}
\|W^{(2)}_{j,\varepsilon}\|_{\varepsilon}
\|\eta_{\varepsilon}\|_{\varepsilon} \big)+O\big(e^{-\eta/\varepsilon}\big)=O\big(\varepsilon^6 \big).
\end{split}
\end{flalign}
Then \eqref{f26}, \eqref{f27}, \eqref{f28}, \eqref{f29}, \eqref{f30} and \eqref{f32} imply \eqref{f25}.
\end{proof}

\noindent{\bf Acknowledgements.}

This work was partially supported by NSFC(No.11571130; No.11671162), self-determined research funds of CCNU from colleges' basic
research and operation of MOE(CCNU16A05011, CCNU17QN0008).

\renewcommand\refname{References}
\renewenvironment{thebibliography}[1]{%
\section*{\refname}
\list{{\arabic{enumi}}}{\def\makelabel##1{\hss{##1}}\topsep=0mm
\parsep=0mm
\partopsep=0mm\itemsep=0mm
\labelsep=1ex\itemindent=0mm
\settowidth\labelwidth{\small[#1]}%
\leftmargin\labelwidth \advance\leftmargin\labelsep
\advance\leftmargin -\itemindent
\usecounter{enumi}}\small
\def\newblock{\ }
\sloppy\clubpenalty4000\widowpenalty4000
\sfcode`\.=1000\relax}{\endlist}
\bibliographystyle{model1b-num-names}

\end{document}